\documentclass[final,3p,times]{elsarticle}

\biboptions{sort&compress}

\usepackage{amssymb}
\usepackage{amsmath}
\usepackage{amsthm}
\usepackage{mathrsfs}
\usepackage{graphicx}                         
\usepackage{hyperref}
\usepackage[acronym]{glossaries}
\usepackage{booktabs}         
\usepackage{nicefrac}         
\usepackage{xcolor}           
\usepackage{algorithm}        
\usepackage{algpseudocode}    
\usepackage{multirow}         
\usepackage{scalerel}         
\usepackage{enumitem}         
\usepackage{cleveref}
\setkeys{glslink}{hyper=false}
\crefname{appendix}{}{}
\usepackage{custom}           
\Crefname{appendix}{Appendix}{Appendices} 
\crefname{appendix}{appendix}{appendices} 

\journal{Journal of Computational Physics}

\begin{document}

\begin{frontmatter}


\title{Streaming Operator Inference for Model Reduction of \\ Large-Scale Dynamical Systems\tnoteref{funding}}
\tnotetext[funding]{The authors were supported by the Department of Energy Office of Science Advanced Scientific Computing Research, DOE Award DE-SC0024721.}


\author[gatech-ae]{Tomoki Koike\corref{cor}} 
\ead{tkoike@gatech.edu}
\cortext[cor]{Corresponding author}
\author[nlr]{Prakash Mohan}
\author[nlr]{Mark T. Henry de Frahan}
\author[nlr]{Julie Bessac}
\author[gatech-ae,gatech-cse]{Elizabeth Qian}

\affiliation[gatech-ae]{
    organization={
        School of Aerospace Engineering, Georgia Institute of Technology
    }, 
    city={Atlanta},
    postcode={30332}, 
    state={GA},
    country={USA}
}
\affiliation[gatech-cse]{
    organization={
        School of Computational Science and Engineering, Georgia Institute of Technology
    }, 
    city={Atlanta},
    postcode={30332}, 
    state={GA},
    country={USA}
}
\affiliation[nlr]{
    organization={
        Computational Science Center, National Laboratory of the Rockies (NLR)
    }, 
    addressline={15013 Denver West Parkway}, 
    city={Golden},
    postcode={80401}, 
    state={CO},
    country={USA}
}

\begin{abstract}
    Projection-based model reduction enables efficient simulation of complex dynamical systems by constructing low-dimensional surrogate models from high-dimensional data. The Operator Inference (OpInf) approach learns such reduced surrogate models through a two-step process: constructing a low-dimensional basis via Singular Value Decomposition (SVD) to compress the data, then solving a linear least-squares (LS) problem to infer reduced operators that govern the dynamics in this compressed space, all without access to the underlying code or full model operators, i.e., non-intrusively.  Traditional OpInf operates as a batch learning method, where both the SVD and LS steps process all data simultaneously. This poses a barrier to deployment of the approach on large-scale applications where data set sizes prevent the loading of all data into memory simultaneously. Additionally, the traditional batch approach does not naturally allow model updates using new data acquired during online computation.  To address these limitations, we propose Streaming OpInf, which learns reduced models from sequentially arriving data streams.  Our approach employs incremental SVD for adaptive basis construction and recursive LS for streaming operator updates, eliminating the need to store complete data sets while enabling online model adaptation. The approach can flexibly combine different choices of streaming algorithms for numerical linear algebra: we systematically explore the impact of these choices both analytically and numerically to identify effective combinations for accurate reduced model learning. Numerical experiments on benchmark problems and a large-scale turbulent channel flow demonstrate that Streaming OpInf achieves accuracy comparable to batch OpInf while reducing memory requirements by over 99\% and enabling dimension reductions exceeding 31,000x, resulting in orders-of-magnitude faster predictions. Our results establish Streaming OpInf as a scalable framework for reduced operator learning in large-scale and online streaming settings.
\end{abstract}



\begin{keyword}
model reduction \sep operator inference \sep data streaming \sep incremental SVD \sep recursive least-squares
\end{keyword}

\end{frontmatter}


\section{Introduction}\label{sec:intro}

The ability to use data to learn accurate and efficient reduced models for complex dynamical systems is crucial across scientific and engineering applications, as these models facilitate rapid simulations for many-query computations such as uncertainty quantification, inverse problems, optimization, and control. \Gls{opinf}~\cite{peherstorferW2016,kramer2024Learning} refers to a family of methods that learn such efficient reduced models from data by fitting reduced model operators to minimize the residual of the reduced system equations. Assuming a polynomial reduced model, \gls{opinf} learns the coefficient matrices representing polynomial nonlinearities. \Gls{opinf} computes such reduced operators following two key steps: (i) dimension reduction by projecting data onto a reduced basis constructed via \gls{pod}, which computes an optimal low-dimensional approximation via \gls{svd}, and (ii) solving a linear \gls{ls} problem to compute the reduced operators by fitting them to all projected data at once. The first step emanates from classical projection-based model reduction, which projects the full-dimensional model onto a low-dimensional subspace to yield accurate yet computationally efficient reduced models~\cite{benner2015survey}, while the second step leverages data-driven learning to infer reduced operators without intrusive access to the underlying code or full model operators, and thereby making \gls{opinf} a \emph{non-intrusive} method.

The two-step \emph{reduce-then-learn} formulation of \gls{opinf} is highly amenable to extensions.
Some researchers modify the projection step by replacing the linear \gls{pod} basis with nonlinear manifolds~\cite{geelen2023Operator,geelen2024Learning,schwerdtner2025operator} or adopting oblique projections~\cite{padovan2024DataDriven}. Others alter the learning step through regularized or constrained formulations for robustness~\cite{mcquarrie2021Datadriven,sawant2023Physicsinformed,boef2024Stable}, tensor or nested operator training for scalability~\cite{danis2025TensorTrain,aretz2025Nested}, Bayesian or Gaussian process regression for uncertainty quantification~\cite{guo2022Bayesiana,mcquarrie2025Bayesian}, or physics-informed constraints on coefficients for stability and structure preservation~\cite{sharma2022Hamiltonian,sharma2024lagrangian,koike2024Energy,gruber2025Variationally,goyal2025guaranteed,shuai2025Symmetryreduced,zastrow2025block}. The work~\cite{prakash2025Nonintrusive} reverses the standard workflow by first learning full-dimensional operators and then projecting them onto a reduced basis. Complementary strategies include data reprojection for noise attenuation~\cite{peherstorfer2020Sampling,uy2023Active}, variable lifting to extend \gls{opinf} to non-polynomial systems~\cite{qian2020Lift}, and exact reconstruction of intrusive operators~\cite{rosenberger2025exact}. These extensions have enabled vast applications across fluid dynamics~\cite{mcquarrie2021Datadriven,qian2022Reduced}, mechanical systems~\cite{filanova2023Operatora,gruber2023canonical,sharma2024preserving,geng2024Gradient}, chemical kinetics~\cite{kim2025physically}, ice sheet modeling~\cite{aretz2025Nested,rosenberger2025exact}, plasma turbulence~\cite{gahr2024Scientific}, aeroelastic flutter~\cite{zastrow2023data}, and solar wind dynamics~\cite{issan2023Predicting}. 

However, \gls{opinf} and other existing reduced model learning methods~\cite{lucia2004reduced,mignolet2013review,guo2019data} are typically \emph{batch} learning methods which rely on collecting ample data over time, storing them in their entirety, and then learning the reduced models offline by processing all data at once. However, many real-world applications---such as climate modeling~\cite{randall2007climate,schar2020kilometer,hersbach2020era5}, fluid dynamics~\cite{willcox2007model,georgiadis2010large,nielsen2024large}, and astrophysics~\cite{schmidt2015large,burkhart2020Catalogue,angulo2022large}---present two critical challenges. First, high-resolution simulations generate \emph{large-scale} datasets requiring terabytes to petabytes of storage~\cite{ohana2024well,muthukrishnan2005data}, making it infeasible to store all data in memory or on disk. Second, some large-scale applications such as digital twins~\cite{torzoni2024digital,henneking2025real}, system monitoring~\cite{tomsovic2005designing}, and anomaly detection~\cite{lan2009toward} often demand online model learning for \emph{real-time} decision-making and control, which limits the possibility of waiting for complete data collection before training~\cite{muthukrishnan2005data}. Hence, in this work, we propose a new reduced model learning paradigm based on \emph{streaming} data, where state snapshots and possibly their time derivatives arrive sequentially with only a small subset accessible at any given moment, thereby achieving significant memory savings while enabling real-time model updates and predictions.

To address memory and storage challenges, parallel/distributed computing~\cite{farcas2025Distributed} and domain decomposition methods~\cite{geelen2022Localized,farcas2024Domain,moore2024Domain,gkimisis2025non} for \gls{opinf} have been proposed. The parallel approach distributes data and basis computation across multiple processors, aggregates results to obtain reduced data, and then computes the reduced operators. Domain decomposition methods exploit spatial locality by dividing the domain into subdomains, learning reduced models within each, and coupling them to form a global model. Similar strategies have been applied to other model reduction techniques~\cite{benner2000balanced,antil2010domain,knezevic2011high,ares2025parallel,perotto2010hierarchical,hoang2021domain,diaz2024fast}.

While these approaches can make learning reduced models tractable for some large-scale problems by partitioning the data set along rows, they remain batch methods in the sense that each partition requires simultaneous access to all time snapshots. In contrast, exploitation of temporal partitioning, where data can be processed incrementally column-by-column as snapshots arrive, remains largely unexplored. This motivates our development of \emph{Streaming \gls{opinf}}, which reformulates the two main components of \gls{opinf} for streaming data by using streaming \gls{svd} formulations for data reduction and using a \gls{rls} formulation to solve the reduced operator learning problem. In summary, this work makes the following contributions:
\begin{enumerate}[leftmargin=*,itemsep=0em]
    \item We propose a streaming approach to \gls{opinf} that learns reduced models incrementally as data arrives, enabling scalability to large-scale datasets and laying the groundwork for real-time model updates and predictions. 
    \item We implement and compare several state-of-the-art streaming algorithms from numerical linear algebra for reduced basis construction and operator learning, including \gls{isvd} methods~\cite{baker2012Lowrank}, randomized streaming \gls{svd}~\cite{tropp2019Streaming}, and \gls{rls} variants~\cite{sayed2011Adaptive,haykin2014Adaptive}, evaluating their accuracy, memory requirements, and computational efficiency.
    \item We demonstrate that Streaming \gls{opinf} achieves significant memory savings compared to batch \gls{opinf} while learning accurate reduced models from streaming data on benchmark problems and a large-scale turbulent three-dimensional channel flow simulation with nearly 10 million degrees of freedom.
\end{enumerate}

Batch processing challenges have motivated streaming approaches for other operator learning and modal analysis methods such as \gls{dmd}, spectral \gls{pod}, and their variants, which identify a system's natural frequencies, damping factors, and mode shapes~\cite{schmid2010Dynamic,aarontowne2018Spectral}. Streaming versions of spectral \gls{pod}~\cite{schmidt2019Efficient} and quadratic manifold construction~\cite{schwerdtner2025Online} employ \gls{isvd} for incremental basis construction but focus solely on extracting reduced manifolds without addressing operator learning. Streaming \gls{dmd}~\cite{hemati2014Dynamic,matsumoto2017Onthefly,liew2022Streaming} uses \gls{rls} or rank-1 updates to incrementally update decompositions, with extensions including parallel~\cite{anantharamu2019Parallel} and adaptive~\cite{zhang2019Online,alfatlawi2020Incremental} settings. While \gls{dmd} enables both analysis and prediction, it differs fundamentally from \gls{opinf} in its treatment of nonlinear systems. That is, \gls{dmd} identifies linear dynamics either directly or through the Koopman framework~\cite{rowley2009spectral,mezic2013analysis}, where nonlinear dynamics are embedded into a higher-dimensional space with linear evolution. This lifting introduces practical challenges including approximation errors from finite-dimensional truncation of the infinite-dimensional Koopman operator and potentially large lifted dimensions that negate computational benefits. This challenge of learning reduced operators in the lifted space arises in another similar work of online robust Koopman operator learning~\cite{sinha2023Online}.
In contrast, \gls{opinf} directly learns polynomial operators in the reduced space without such lifting.

Another related work, Streaming weak-SINDy~\cite{russo2025Streaming}, learns weak-form \gls{pde} solutions. The method accumulates feature and target matrices via streaming integration with streaming \gls{pod}, deferring regression to an offline stage where sequential thresholding \gls{ls} requires solving the batch \gls{ls} problem repeatedly until sparsity is achieved, whereas we consider online updating of the \gls{ls} solution via \gls{rls} along with different combinations of algorithms for streaming reduced operator learning. All existing streaming methods introduced so far show growing recognition that incremental learning is essential for scalability and real-time adaptation. However, we emphasize that no existing approach combines dimension reduction via incremental basis updates with direct learning of reduced polynomial operators from streaming data, which is the contribution of our Streaming \gls{opinf} framework.

The remainder of this paper is organized as follows. Section~\ref{sec:bg} reviews \gls{opinf} and the key ingredients of Streaming \gls{opinf}: streaming \gls{svd} and \gls{rls}. Section~\ref{sec:stream} presents the Streaming \gls{opinf} formulation with guidance on algorithm selection for incremental basis construction and operator learning. Section~\ref{sec:experiments} demonstrates effectiveness on benchmark problems and a large-scale turbulent three-dimensional channel flow simulation. Section~\ref{sec:conclusion} summarizes findings and discusses future directions.

\section{Background}\label{sec:bg}

In~\Cref{sec:bg:opinf}, we provide the necessary background on \gls{opinf} for reduced operator learning. We then introduce the streaming numerical linear algebra methods upon which our proposed streaming operator learning framework is built, beginning with streaming SVD methods in~\Cref{sec:bg:isvd} and finishing with recursive least squares methods in~\Cref{sec:bg:rls}.

\subsection{The Operator Inference Problem}\label{sec:bg:opinf}

In this work, we consider an \(n\)-dimensional \gls{ode} with a quadratic nonlinearity in the state vector \(\xvec(t) \in \R^n\): 
\begin{equation}\label{eqn:poly-sys}
    \xdot(t) = \Amat_1 \xvec(t) + \Amat_2 (\xvec(t) \otimes \xvec(t)) + \Bmat\uvec(t) + \cvec,
\end{equation}
where \( \otimes \) denotes the Kronecker product. The system evolves over time \( t \in [0, T]\) from an initial condition \( \xvec(0) = \xvec_0 \). Here, \(\Amat_1\) and \(\Amat_2\) are the linear and quadratic operators, \(\Bmat\in\R^{n\times m}\) is the input-to-state mapping for a bounded input \(\uvec\in\R^m\) with \( m \ll n \), and \(\cvec\in\R^n\) is a constant term representing a mean shift or other affine contribution. This linear-quadratic structure in~\eqref{eqn:poly-sys} is representative but not restrictive; the \gls*{opinf} framework accommodates alternative polynomial structures, including cubic or higher-order terms, as well as systems without input or constant terms. We demonstrate this flexibility in~\Cref{sec:experiments}.

This work is motivated by the setting where \( n \) is on the order of millions or larger, as commonly arises from spatial discretization of \glspl*{pde}. We assume that the matrix operators in~\eqref{eqn:poly-sys} are unavailable in explicit form (i.e., non-intrusive). Our goal is to learn a reduced model of dimension \( r \ll n \) from sampled data that enables rapid, accurate prediction. Specifically, we assume that the following data are available: \( K \) state snapshots \( \{\xvec_k\}_{k=1}^K \) and their corresponding time derivatives \( \{\xdot_k\}_{k=1}^K \) and input snapshots \( \{\uvec_k\}_{k=1}^K \). These data usually correspond to saving the numerical solution of~\eqref{eqn:poly-sys} at discrete times \( t_1, \ldots, t_K \). If time derivatives are not directly available they may be approximated via finite differences.

To learn a reduced model, we employ \gls{opinf}~\cite{peherstorferW2016}, which follows a \emph{reduce-then-learn} paradigm: first project the high-dimensional data onto a low-dimensional subspace, then solve a data-driven regression problem to learn reduced operators from the projected data to obtain a reduced model that preserves the polynomial structure of~\eqref{eqn:poly-sys}. This approach avoids access to the underlying code or full operators in~\eqref{eqn:poly-sys}. We summarize these two key steps below:
\begin{enumerate}[label=\textbf{\arabic*.}, leftmargin=*, itemsep=0em]
    \item \textbf{Reduction step:} This step constructs a reduced basis via \gls{pod}~\cite{lumley1967,sirovich1987,berkoozHL1993} and projects the data onto this basis. Let us form the snapshot matrix \( \Xmat = [\xvec_1, \ldots, \xvec_K] \in \R^{n \times K} \) from the sampled snapshots. \Gls{pod} is based on the \gls*{svd} \(\Xmat = \Vmat \mathbf{\Sigma} \Wmat^\top \), where \(\Vmat\in\R^{n\times \ell}\) and \(\Wmat\in\R^{K\times \ell}\) are orthogonal matrices satisfying \(\Vmat^\top \Vmat = \Wmat^\top\Wmat= \Id_\ell \), and \(\mathbf{\Sigma}\in\R^{\ell\times \ell}\) is a diagonal matrix with singular values \(\sigma_1 \geq \cdots \geq \sigma_\ell \geq 0\) for \(\ell \leq \min \{n,K\} \). The \gls{pod} basis \( \Vmat_r \in \R^{n\times r}\) is then constructed from the leading \( r \) left singular vectors that capture the dominant energy modes. Further, define the reduced state \(\xhat(t) = \Vmat_r^\top\xvec(t) \in \R^r\). Projecting the \( K \) snapshots onto the \gls*{pod} basis yields reduced states \( {\{\xhat_k \}}_{k=1}^K = {\{ \Vmat_r^\top\xvec_k \}}_{k=1}^K \) and reduced time derivatives \( {\{\dot{\xhat}_k \}}_{k=1}^K = {\{ \Vmat_r^\top\xdot_k \}}_{k=1}^K  \). This reduced data is the basis of the operator learning step.
    \item \textbf{Learning step:} Using the projected data, infer the reduced operators:
    \begin{equation}\label{eqn:intrusive-reduced-model}
        \dot{\xhat}(t) = \Ahat_1 \xhat(t) + \Ahat_2 (\xhat(t) \otimes \xhat(t)) + \Bhat\uvec(t) + \chat,
    \end{equation}
    where \( \Ahat_1 \in \R^{r\times r}\), \( \Ahat_2 \in \R^{r\times r^2}\), \( \Bhat \in \R^{r\times m}\), and \( \chat \in \R^r \) are determined by solving the regularized \gls*{ls} problem:
    \begin{equation}\label{eqn:opinf-euclidean}
        \min_{\Ahat_1, \Ahat_2, \Bhat, \chat} \frac{1}{K} \sum_{k=1}^K \left \| \dot{\xhat}_k - \Ahat_1\xhat_k - \Ahat_2(\xhat_k \otimes \xhat_k) - \Bhat\uvec_k - \chat \right \|^2_2 
        + 
        \left \| \Gammamat^{1/2} {[\Ahat_1, \Ahat_2, \Bhat, \chat]}^\top \right \|_F^2,
    \end{equation}
    where the positive definite matrix \( \Gammamat = \Gammamat^\top \in \R^{d \times d} \) with \( d = r + r^2 + m + 1 \) defines a Tikhonov regularization term~\cite{mcquarrie2021Datadriven}.
\end{enumerate}
When the fitted model~\eqref{eqn:intrusive-reduced-model} has the same polynomial structure as~\eqref{eqn:poly-sys}, unregularized \gls*{opinf} with \( \Gammamat = \0_{d\times d} \) recovers the intrusive Galerkin projection operators under certain conditions~\cite[Thm.\,1]{peherstorferW2016}.
However, regularization is particularly important in the streaming setting due to error accumulation over many operator updates, as we will discuss in~\Cref{sec:stream}. 

The \gls*{opinf} problem~\eqref{eqn:opinf-euclidean} may be expressed as a standard regularized linear \gls*{ls} problem as follows:
\begin{equation}\label{eqn:opinf-frobenius}
    (\textbf{OpInf}\,) \qquad \min_{\Omat} \frac{1}{K} \|\Rmat - \Dmat\Omat \|_F^2
    +
    \|\Gammamat^{1/2}\Omat \|_F^2,
\end{equation}
where the right-hand side, data, and operator matrices are defined to be
\begin{equation}\label{eqn:opinf-main-matrices}
    \begin{gathered}
        \Rmat = {[\dot{\xhat}_1,\, \ldots,\, \dot{\xhat}_K]}^\top \in \R^{K\times r}, \qquad
        \Dmat = [\Xhat^\top, ~{(\Xhat \odot \Xhat)}^\top,~\Umat^\top,~\mathbf{1}_K] \in \R^{K\times d}, \qquad
        \Omat = {[\Ahat_1, ~\Ahat_2,~\Bhat,~\chat]}^\top \in \R^{d\times r},
    \end{gathered}
\end{equation}
with
\begin{equation}\label{eqn:opinf-data-matrices}
    \begin{gathered}
        \Xhat = [\xhat_1,\, \ldots,\, \xhat_K] \in \R^{r\times K}, \qquad  
        \Xhat \odot \Xhat = [\xhat_1 \otimes \xhat_1,\, \ldots,\, \xhat_K \otimes \xhat_K] \in \R^{r^2\times K}, \qquad
        \Umat = [\uvec_1,\, \ldots,\, \uvec_K] \in \R^{m\times K},
    \end{gathered}
\end{equation}
where \(\mathbf{1}_K \in \R^{K}\) is a vector of ones and \( \odot \) denotes the Khatri-Rao product (column-wise Kronecker product).

Alternatively, augmenting the data and right-hand side matrices yields the following linear least squares formulation:
\begin{equation}\label{eqn:opinf-pinv-regularized}
    \min_{\Omat} \frac{1}{K}\|\bbar{\Rmat} - \bbar{\Dmat}\Omat \|_F^2, \qquad \text{where} \qquad 
    \bbar{\Dmat} = \begin{bmatrix} \Dmat \\ \Gammamat^{1/2} \end{bmatrix} \in \R^{(K+d)\times d} \quad \text{ and } \quad 
    \bbar{\Rmat} = \begin{bmatrix} \Rmat \\ \0_{d\times r} \end{bmatrix} \in \R^{(K+d)\times r}.
\end{equation}
If \( \Dmat \) has full column rank and \( \Gammamat \neq \0_{d \times d} \), this overdetermined system has unique solution \( \Omat = \bbar{\Dmat}^\dagger\bbar{\Rmat} \). We present formulation~\eqref{eqn:opinf-pinv-regularized} since it makes clear the \gls{rls} formulations introduced in~\Cref{sec:bg:rls}.

Standard \gls*{opinf} is a \textit{batch} method built on standard direct computation of the \gls*{svd} and direct solution of the \gls*{ls} problem. Loading and storing the entire data set in working memory incurs a memory cost of \( O(n(K+r)) \) for the \gls*{svd} and \( O(d(K+d)) \) for the \gls*{ls} solution, which is infeasible for real-time streaming data settings and large-scale problems. While distributed computing~\cite{farcas2025Distributed} and domain decomposition~\cite{farcas2024Domain} address the memory bottleneck, we pursue an alternative by adapting OpInf to the streaming setting, where snapshots \(\xvec_1,\ldots,\xvec_k \) are processed incrementally as they arrive. This requires two key ingredients: (i) incremental basis construction via \gls*{isvd}, and (ii) recursive operator learning via \gls*{rls}\@ which we now introduce. 

\subsection{Streaming SVD Methods}\label{sec:bg:isvd}

 We seek to address the high memory cost in standard \gls*{svd} for storing the data and basis, by introducing two streaming \gls*{svd} algorithms, Baker's \gls*{isvd}~\cite{baker2012Lowrank} and SketchySVD~\cite{tropp2019Streaming}.

\subsubsection{Baker's iSVD Algorithm}\label{sec:bg:baker-isvd}

Assume at time step \(k\) we have streamed snapshots \( [\xvec_1,\ldots,\xvec_k]\) and computed a rank-\( r_k \) SVD of the data matrix \(\Xmat_k \in \R^{n \times k}\) as \( \Xmat_k = \Vmat_{r_k} \, \boldsymbol{\Sigma}_{r_k}\, \Wmat^\top_{r_k} \), where \( k \in \{1,\ldots,K\} \) and \( r_k \leq r \). 

The core innovation in iSVD algorithms~\cite{bunch1978Updating,gu1993Stable,chandrasekaran1997Eigenspace,levey2000Sequential,brand2002Incremental,baker2012Lowrank} is efficiently updating the SVD components when a new snapshot \(\xvec_{k+1}\) arrives. This rank-1 increment is based on the identity:
\begin{equation}\label{eqn:isvd-identity}
    \begin{bmatrix} \Vmat_{r_k} \,\Sigmamat_{r_k}\,\Wmat^\top_{r_k} & \xvec_{k+1} \end{bmatrix}
    = \begin{bmatrix} \Vmat_{r_k} & \xvec_\perp/p \end{bmatrix} 
    \begin{bmatrix} \Sigmamat_{r_k} & \qvec \\ \mathbf{0}_{1 \times r_k} & p \end{bmatrix} 
    \begin{bmatrix} \Wmat_{r_k} & \mathbf{0}_{k\times 1} \\ \mathbf{0}_{1\times r_k} & 1 \end{bmatrix}^\top
    = \wh{\Vmat}\Jmat\wh{\Wmat}^\top,
\end{equation}
where \( \qvec = \Vmat^\top_{r_k} \xvec_{k+1} \in \R^{r_k}\), \( \xvec_\perp = \xvec_{k+1} - \Vmat_{r_k} \,\qvec \in \R^n \), and \( p = \|\xvec_\perp \|_2 \). Here, \( \Vmat_{r_k} \) is augmented by the normalized residual \( \xvec_\perp/p \), creating a small \( (r_k+1) \times (r_k+1) \) intermediate matrix \(\Jmat \), while the rightmost matrix expands \( \Wmat_{r_k} \) to accommodate the new snapshot.

The SVD matrices are then updated by computing \( \Vmat_J \,\Sigmamat_J \,\Wmat_J^\top = \mathsf{svd}(\Jmat) \) and setting
\begin{equation}\label{eqn:isvd-update}
    \Vmat_{r_{k+1}} = \wh{\Vmat}\Vmat_J, \qquad 
    \boldsymbol{\Sigma}_{r_{k+1}} = \boldsymbol{\Sigma}_J, \qquad 
    \Wmat_{r_{k+1}} = \wh{\Wmat} \Wmat_J.
\end{equation}

In this work, we employ Baker's iSVD framework~\cite{baker2012Lowrank}, which builds on these identities~\eqref{eqn:isvd-identity} and~\eqref{eqn:isvd-update} as summarized in~\Cref{alg:baker-isvd}. Baker's method is selected specifically because, similar to~\cite{zha1999Updating}, it addresses two key challenges in standard \gls*{isvd} methods: determining whether new snapshots warrant basis updates (typically via residual norm \(p\) thresholding~\cite{brand2002Incremental}) and maintaining numerical stability against accumulated errors (often requiring ad-hoc reorthogonalization~\cite{oxberry2017Limitedmemory,brand2002Incremental}).

\subsubsection{Randomized SketchySVD Algorithm}\label{sec:bg:sketchysvd}

Rather than updating the \gls*{svd} components at each iteration as in Baker's \gls*{isvd}, SketchySVD~\cite{tropp2019Streaming} is a randomized algorithm that computes the truncated \gls*{svd} by multiplying the data with predefined random matrices to compress them into low-dimensional randomized \textit{sketches} during streaming and computing the final \gls*{svd} only after processing the entire dataset.

Let the sketch sizes be denoted \( q \) and \( s \), which control the dimensions of the compressed representations and determine the accuracy-memory tradeoff. SketchySVD maintains three randomized sketches \( \Xcal_{\mathsf{range}}\in\R^{n\times q},\, \Xcal_{\mathsf{corange}}\in\R^{q\times K},\, \Xcal_{\mathsf{core}}\in\R^{s\times s} \) that capture the range, corange, and core information of the data matrix \( \Xmat \), respectively~\cite{tropp2019Streaming}. The core information is essential to improve the estimate of the singular values and vectors. These sketches are formed as products of \( \Xmat \) with four fixed random reduction maps \( \boldsymbol{\Upsilon}\in \R^{q\times n},\, \boldsymbol{\Omega}\in\R^{q\times K},\, \boldsymbol{\Xi}\in\R^{s\times n},\, \boldsymbol{\Psi}\in\R^{s\times K} \):
\begin{equation*}
    \Xcal_{\mathsf{range}} = \Xmat \boldsymbol{\Omega}^\top, \qquad 
    \Xcal_{\mathsf{corange}} = \boldsymbol{\Upsilon} \Xmat, \qquad 
    \Xcal_{\mathsf{core}} = \boldsymbol{\Xi} \Xmat \boldsymbol{\Psi}^\top.
\end{equation*}
These randomized sketches preserve essential geometric and algebraic properties (norms, inner products, and low-rank structure) of the original data, enabling memory-efficient approximate SVD recovery with high probability~\cite{halko2011Findinga}. In this work, we employ \textit{sparse sign matrices} for the four random reduction maps to reduce the computational cost of matrix-vector and matrix-matrix multiplications when updating sketches with each new data stream (see~\cite{tropp2019Suplementary} for additional details). Note that other types of random matrices, such as Gaussian and subsampled randomized Fourier transform (SRFT) matrices, can also be used~\cite{tropp2019Streaming}.

As summarized in~\Cref{alg:sketchy}, SketchySVD processes each incoming snapshot \( \xvec_k \) by updating the three sketches via efficient sparse matrix-vector multiplications, and then the random sketches are used to compute an estimate of the \gls{svd} only after all snapshots have been processed.
To ensure a truncated SVD with sufficiently tight theoretical error bounds, the recommended sketch sizes are \( q = 4r+1 \) and \( s = 2q+1 \)~\cite{tropp2019Streaming}. Further details are available in~\cite{tropp2019Streaming}.

\subsubsection{Cost and Error Analysis}\label{sec:bg:isvd-costs}
To assess suitability for our proposed method, we analyze memory and computational costs as well as sources of error for streaming \gls{svd} algorithms relative to batch \gls{svd}, tabulated in~\Cref{tab:isvd-costs}.

Baker's method has space complexity \( O(nr) \) for storing \(\Vmat_r \), making it suitable for large-scale problems where \( n \gg r \) and \( r < K \). The per-iteration cost of \( O(nr) \) yields total complexity \( O(nKr) \), matching batch \gls*{svd}. However, error characteristics differ fundamentally. While batch \gls*{svd} error stems solely from truncating small singular values~\cite{eckart1936approximation}, Baker's \gls*{isvd} introduces incremental error accumulating with each update. Denoting the final \gls*{isvd} approximation as \( \widetilde{\Xmat}_{\mathsf{isvd}} = \widetilde{\Vmat}\widetilde{\Sigmamat}\widetilde{\Wmat}^\top \) corresponding to time step \( k = K \) versus the batch result \( \Xmat = \Vmat\Sigmamat\Wmat^\top \), there exists a bound \( \varepsilon > 0 \) satisfying \( \| \Xmat - \widetilde{\Xmat}_{\mathsf{isvd}} \|_2 \leq \varepsilon \)~\cite{fareed2020Error}. This bound implies singular value perturbations \( |\sigma_j - \widetilde{\sigma}_j| \leq \varepsilon \) between diagonal entries of \( \Sigmamat \) and \( \ot\Sigmamat \), while singular vector errors depend on spectral gaps between consecutive singular values. Notably, right singular vectors exhibit larger errors than left vectors~\cite[Thm.\,3.9]{fareed2020Error}:
\begin{equation}\label{eqn:isvd-singular-vector-errors}
    \| \vvec_j - \widetilde{\vvec}_j \|_2 \leq E_j^{1/2}, \qquad \| \wvec_j - \widetilde{\wvec}_j \|_2 \leq E_j^{1/2} + 2\sigma_j^{-1}\varepsilon_j,
\end{equation}
where \( \vvec_j, \wvec_j \) (resp.~\( \ot\vvec_j,\ot\wvec_j \)) denote the \(j\)-th left and right singular vectors of \( \Vmat,\Wmat \) (resp.~\( \ot\Vmat,\ot\Wmat \)), and \( E_j, \varepsilon_j > 0 \) depend on \( \varepsilon \), singular values, and their gaps. Errors in both vectors increase as \(\sigma_j \to 0\) due to diminishing spectral gaps, which can be mitigated by choosing smaller \( r \) to focus on dominant modes. For more details, see~\cite{fareed2020Error}.

SketchySVD memory requirements are dominated by the dense range sketch \( \Xcal_{\mathsf{range}} \) at \( O(nq) \), with smaller corange and core sketches requiring \( O(qK) \) and \( O(s^2) \). The sparse sign matrix draws each column of the reduction maps independently from a distribution where each entry is \( +1 \) or \( -1 \) with probability based on \( \zeta = \min \{q, 8\} \), a sparsity parameter controlling the number of nonzeros per column, adding \( O(n\zeta) \) memory~\cite{tropp2019Streaming}. Total memory of \( O(n(q+\zeta) + qK + s^2) \) remains well below the \( O(nK) \) needed for full data if \( q, s \ll n \) and \( q, s < K \). The per-iteration cost of \( O(n\zeta) \) for three sparse matrix-vector multiplications yields total complexity \( O(nK\zeta) \), proportional to the large-scale dimension \( n \). Unlike Baker's deterministic accumulation, SketchySVD errors are probabilistic. For sketch parameters \( q \) and \( s \), the expected Frobenius error between \( \Xmat \) and its approximation \( \widetilde{\Xmat}_{\mathsf{sketchy}} \) satisfies~\cite[Thm.\,5.1]{tropp2019Streaming}
\begin{equation}\label{eqn:sketchy-error-bound}
    \mathbb{E}\| \Xmat - \widetilde{\Xmat}_{\mathsf{sketchy}} \|_F^2 \leq \frac{s - 1}{s - q - 1} \cdot \min_{\varrho < q - 1} \frac{q + \varrho - 1}{q - \varrho - 1} \cdot \sum_{j > \varrho+1} \sigma_j^2,
\end{equation}
where the tail energy \( \sum_{j > \varrho+1} \sigma_j^2 \) reveals SketchySVD's ability to exploit spectral decay. That is, increasing \( q \) captures more spectrum, reducing tail energy. Similar to Baker's method, errors grow with smaller spectral gaps, particularly pronounced in turbulent and advection-dominant flows with slow spectral decay. Such problems require larger \( q \) and \( s \) for accuracy, increasing memory usage.

The choice between methods depends on memory constraints, accuracy requirements, and spectral decay rates. Baker's \gls*{isvd} offers minimal memory overhead, ideal for problems with fast spectral decay and modest datasets. SketchySVD trades increased memory for better scalability to large datasets, with tunable sketch sizes balancing accuracy and storage. Section~\ref{sec:experiments} demonstrates effective integration of both methods within our streaming \gls*{opinf} framework across different problem types.

\begin{table}[htbp!]
    \centering
    \caption{Summary of memory and computational costs as well as sources of errors of the standard batch SVD, Baker's iSVD, and SketchySVD algorithms for computing a rank-\( r \) truncated SVD of a data matrix \( \Xmat \in \R^{n\times K} \). Here, \( q \) and \( s \) are the sketch sizes in SketchySVD such that \( r < q \leq s \), and \( \zeta = \min \{q,8\} \) is the sparsity parameter for the sparse sign matrices used as reduction maps.}
    \vspace{2mm}
    \begin{tabular}{c c c l}
        \toprule
        \textbf{Algorithm} & Memory Cost & Time Cost & Source of Error \\
        \midrule
        Standard Batch SVD & \(O(n(K+r))\) & \(O(nKr)\) & Truncation error based on \( r \)~\cite{eckart1936approximation} \\
        Baker's iSVD~\cite{baker2012Lowrank} & \(O(nr)\) & \(O(nKr)\) & Truncation error and spectral gap~\cite{fareed2020Error} \\
        SketchySVD~\cite{tropp2019Streaming} & \(O(n(q+\zeta) )\) & \(O(nK\zeta)\) & Error bound based on \(r, q, s \)~\cite{tropp2019Streaming} \\
        \bottomrule
    \end{tabular}\label{tab:isvd-costs}
\end{table}

\subsection{Recursive Least-Squares Methods}\label{sec:bg:rls}

We introduce two \gls{rls} algorithms: standard \gls{rls} and a more numerically stable \gls{iqrrls} in~\Cref{sec:bg:standard-rls,sec:bg:iqrrls}, respectively.

\subsubsection{Standard RLS Algorithm}\label{sec:bg:standard-rls}
We consider a streaming setting in which the rows of the OpInf least squares problem become available sequentially. 
At each step \( k \), new streaming data arrives as \( (\dvec_k, \rvec_k) \), where \( \dvec_k \in \R^{1\times d} \) represents the final row from the current data matrix \( \Dmat_k = [\dvec_1; \dvec_2; \ldots; \dvec_k] \in \R^{k\times d} \) and \( \rvec_k \in \R^{1\times r} \) is the corresponding final row from \( \Rmat_k = [\rvec_1; \rvec_2; \ldots; \rvec_k] \in \R^{k\times r} \). The streams \( \dvec_k \) and \( \rvec_k \) are formed by reduced snapshots \( \xhat_k \) and \( \xhatdot_k \) as defined in~\eqref{eqn:opinf-main-matrices}. To develop the RLS formulation, consider the normal equation solution to~\eqref{eqn:opinf-frobenius}, which gives the operator solution \( \Omat_k \) at step \( k \):
\begin{equation*}
    \Omat_k = \Pmat_k \Dmat^\top_k \Rmat_k \quad \text{ where } \quad  \Pmat_k = {(\Dmat^\top_k \Dmat_k + \Gammamat)}^{-1},
\end{equation*}
and \( \Pmat_k \in \R^{d\times d} \) is the \emph{inverse correlation matrix} at step \(k\). In RLS, the inverse \( \Pmat_k^{-1} \) is updated recursively as
\begin{equation*}
    \Pmat_k^{-1} = \Pmat_{k-1}^{-1} + \dvec_k^\top\dvec_k = \Pmat_0^{-1} + \sum_{i=1}^k \dvec_i^\top \dvec_i, \quad k = 1,2,\ldots,K,
\end{equation*}
initialized with \( \Pmat_0^{-1} = \Gammamat \). The inclusion of the Tikhonov regularization term \( \Gammamat \) is critical, since the RLS algorithm consists of many rank-1 updates that are prone to ill-conditioning without regularization.

The key insight of RLS is to efficiently update the inverse correlation matrix \( \Pmat_k \) using the Sherman-Morrison identity, which enables rank-1 updates without recomputing the entire matrix inverse~\cite{sayed2011Adaptive}. This approach transforms the \( O(d^3) \) matrix inversion into \( O(d^2) \) operations per iteration. The RLS algorithm hence becomes an update of the \gls*{ls} solution by adding a correction term based on the prediction error:
\begin{equation}\label{eqn:rls-update}
    \Omat_k = \Omat_{k-1} + \gvec_k \boldsymbol{\xi}_k^-,
\end{equation}
where \( \boldsymbol{\xi}_k^- = \rvec_k - \dvec_k \Omat_{k-1} \in \R^{1\times r} \) is the \textit{a priori} prediction error and \( \gvec_k = \Pmat_{k-1} \dvec_k^\top c_k \in \R^{d} \) is called the Kalman gain vector with conversion factor \( c_k > 0 \)~\cite{sayed2011Adaptive}. Both the conversion factor and the inverse correlation matrix are also updated recursively using the new data.
\gls*{rls} is initialized with \( \Omat_0 = \0_{d\times r} \) and repeats the update~\eqref{eqn:rls-update} as well as the updates for \( c_k \) and \( \Pmat_k \) for each new data stream \( (\dvec_k, \rvec_k) \) until all \( K \) data points are processed to yield the final operator estimate \( \Omat_K \). The \gls*{rls} algorithm is summarized in~\Cref{alg:rls}.

\subsubsection{Inverse QR-Decomposition RLS Algorithm}\label{sec:bg:iqrrls}
While the standard \gls*{rls}~\cref{alg:rls} is effective for recursive parameter estimation, it suffers from numerical instabilities when updating the inverse correlation matrix \( \Pmat_k \) in finite-precision arithmetic. These instabilities arise because the update involves subtracting nearly equal positive definite matrices, causing catastrophic cancellation~\cite{sayed2011Adaptive}. QR-decomposition RLS algorithms address this by organizing iterated quantities and incoming data into arrays that are factorized using QR decomposition to update variables in a numerically stable manner.

These algorithms rely on transformations \( \Acal_k \Thetamat_k = \Bcal_k \), where \( \Acal_k \) is the pre-array containing previous quantities and new data, \( \Thetamat_k \) is a unitary matrix, and \( \Bcal_k \) is the triangular post-array containing updated quantities. This is essentially a QR decomposition of \( \Acal_k^\top \). Two main variants exist: \gls{iqrrls} and \gls{qrrls}, which differ in whether they update the inverse correlation matrix \( \Pmat_k \) or the correlation matrix \( \sum_{i=1}^k\dvec_i^\top\dvec_i \)~\cite{sayed2011Adaptive,haykin2014Adaptive}.

The \gls{iqrrls} algorithm propagates the square root of the inverse correlation matrix \( \Pmat_k^{1/2} \) rather than \( \Pmat_k \) to ensure positive definiteness, using the Cholesky decomposition \( \Pmat_k = \Pmat_k^{1/2} (\Pmat_k^{1/2})\tran \) where \( \Pmat_k^{1/2} \) is lower triangular. The algorithm organizes the transposed square root \( \Pmat_{k-1}^{\top/2} := (\Pmat_{k-1}^{1/2})\tran \) and incoming data into the pre-array 
\begin{equation*}
    \Acal_k^\top = \begin{bmatrix}
        1 & \mathbf{0}_{1 \times d} \\[0.3em]
        \Pmat_{k-1}^{\top/2} \dvec_k^\top & \Pmat_{k-1}^{\top/2}
    \end{bmatrix} \in \R^{(d+1) \times (d+1)},
\end{equation*} 
then performs QR decomposition to obtain the post-array from which updated quantities can be extracted. We summarize \gls*{iqrrls} in~\Cref{alg:iqrrls}.

\subsubsection{Cost Analysis}\label{sec:bg:rls-costs}
To assess the computational efficiency of \gls{rls} methods for streaming \gls*{opinf}, we present their memory and computational costs, comparing them to standard batch \gls*{ls} in~\Cref{tab:rls-costs}. A key advantage of \gls*{rls} methods is their constant memory requirement of \( O(d^2) \), dominated by storing the inverse correlation matrix \( \Pmat_k \in \R^{d \times d} \), compared to the \( O(d(K+d)) \) storage required by batch \gls*{ls} for the concatenated data matrix \( \bbar{\Dmat} \). In the standard \gls*{rls} algorithm, each iteration requires \( O(d^2) \) operations for updating \( \Pmat_k \) and computing the Kalman gain, yielding a total cost of \( O(Kd^2) \) for processing \( K \) data points, which is comparable to the \( O(d^2\max \{K,d\} ) \) cost of batch \gls*{ls} solutions~\cite{sayed2011Adaptive}. Unlike batch methods, \gls*{rls} does not require storing the entire data history, making it well-suited for streaming applications.

The \gls*{iqrrls} algorithm addresses numerical stability by propagating the Cholesky factor \( \Pmat_k^{\top/2} \) (an upper triangular matrix) instead of the full inverse correlation matrix. At each iteration, \gls*{iqrrls} performs a QR decomposition of a \( (d+1) \times (d+1) \) pre-array matrix \( \Acal_k^\top \). A na\"{\i}ve implementation would require \( O(d^3) \) operations per iteration. However, because \( \Pmat_k^{\top/2} \) is upper triangular, \( \Acal_k^\top \) has a nearly upper triangular structure, allowing Givens rotations to be applied selectively to only the first column. This reduces the computational complexity to \( O(d^2) \) per iteration~\cite{sayed2011Adaptive}, matching standard \gls*{rls} asymptotically. For dimensions below 20, the computational difference is negligible, but Givens rotations become essential for larger problems. 
Moreover, the memory cost of \gls*{iqrrls} can be half of the standard \gls*{rls} if we opt to store only the upper triangular elements of \( \Pmat_k^{\top/2} \).

The \gls*{qrrls} algorithm similarly addresses numerical instability by propagating the Cholesky factor of the correlation matrix \( \sum_{i=1}^k\dvec_k^\top\dvec_k \) instead of its inverse. However, a key difference is that QRRLS requires an additional \gls*{ls} solve at each iteration after extracting the updated quantities from the post-array to compute the operator estimate \( \Omat_k  \). This incurs an additional cost per iteration, and thus, we consider iQRRLS as our primary array-based \gls*{rls} method in this work. For more details on both iQRRLS and QRRLS, see~\cite{sayed2011Adaptive}.

\begin{table}[htbp!]
    \centering
    \parbox{0.95\textwidth}{\caption{Summary of memory and computational costs of standard \gls*{ls}, standard RLS, and iQRRLS algorithms for solving the regularized \gls*{ls} problem~\eqref{eqn:opinf-frobenius} with \( K \) data points and \( d \) features.}\label{tab:rls-costs}}\\
    \vspace{2mm}
    \begin{tabular}{c c c}
        \toprule
        \textbf{Algorithm} & Memory Cost & Time Cost \\
        \midrule
        Standard Batch LS & \(O(d(K+d))\) & \(O(d^2\max \{K,d\})\) \\
        Standard RLS~\cite{sayed2011Adaptive} & \(O(d^2)\) & \(O(Kd^2)\) \\
        iQRRLS~\cite{sayed2011Adaptive} & \(O(d^2)\) & \(O(Kd^2)\) \\
        \bottomrule
    \end{tabular}
\end{table}

\section{Streaming Operator Inference for Large-Scale Data}\label{sec:stream}

In this section, we present Streaming \gls*{opinf}, a reduced operator learning method that extends the OpInf framework to the data-streaming setting. In~\Cref{sec:stream:formulation}, we detail the mathematical reformulation of the learning step for Streaming \gls*{opinf}. In~\Cref{sec:stream:selection}, we provide guidance on selecting appropriate algorithmic components based on problem characteristics. In~\Cref{sec:stream:error}, we analyze the sources of error in Streaming \gls*{opinf}. Finally, in~\Cref{sec:stream:practical}, we provide practical implementation details.

\subsection{Streaming Reformulation of the Learning Step}\label{sec:stream:formulation}

Adapting the standard \gls*{opinf} framework~\eqref{eqn:opinf-frobenius} to the data-streaming setting requires replacing the batch \gls*{svd} and \gls*{ls} computations with their streaming counterparts from~\Cref{sec:bg:isvd,sec:bg:rls}. A straightforward approach would be to first compute the \gls*{pod} basis via \gls*{isvd}, then make an additional pass through the data to project each snapshot and its time derivative onto this basis, and finally solve for the operators using either batch \gls*{ls} or \gls*{rls}. We refer to this as the \emph{projection} approach, which yields the \emph{iSVD-Project-LS} and \emph{iSVD-Project-RLS} paradigms. However, this projection step incurs a computational cost of \( O(nr) \) per snapshot---totaling \( O(nKr) \) for \( K \) snapshots---which can be prohibitive when \( n \) is extremely large.

To circumvent this additional data pass and its associated cost, we introduce a reformulation that expresses the \gls*{ls} data matrices directly in terms of the truncated \gls*{svd} matrices obtained from \gls*{isvd}, thereby avoiding explicit projection entirely. This \emph{reformulation} approach yields the \emph{iSVD-LS} and \emph{iSVD-RLS} paradigms. While this reformulation eliminates the \( O(nKr) \) projection cost and associated data pass, the projection-based paradigms offer their own advantages in terms of error control, as we elaborate in~\Cref{sec:stream:error:svd}.

Specifically, suppose streamed snapshots \( \Xmat_k = [\xvec_1, \ldots, \xvec_k] \in \R^{n \times k} \) have been processed by \gls*{isvd} up to time step \( k \), yielding a rank-\(r_k\) truncated \gls*{svd} approximation \( \Xmat_k \approx \Vmat_{r_k} \Sigmamat_{r_k} \Wmat_{r_k}^\top \) with \( r_k \leq r \). The reduced snapshot matrix can then be expressed as
\begin{equation}\label{eqn:svd-reduced-snapshot}
    \Xhat_k = \Vmat_{r_k}^\top \Xmat_k = \Sigmamat_{r_k} \Wmat_{r_k}^\top,
\end{equation}
where each row corresponds to a right singular vector scaled by its singular value. For quadratic terms, we employ properties of the Kronecker product~\cite{brewer1978Kronecker} and Khatri-Rao product~\cite{khatri1968Solutions,rao1970Estimation} to obtain
\begin{equation}\label{eqn:svd-reduced-quadratic}
    \begin{gathered}
    \Xmat_k \odot \Xmat_k \approx 
    (\Vmat_{r_k} \otimes \Vmat_{r_k}) (\Sigmamat_{r_k} \otimes \Sigmamat_{r_k}) (\Wmat_{r_k}\tran \odot \Wmat_{r_k}\tran)
    \implies 
    \Xhat_k \odot \Xhat_k = \Sigmamat_{r_k}\Wmat_{r_k}^\top \odot \Sigmamat_{r_k}\Wmat_{r_k}^\top,
    \end{gathered}
\end{equation}
where we use \( (\Vmat_{r_k} \otimes \Vmat_{r_k})\tran(\Vmat_{r_k}\otimes\Vmat_{r_k})=\Id_{r_k^2} \). 

Let \( (i,j) \) denote the matrix entry at row \( i \) and column \( j \). Using~\eqref{eqn:svd-reduced-snapshot} and~\eqref{eqn:svd-reduced-quadratic}, the data matrix \( \Dmat_k \) from~\eqref{eqn:opinf-main-matrices} can be reformulated solely in terms of the rank-\( r \) truncated matrices \( \Sigmamat_r = \Sigmamat_{r_k}(1:r,1:r) \) and \( \Wmat_r = \Wmat_{r_k}(:,1:r) \):
\begin{equation*}
    \Dmat_k = \begin{bmatrix}
        (\Sigmamat_r\Wmat_r^\top)\tran &
        (\Sigmamat_r\Wmat_r^\top \odot \Sigmamat_r\Wmat_r^\top)\tran &
        \Umat_k^\top & \mathbf{1}_k 
    \end{bmatrix}.
\end{equation*}

The right-hand side matrix \( \Rmat_k \) can similarly be expressed using \( \Sigmamat_r \) and \( \Wmat_r \) when time derivatives are approximated via finite differences. Denoting the finite difference operator matrix as \( \Deltamat_k \in \R^{k \times \tilde{k}} \), where \( \widetilde{k} \le k \) depends on the chosen scheme, the reduced time derivative matrix becomes \( \Xhatdot_k = \Sigmamat_r\Wmat_r^\top\Deltamat_k \), and thus \( \Rmat_k = \Deltamat_k^\top\Wmat_r\Sigmamat_r \). For example, forward differences with constant time step \( \delta t \) yield
\begin{equation*}
    \Deltamat_k = \frac{1}{\delta t}\begin{bmatrix}
        -1 & 0 & 0 & \cdots & 0 \\
        1 & -1 & 0 & \cdots & 0 \\
        0 & 1 & -1 & \cdots & 0 \\
        \vdots &  & \vdots & \ddots & \vdots \\
        0 & \cdots & 0 & 1 & -1 \\
        0 & \cdots & 0 & 0 & 1
    \end{bmatrix} \in \R^{k \times (k-1)}.
\end{equation*}

Since finite differencing can change the column dimension from \( k \) to \( \widetilde{k} \), a column selector matrix \( \Smat_k \in \R^{k \times \tilde{k}} \) extracts the appropriate indices from the snapshot data (e.g., indices \( 1 \) to \( k-1 \) for forward differences). The final formulation with Tikhonov regularization becomes
\begin{equation}\label{eqn:compact-data-matrices-regularized}
    \bbar{\Dmat}_k = \begin{bmatrix}
        (\Sigmamat_r\Wmat_r^\top\Smat_k)\tran & 
        (\Sigmamat_r\Wmat_r^\top\Smat_r \odot \Sigmamat_r\Wmat_r^\top\Smat_k)\tran &
        (\Umat_k\Smat_k)\tran & 
        \mathbf{1}_{\tilde{k}} \\
        \multicolumn{4}{c}{{\Gammamat}^{1/2}}
    \end{bmatrix}, \qquad 
    \bbar{\Rmat}_k = \begin{bmatrix} 
        (\Sigmamat_r\Wmat_r^\top\Deltamat_k)\tran \\
        \0_{d \times r}
    \end{bmatrix}.
\end{equation}

This reformulation offers two key advantages. First, it computes the data matrices \( \bbar{\Dmat}_k \) and \( \bbar{\Rmat}_k \) directly from the \gls*{isvd} output without storing more than one snapshot at a time. Second, it avoids explicit projection of \( \Xmat_k \) onto the \gls*{pod} basis, eliminating the \( O(nKr) \) computational overhead of the projection step. The operator solution \( \Omat_k = \bbar{\Dmat}_k^\dagger\bbar{\Rmat}_k \) can then be computed at any time step \( k \) using only the data processed thus far, enabling online operator learning with the potential for real-time predictions.

\subsection{Streaming Paradigms and Algorithm Selection}\label{sec:stream:selection}

As summarized in~\Cref{fig:streaming-opinf-flowchart} and~\Cref{tab:streaming-opinf-costs}, we thoroughly explore the design space of Streaming \gls*{opinf} and propose four distinct paradigms---\emph{iSVD-Project-LS}, \emph{iSVD-Project-RLS}, \emph{iSVD-LS}, and \emph{iSVD-RLS}---distinguished by two independent design choices: (i) whether time derivative data is available directly or must be approximated via finite differences, and (ii) whether to solve the \gls*{ls} problem in batch or recursively.

\begin{figure}[htbp!]
    \centering
    \includegraphics[width=\textwidth]{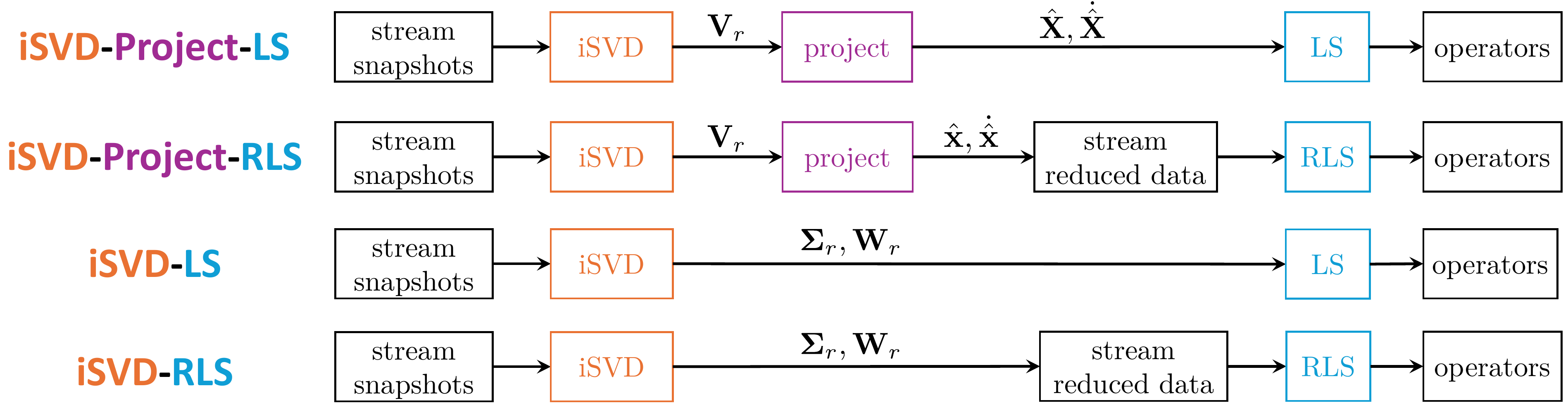}
    \vspace{-1.5em}
    \caption{Flowchart of each paradigm of Streaming OpInf showing different algorithmic components, selected based on time derivative data availability and memory/computational cost considerations.}\label{fig:streaming-opinf-flowchart} 
\end{figure}

The first design choice concerns how the reduced time derivative data is obtained. When time derivative data is unavailable or unreliable (e.g., corrupted by measurement noise), the reformulation~\eqref{eqn:compact-data-matrices-regularized} approximates derivatives via finite differences applied directly to the \gls*{isvd} output, yielding the iSVD-LS/RLS paradigms. This approach avoids the computational cost of projecting data onto the \gls*{pod} basis. Conversely, when high-quality time derivative data is available from simulations or experiments, it can be projected onto the \gls*{pod} basis after computing the basis via \gls*{isvd}. This projection approach, yielding the iSVD-Project-LS/RLS paradigms, avoids the potential inaccuracies of finite difference approximations and is generally preferred when accurate derivative data is accessible. However, it requires an additional pass through the data with a total cost of \( O(2nKr) \) for projecting both snapshots and derivatives. This cost is often manageable through parallelized batch processing, though it may be prohibitive for extremely large \( n \) or when data cannot be revisited.

The second design choice concerns how the linear \gls*{ls} problem is solved. In the iSVD-LS and iSVD-Project-LS paradigms, the operator is computed via batch \gls*{ls} using the data matrices from~\eqref{eqn:compact-data-matrices-regularized} or~\eqref{eqn:opinf-main-matrices}, respectively. For iSVD-LS, this approach enables computation of the operator at any intermediate time step \( k \), which may be valuable for applications requiring online operator estimates. Alternatively, the iSVD-RLS and iSVD-Project-RLS paradigms employ the \gls*{rls} methods from~\Cref{sec:bg:rls}.\ iSVD-RLS streams \( \xhat_k = \Sigmamat_r \Wmat_r^\top(:,k) \), with time derivatives approximated using finite differences over a sliding window while iSVD-Project-RLS projects streams \( \{ \xhat,\xhatdot \} \). Unlike the batch approaches, the \gls*{rls} paradigms yield the operator only after processing all \( K \) snapshots, since \gls*{rls} updates commence after \gls*{isvd} completion. However, this approach requires only \( O(d^2) \) additional memory compared to \( O(d(K+d)) \) for the batch paradigms (see~\Cref{tab:streaming-opinf-costs}), making it more memory-efficient and fully streaming without batch computations.

\begin{table}[htbp!]
    \centering
    \caption{Summary of memory and computational costs for Streaming OpInf paradigms. Here, \( n \) is the full system dimension, \( K \) is the number of snapshots, \( r \) is the reduced basis dimension, \( d = r + r^2 + m + 1 \) is the total operator dimension, \( m \) is the input dimension, \( r < q, s \) are SketchySVD sketch sizes, and \( \zeta = \min \{q,8\} \) is the sparsity parameter.}
    \vspace{2mm}
    \begin{tabular}{c c c c}
        \toprule
        \textbf{Streaming OpInf} & \textbf{iSVD Method} & \textbf{Memory Cost} & \textbf{Computational Cost} \\
        \midrule
        \multirow{2}{*}{\textbf{iSVD-LS}} 
        & Baker's iSVD & \(O(nr + d(K+d))\) & \(O(nKr + d^2\max \{d,K\})\) \\
        & SketchySVD & \(O(n(q+\zeta) + d(K + d))\) & \(O(nK\zeta + d^2\max \{d,K\})\) \\
        \midrule
        \multirow{2}{*}{\textbf{iSVD-RLS}}
        & Baker's iSVD & \(O(nr + d^2)\) & \(O(nKr + Kd^2)\) \\
        & SketchySVD & \(O(n(q+\zeta) + d^2)\) & \(O(nK\zeta + Kd^2)\) \\
        \midrule
        \multirow{2}{*}{\textbf{iSVD-Project-LS}}
        & Baker's iSVD & \(O(nr + d(K+d))\) & \(O(3nKr + d^2\max \{d,K\})\) \\
        & SketchySVD & \(O(n(q+\zeta) + d(K+d))\) & \(O(nK(\zeta + 2r) + d^2\max \{d, K\})\) \\
        \midrule
        \multirow{2}{*}{\textbf{iSVD-Project-RLS}}
        & Baker's iSVD & \(O(nr + d^2)\) & \(O(3nKr + Kd^2)\) \\
        & SketchySVD & \(O(n(q+\zeta) + d^2)\) & \(O(nK(\zeta + 2r) + Kd^2)\) \\
        \bottomrule
    \end{tabular}\label{tab:streaming-opinf-costs}
\end{table}

To summarize, practitioners should select among the four paradigms based on the following considerations: (i) use projection paradigms (iSVD-Project-LS/RLS) when accurate derivative data is available and the \( O(2nKr) \) projection cost is acceptable, and reformulation paradigms (iSVD-LS/RLS) otherwise; (ii) use the iSVD-LS paradigm when intermediate operator solutions at time \( k < K \) are needed for possibly real-time scenarios; and (iii) use RLS paradigms (iSVD-RLS or iSVD-Project-RLS) when memory is limited and \( K \) is large, as they require only \( O(d^2) \) additional storage. Additional practical considerations such as implementation complexity, numerical stability, and computational infrastructure may also influence the choice in specific application contexts. Moreover, other paradigms, for example, hybrid approaches which combine iSVD-LS/\gls*{rls} by computing a sufficiently accurate \gls*{pod} basis up to time \( k' < K \) using a subset of the data and then performing projection and \gls*{rls} updates for \( k > k' \), can also be devised within the Streaming \gls*{opinf} framework to suit particular needs.

\subsection{Sources of Error}\label{sec:stream:error}

Streaming \gls*{opinf} introduces approximation errors relative to batch \gls*{opinf} from two sources: (i) errors in the \gls*{pod} basis from streaming \gls*{svd} methods, and (ii) errors in the operator solution from using \gls*{rls} instead of batch \gls*{ls}. We analyze each source in this section.

\subsubsection{Errors from Streaming SVD}\label{sec:stream:error:svd}

Baker's \gls*{isvd} produces deterministic errors that accumulate with each truncation~\cite{fareed2020Error,baker2012Lowrank}, while SketchySVD produces probabilistic errors controlled by sketch parameters~\cite{tropp2019Streaming}. Both methods exhibit increased errors when spectral gaps diminish, as discussed in \Cref{sec:bg:isvd-costs}. These errors propagate to the learned operators through the \gls*{ls} data matrices. Specifically, the iSVD-LS/RLS paradigms construct data matrices~\eqref{eqn:compact-data-matrices-regularized} using right singular vectors \( \Wmat_r \) and singular values \( \Sigmamat_r \), which carry larger errors than left singular vectors according to~\eqref{eqn:isvd-singular-vector-errors}. In contrast, the iSVD-Project-LS/RLS paradigms project data onto the \gls*{pod} basis using left singular vectors \( \Vmat_r \), yielding lower operator error at additional computational cost. We formalize this error propagation in the following lemma and theorem proven in~\Cref{app:proof}. \Cref{lem:operator-perturbation} establishes a general perturbation bound for the learned operators in terms of perturbations to the data matrices.

\begin{lemma}[Operator perturbation bound]\label[lemma]{lem:operator-perturbation}
    Let \( \ot{\Dmat} = \bbar\Dmat + [\Emat_D^\top, \0_{d\times d}]\tran \) and \( \ot{\Rmat} = \bbar\Rmat + [\Emat_R^\top, \0_{r \times d}]\tran \) denote perturbed data matrices with perturbations \( \Emat_D \in \R^{K \times d} \) and \( \Emat_R \in \R^{K \times r} \). Assume \( \bbar\Dmat \) and \( \ot{\Dmat} \) both have full column rank \( d \). Then the operator error between solutions \( \Omat = \bbar\Dmat^\dagger\bbar\Rmat \) and \( \ot{\Omat} = \ot{\Dmat}^\dagger\ot{\Rmat} \) satisfies
    \begin{equation}\label{eqn:lem:operator-perturbation}
        \| \Omat - \ot{\Omat} \|_F \leq 
        \alpha \| \bbar\Dmat^\dagger \|_2 \| \ot\Dmat^\dagger \|_2 \| \Emat_D \|_2 \| \ot\Rmat \|_F +
        \| \bbar\Dmat^\dagger \|_2 \| \Emat_R \|_F,
    \end{equation}
    where \( \alpha = \sqrt{2} \) if \( K \neq d \) and \( \alpha = 1 \) if \( K = d \).
\end{lemma}

We now combine~\Cref{lem:operator-perturbation} with explicit bounds on the data matrix perturbations \( \Emat_D \) and \( \Emat_R \) arising from streaming \gls*{svd} errors to establish the following theorems, which bound the resulting operator errors for each paradigm.

\begin{theorem}[Operator error bounds: Projection paradigms]\label{thm:operator-error-projection}
    Let \( \bbar\Dmat \) and \( \bbar\Rmat \) be constructed using the exact rank-\( r \) \gls*{svd} of \( \Xmat \), and let \( \ot{\Dmat} \) and \( \ot{\Rmat} \) be constructed using approximate factors \( \ot{\Vmat}_r \), \( \ot{\Sigmamat}_r \), \( \ot{\Wmat}_r \) from streaming \gls*{svd} with the left singular vector error \( \tau_v = \| \Vmat_r - \ot{\Vmat}_r \|_F \). Assume the time-derivative data is computed via finite differences, \( \Xdot = \Xmat\Deltamat_K \). Furthermore, suppose \( \| \Xmat - \ot{\Xmat} \|_2 \leq \varepsilon \) and \( \Gammamat = \mathrm{diag}(\gamma_1, \ldots, \gamma_d) \) with \( \gamma_j > 0 \) for all \( j \), and let \( \sqrt{\gamma_{\min}} = \min_j \sqrt{\gamma_j} \). Then for the iSVD-Project-LS/RLS paradigms, the data matrix perturbations satisfy
    \begin{equation}\label{eqn:thm:perturbation-projection}
        \| \Emat_D \|_2 \leq \frac{\beta_1}{\sqrt{\min \{n,K\}}}\tau_v, \qquad 
        \| \Emat_R \|_F \leq \sigma_1 \|\Deltamat_K \|_2 \tau_v,
    \end{equation}
    and the operator error between the batch solution \( \Omat = \bbar\Dmat^\dagger\bbar\Rmat \) and the streaming solution \( \ot{\Omat} = \ot{\Dmat}^\dagger\ot{\Rmat} \) satisfies
    \begin{equation}\label{eqn:thm:operator-error-projection}
        \| \Omat - \ot{\Omat} \|_F \leq \frac{\sigma_1\|\Deltamat_K \|_2}{\sqrt{\gamma_{\min}}} \tau_v \Bigl(1 + \frac{\alpha\beta_1}{\sqrt{\gamma_{\min}}}\Bigr),
    \end{equation}
    where \( \beta_1 = 2 (\sigma_1^4 + \sigma_1^2 + \eta^2 + K)^{1/2}\sqrt{\min \{ n, K \}} \) with \( \eta = \| \Umat \|_2 \) and \( \sigma_1 \) is the largest singular value of \( \Xmat \).
\end{theorem}

\begin{theorem}[Operator error bounds: Reformulation paradigms]\label{thm:operator-error-reformulation}
    Under the setup of Theorem~\ref{thm:operator-error-projection}, define the right singular vector error \( \tau_w = \| \Wmat_r - \ot{\Wmat}_r \|_F \). Then for the iSVD-LS/RLS paradigms, the data matrix perturbations satisfy
    \begin{equation}\label{eqn:thm:perturbation-reformulation}
        \| \Emat_D \|_2 \leq \frac{\sigma_1\beta_2}{\sqrt{r}(\sigma_1+\varepsilon)}\tau_w + \frac{\beta_2}{\sigma_1+\varepsilon}\varepsilon, \qquad
        \| \Emat_R \|_F \leq  \sigma_1\| \Deltamat_K \|_2 \tau_w + \sqrt{r}\| \Deltamat_K \|_2 \varepsilon,
    \end{equation}
    and the operator error satisfies
    \begin{equation}\label{eqn:thm:operator-error-reformulation}
        \| \Omat - \ot{\Omat} \|_F \leq 
        \frac{\sigma_1\| \Deltamat_K \|_2}{\sqrt{\gamma_{\min}}} \tau_w\Bigl(1 + \frac{\alpha\beta_2}{\sqrt{\gamma_{\min}}}\Bigr) + 
        \frac{(\alpha\beta_2 + \sqrt{\gamma_{\min}})\sqrt{r}\| \Deltamat_K \|_2}{\gamma_{\min}} \varepsilon,
    \end{equation}
    where  \( \beta_2  = \sqrt{r}(\sigma_1 + \varepsilon)\sqrt{1 + (2\sigma_1 + \varepsilon)^2} \) and \( \sigma_1 \) is the largest singular value of \( \Xmat \).
\end{theorem}

These theoretical results yield two key observations. First, when \( \tau_w > \tau_v \), as is the case for \gls*{isvd} algorithms satisfying~\eqref{eqn:isvd-singular-vector-errors}, the iSVD-Project-LS/RLS paradigms can yield lower operator error than iSVD-LS/RLS\@. This follows because the bound~\eqref{eqn:thm:operator-error-reformulation} depends on \( \tau_w \) and includes additional terms in \( \varepsilon \), whereas the bound~\eqref{eqn:thm:operator-error-projection} depends only on \( \tau_v \). Thus, projection paradigms are preferable when operator accuracy is critical and the additional computational cost and data passes are acceptable.

Second, when the spectral decay is slow (e.g., the turbulent channel flow example in \Cref{sec:exp:channel}), the errors \( \tau_w \), \( \tau_v \), and \( \varepsilon \) increase due to an insufficient reduced rank \( r \) to capture the dominant modes. In such cases, increasing the regularization parameter \( \gamma_{\min} \) can mitigate operator error by reducing sensitivity to these perturbations. 

\subsubsection{Errors from Recursive Least Squares}\label{sec:stream:error:rls}

When using \gls*{rls} instead of batch \gls*{ls}, an additional source of error arises from the recursive updates. We define the \emph{streaming operator error} (SOE) as the difference between the \gls*{rls} solution \( \Omat_k \) at time step \( k \) and the batch \gls*{ls} solution \( \Omat \) obtained using all available data:
\begin{equation}\label{eqn:streaming-error}
    \mathrm{SOE}(k) = \Omat - \Omat_k = \mathrm{SOE}(k-1) - \gvec_k \, \boldsymbol{\xi}_k^-,
\end{equation}
following a recursive update analogous to~\eqref{eqn:rls-update}. Assessing this error is important when considering \gls*{rls} over batch \gls*{ls}, as it reveals the accuracy trade-offs introduced by the recursive updates.

The convergence of \( \mathrm{SOE}(k) \to 0 \) as \( k \to K \) indicates that the \gls*{rls} algorithm effectively approximates the batch solution as more data is processed. In noise-free scenarios, \( \mathrm{SOE}(k) \) converges to zero exactly. In the presence of noise, \( \mathrm{SOE}(k) \) may not vanish but should decrease in a mean-square sense, reflecting the algorithm's ability to learn the underlying operators despite measurement noise~\cite[p.\,21]{ljung1987Theory}.

Overall, the total operator error in the iSVD-LS and iSVD-Project-LS frameworks is due solely to errors from streaming \gls*{svd} defined in~\Cref{thm:operator-error-projection,thm:operator-error-reformulation}. In contrast, the total operator error in iSVD-RLS and iSVD-Project-RLS frameworks is a combination of operator errors in~\Cref{thm:operator-error-projection,thm:operator-error-reformulation} and errors from \gls*{rls} updates~\eqref{eqn:streaming-error}, in which the \gls*{rls} solution approaches \( \ot\Omat \) solved using matrices from streaming \gls*{svd} with batch \gls*{ls}. In~\Cref{sec:experiments}, we assess the streaming operator error to evaluate the convergence behavior of the \gls*{rls} algorithms in practice.

\subsection{Practical Considerations}\label{sec:stream:practical}

The operator inference problem is formulated as a linear \gls*{ls} problem where conditioning critically affects operator quality and stability. Linear and polynomial state dependencies due to small time step \(\delta t\) yield high condition numbers, compounded by noise from numerical time-derivative approximation, model misspecification, and unresolved dynamics from \gls*{pod} truncation~\cite{geelen2022Localized}.

To improve conditioning, we recommend preprocessing of snapshot data: mean-centering ensures the \gls*{pod} basis captures dominant variation modes~\cite{taira2017modal}, while min-max normalization (e.g., scaling to \([0, 1]\) or \([-1, 1]\)) prevents bias toward larger-magnitude variables in multi-physics systems~\cite{mcquarrie2021Datadriven,kramer2024Learning}. This is crucial for turbulent channel flow (Section~\ref{sec:exp:channel}), where streamwise velocity dominates. When making predictions, these transformations are inverted to return solutions in physical coordinates. Computing mean and min-max values requires an additional data pass, efficiently handled via batch processing and parallelization.

The initialization of \( \Pmat_0 \) (see~\Cref{alg:rls,alg:iqrrls}) avoids impractical initialization of \( \Dmat_k^\top\Dmat_k \) as \(\infty \cdot \Id_d \) at \( k = 0 \), which suppresses the solution to zero~\cite[p.\,495]{sayed2011Adaptive}. This initialization also provides regularization, with the common choice \( \Pmat_0 = \gamma^{-1} \Id_d \) for small \( \gamma > 0 \) corresponding to Tikhonov regularization~\cite[p.\,20]{ljung1987Theory}. We use \( \gamma = 1\times 10^{-9} \), which yielded stable models in our benchmark experiments. Since \gls*{rls} algorithms suffer from ill-conditioning of the inverse correlation matrix with many rank-1 updates (or equivalently, small window sizes), this regularization is crucial for numerical stability~\cite{stotsky2025recursive}.

For large-scale problems, optimal regularization balancing training error and long-time stability is critical. Following~\cite{mcquarrie2021Datadriven}, we employ \( \Gammamat = \mathrm{diag}(\gamma_1 \Id_{d_1}, \gamma_2 \Id_{d_2}) \), where \( \gamma_1 > 0 \) acts on linear and constant terms and \( \gamma_2 > 0 \) on quadratic terms, with \( d_1 = r + m + 1 \), \( d_2 = r^2 \), and \( d = d_1 + d_2 \). We select \( \gamma_1, \gamma_2 \) via grid search, evaluating state prediction error of the inferred model over a validation dataset for each parameter pair.
\section{Numerical Experiments}\label{sec:experiments}

In this section, we present numerical experiments to evaluate the performance of the proposed Streaming OpInf method across different Streaming \gls*{opinf} paradigms as defined in~\Cref{sec:stream:selection}. We consider three distinct problems: (i) one-dimensional viscous Burgers' Equation (\Cref{sec:exp:burgers}), (ii) \gls*{kse} (\Cref{sec:exp:kse}), and (iii) three-dimensional turbulent channel flow (\Cref{sec:exp:channel}).

The first two experiments are conducted on a desktop machine with 128GB of memory. The iSVD phase of the channel flow problem was performed on the National Laboratory of the Rockies (NLR) high-performance computing cluster, Kestrel, using a single node with 300GB of allocated memory, while the LS phase was performed on the desktop machine. The experiment codes are publicly available for reproducibility\footnote{Experiment code: \href{https://github.com/smallpondtom/StreamingOpInf}{ https://github.com/smallpondtom/StreamingOpInf}}.

\subsection{Viscous Burgers' Equation}\label{sec:exp:burgers}

Our first numerical experiment uses a benchmark model reduction problem based on the one-dimensional viscous Burgers' equation, a simplified fluid dynamics model describing the motion of a one-dimensional viscous fluid~\cite{burgers1948mathematical,peherstorferW2016,peherstorfer2020Sampling,koike2024Energy}. Despite this problem being a small example which can be handled by standard batch \gls*{opinf}, it serves as a useful test case to validate the proposed method and analyze its performance with different paradigms and algorithmic choices in Streaming \gls*{opinf}. We present the problem setup, evaluation metrics, and results with discussions in~\Cref{sec:exp:burgers:setup,sec:exp:burgers:metrics,sec:exp:burgers:results}, respectively.

\subsubsection{Problem Setup}\label{sec:exp:burgers:setup}

The governing \gls*{pde} is
\begin{equation}\label{eqn:burgers}
    \frac{\partial }{\partial t}x(\omega,t) = \mu \frac{\partial^2 }{\partial \omega^2}x(\omega, t) - x(\omega,t)\frac{\partial }{\partial \omega}x(\omega,t),
\end{equation}
where \( x(\omega,t) \) represents the fluid velocity field, \( t \) denotes time, \( \omega \) is the spatial coordinate, and \( \mu > 0 \) is the kinematic viscosity. This equation emerges from the Navier-Stokes equations through specific simplifications: restriction to one-dimensional flow and omission of pressure gradient effects. The first term represents viscous diffusion, while the second term captures nonlinear advective effects.
Unlike the non-parametric formulation in~\cref{eqn:poly-sys}, we consider a parametric model reduction problem from the setup in~\cite{peherstorferW2016} where the viscosity parameter \( \mu \) varies across a prescribed range. After spatial discretization, the resulting full model takes the form:
\begin{equation*}
    \xdot(t;\mu) = \Amat_1(\mu)\xvec(t;\mu) + \Amat_2(\mu)(\xvec(t;\mu) \otimes \xvec(t;\mu)) + \Bmat(\mu)u(t),
\end{equation*}
where \( \xvec(t;\mu) \in \mathbb{R}^n \) represents the state vector, \( \Amat_1(\mu) \in \mathbb{R}^{n \times n} \) is the linear operator matrix, \( \Amat_2(\mu) \in \mathbb{R}^{n \times n^2} \) is the quadratic operator matrix, and \( \Bmat(\mu) \in \mathbb{R}^{n \times 1} \) is the input operator matrix for a specific viscosity parameter \( \mu \).

We consider the spatial domain \( \omega \in [0,1] \) with Dirichlet boundary conditions controlled by the input: \( x(0,t) = u(t) \) and \( x(1,t) = -u(t) \). The initial condition is specified as \( x(\omega,0) = 0.1\sin(2\pi\omega) \), and the simulation time extends from \( t = 0 \) to \( t = 1.0 \). The viscosity parameter is set to \( \mu_i \in \{0.1, 0.2, \ldots, 1.0\} \) with \( i = 1, \ldots, M \) and \( M = 10 \) total parameter values. Spatial discretization employs \( n = 128 \) grid points using second-order central differences for the diffusion term and leap-frog scheme for the advection term. Time integration is performed using the semi-implicit Euler method with time step \( \delta t = 1 \times 10^{-4} \), where the linear diffusion term is treated implicitly and the nonlinear advection and control terms are handled explicitly.

For training data generation, we sample 10 parameter values \( \mu_i \) from the parameter space. At each parameter value, we generate 10 independent trajectories with control inputs \( u(t) \) drawn uniformly from \( [0,1] \). Pooling the time-series data from these trajectories yields \( K = 10,001 \) snapshots per parameter value (including initial conditions). This results in a total training dataset of \( 10K = 100,010 \) snapshots across all parameter values. For the \gls*{isvd} phase, we compute the \gls*{svd} for the entire parameter space comprising these \(10K\) snapshots. We investigate reduced dimensions \( r \in \{1, 2, \ldots, 14\} \). To benchmark the Streaming OpInf method, we consider iSVD-LS and iSVD-RLS paradigms, where the time derivative data is approximated using first-order finite differences.

Training models are evaluated using a constant reference input \( u(t) = 1.0 \) for \( t \in [0,1] \). We learn one set of reduced operators per training parameter value \( \mu_i \). For new parameter values within the training range, we obtain the corresponding operators via element-wise B-spline interpolation (degree three) of the learned operator coefficients~\cite{degroote2010Interpolation}. For generalization testing, we generate five new sets of reduced operators with five parameter values uniformly sampled within \( [0.1, 1.0] \) and evaluate them on the same temporal grid with reference input \( u(t) = 1.0 \). We compute the \gls*{svd} using both Baker's \gls*{isvd} and SketchySVD methods for comparison, with the \gls*{svd} components from SketchySVD used in the subsequent \gls*{ls} phase due to better performance in the \gls*{isvd} evaluations shown in the subsequent sections. Note that in the \gls*{ls} phase, regularization is \( \gamma = 1\times 10^{-9} \) in both batch and streaming methods.

\subsubsection{Evaluation Metrics}\label{sec:exp:burgers:metrics}

\paragraph{iSVD Method Assessment}We measure the alignment between \gls*{pod} bases computed from full \gls*{svd} and \gls*{isvd} methods using the subspace angle error metric:
\begin{equation}\label{eqn:subspace-angles}
    \text{Subspace Angle Error} = \frac{1}{\sqrt{2}}\|\Vmat_{\text{batch}}\Vmat_{\text{batch}}^\top - \Vmat_{\text{isvd}}\Vmat_{\text{isvd}}^\top \|_F = \|{(\mathbf{I} - \cos\theta)}^{1/2}\|_F,
\end{equation}
where \( \Vmat_{\text{batch}} \in \mathbb{R}^{n\times r} \) is the \gls*{pod} basis from full \gls*{svd}, \( \Vmat_{\text{isvd}} \in \mathbb{R}^{n \times r} \) is the \gls*{pod} basis from \gls*{isvd}, and \( \cos\theta = \mathsf{diag}(\sigma(\Vmat_{\text{batch}}^\top\Vmat_{\text{isvd}})) \) with \( \sigma(\cdot) \) representing the singular values. Smaller values indicate better subspace alignment~\cite{bjorck1973Numerical}. 
Additionally, we assess the approximation quality solely based on the POD basis using the relative projection error:
\begin{equation}\label{eqn:rpe}
    \text{Relative Projection Error} = \frac{\|\Xmat - \Vmat_r\Vmat_r^\top\Xmat \|_F}{\| \Xmat \|_F},
\end{equation}
where \( \Xmat \) is the snapshot matrix and \( \Vmat_r \) is the rank-\( r \) \gls*{pod} basis.

\paragraph{Streaming Operator Learning Assessment}We monitor the operator convergence through the mean relative streaming operator error (MR-SOE) at stream \( k \):
\begin{equation}\label{eqn:mrsoe}
    \text{MR-SOE}(k,r) = \frac{1}{M} \sum_{i=1}^M \frac{\|\Omat(\mu_i) - \Omat_k(\mu_i)\|_F}{dr\| \Omat(\mu_i) \|_F},
\end{equation}
where this quantity is averaged over all parameters, \( \Omat_k(\mu_i) \) contains all reduced operators obtained via RLS at stream \( k \) with \gls*{pod} basis of rank-\( r \), and the normalization factor \( dr \) accounts for the total number of operator elements. As discussed in~\Cref{sec:stream:error}, this metric provides insight into the convergence behavior of streaming operator estimates. 

We further evaluate the streaming reconstruction accuracy using the mean relative streaming state error (MR-SSE) at stream \( k \):
\begin{equation}\label{eqn:mrse}
    \text{MR-SSE}(k,r) = \frac{1}{M} \sum_{i=1}^M \frac{\|\Xmat(\mu_i) - \Vmat_r \widetilde{\Xmat}_k(\mu_i) \|_F}{\|\Xmat(\mu_i)\|_F},
\end{equation}
where \( \widetilde{\Xmat}_k(\mu_i) \) is the state trajectory simulated using the reduced model obtained at stream \( k \) for parameter \( \mu_i \). We track MR-SOE and MR-SSE at each data stream for different reduced dimensions \( r \), reporting results for both \gls*{rls} and \gls*{iqrrls} algorithms. 
We additionally present the final reconstruction accuracy using the final relative state error (Final RSE) over different reduced dimensions \( r \) after processing all data streams, which is equivalent to \( \text{MR-SSE}(K,r) \).
We compute the final RSEs for both training and testing datasets, comparing streaming models against baselines: intrusive POD and standard \gls*{opinf} with Tikhonov regularization. 

\begin{figure}[t!]
    \centering
    \includegraphics[width=0.49\textwidth]{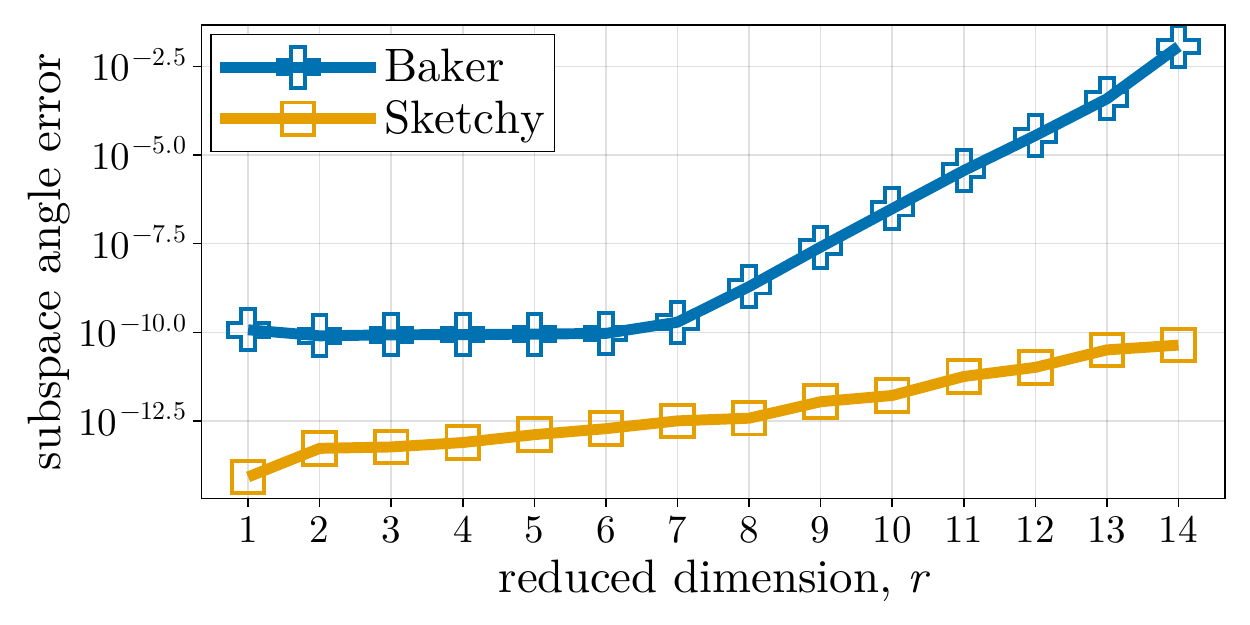} 
    \includegraphics[width=0.49\textwidth]{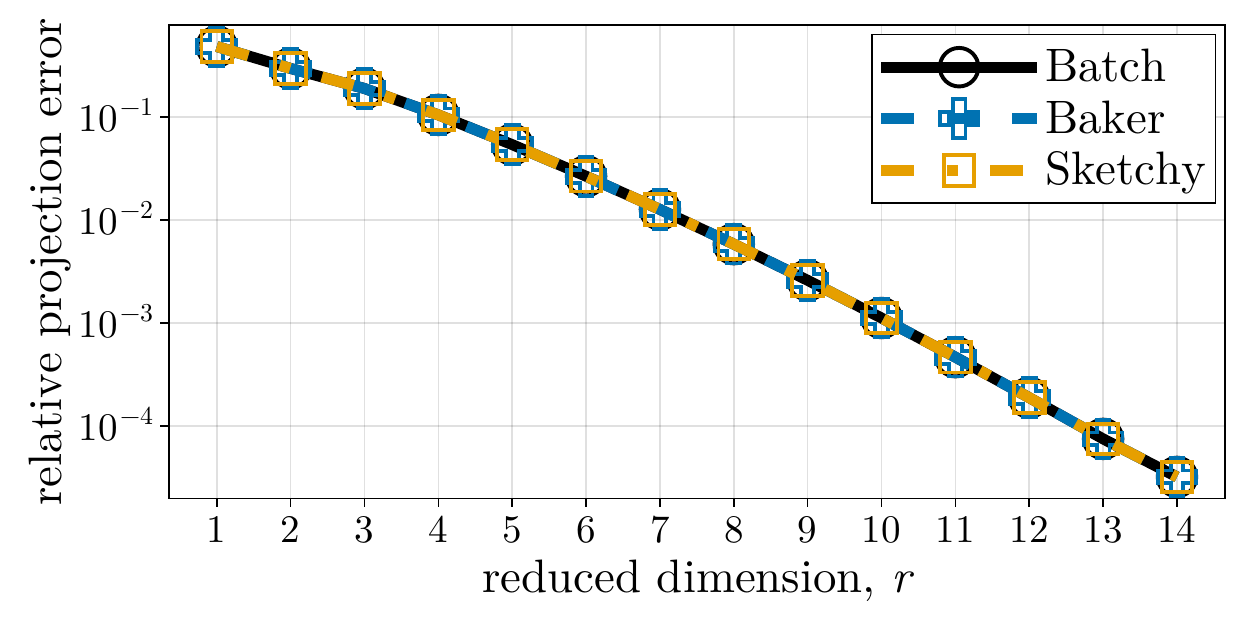}
    \caption{Burgers' Streaming SVD Assessment.\ \textbf{Left:} Subspace angles between POD bases computed from full SVD and iSVD methods.\ \textbf{Right:} Relative projection errors of the snapshot matrix onto POD bases for all SVD methods.}\label{fig:burgers_isvd} 
\end{figure}

\begin{figure}[t!]
    \centering
    \includegraphics[width=0.93\textwidth]{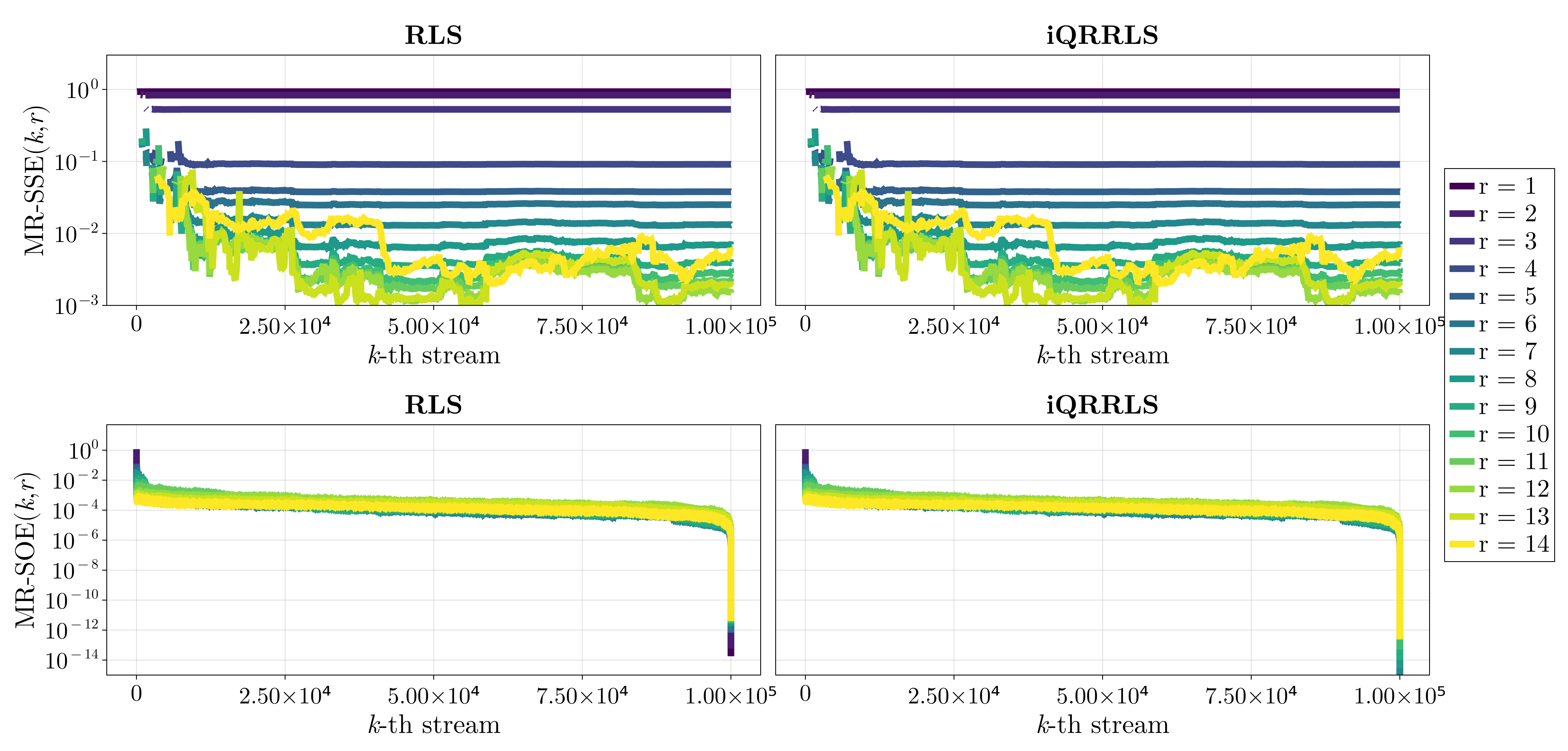}
    \vspace{-1.0em}
    \caption{Burgers' RLS Assessment.\ \textbf{Top:} Relative state error averaged over all parameters at each stream.\ \textbf{Bottom:} Relative streaming operator error normalized by the total number of operator elements, averaged over all parameters at each stream. Gradient lines represent different reduced dimensions \( r \).}\label{fig:burgers_rls}
\end{figure}

\subsubsection{Results \& Discussion}\label{sec:exp:burgers:results}

\Cref{fig:burgers_isvd} compares iSVD methods against full SVD, showing subspace angle errors between the streaming \gls*{svd} methods and the batch \gls*{svd} (left) and relative projection errors (right). SketchySVD exhibits superior subspace alignment across all reduced dimensions \( r \), particularly for higher-dimensional subspaces where smaller singular values are included. This aligns with theoretical predictions~\cite{fareed2020Error}: Baker's iSVD is more sensitive to spectral gap closure, while SketchySVD maintains lower errors through a tight error bound determined by a sufficiently large sketch size, as discussed in~\Cref{sec:bg:isvd}. The gradual increase in subspace angles reflects higher relative errors in smaller singular values. Moreover, both iSVD methods maintain projection errors comparable to batch \gls*{svd} showing their reliability as alternatives for standard \gls*{svd}.

\begin{figure}[t!]
    \centering
    \includegraphics[width=0.45\textwidth]{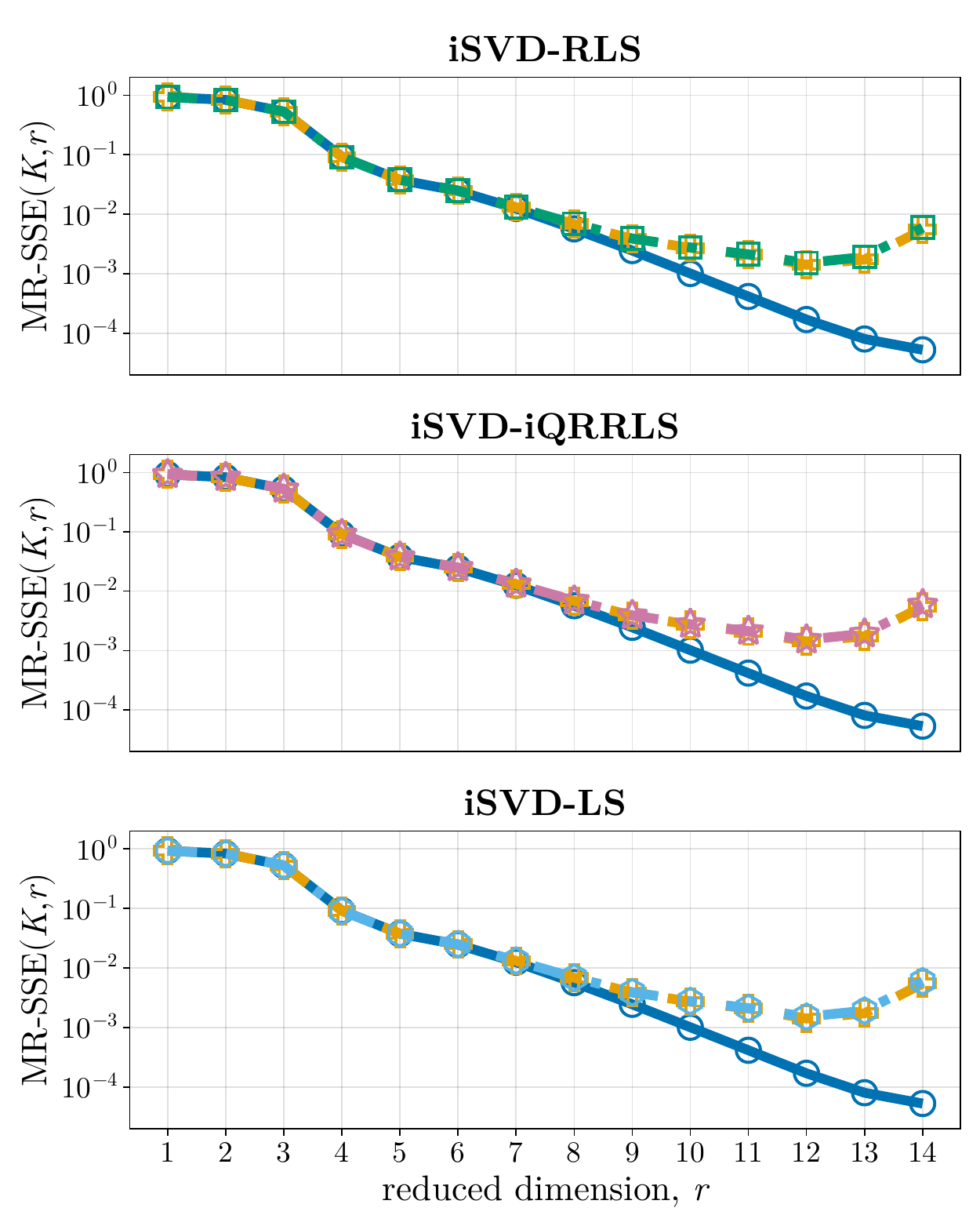}
    \includegraphics[width=0.45\textwidth]{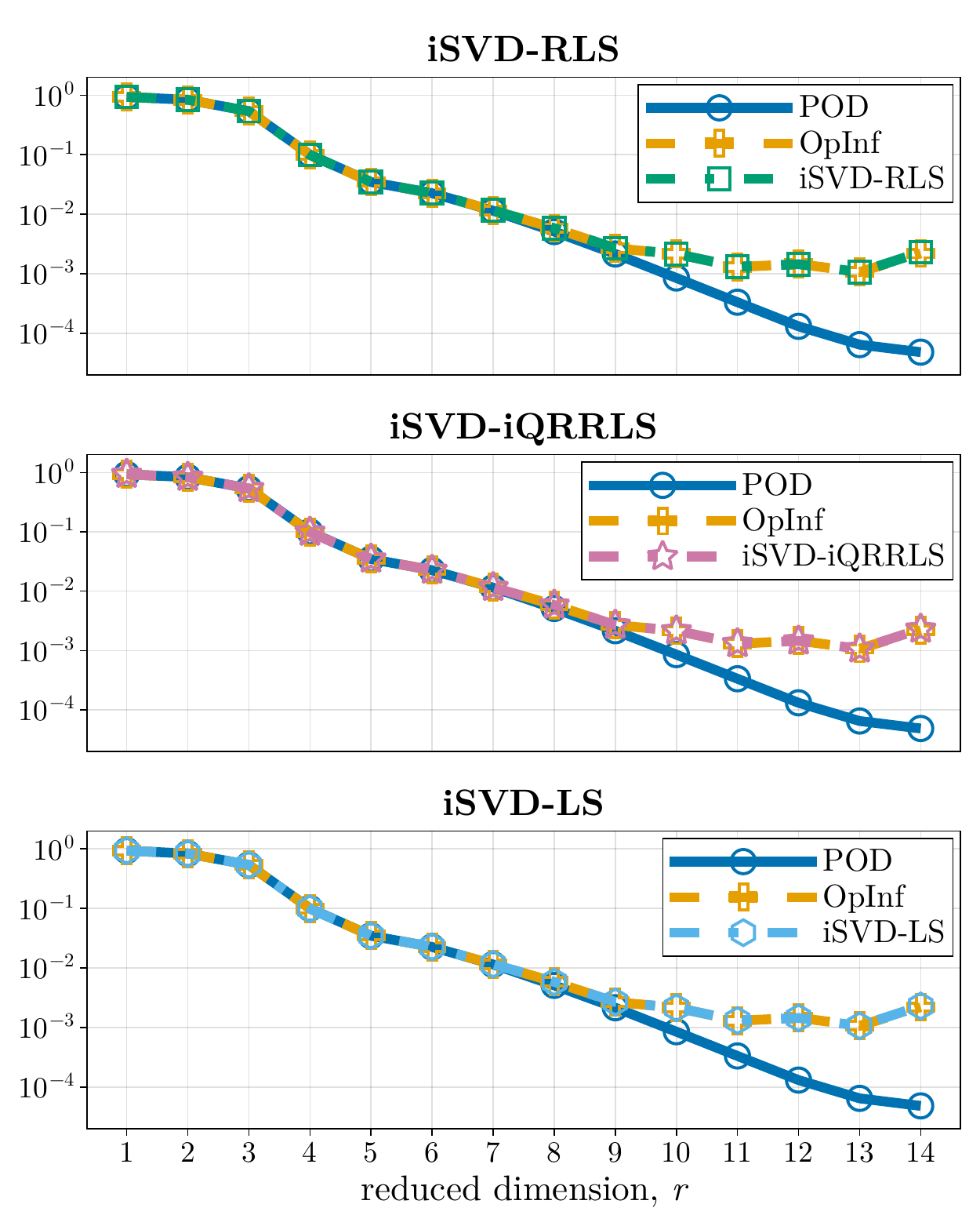}
    \caption{\textbf{Left Column:} Final relative state errors for training data across all reduced dimensions \( r \) for Burgers' equation.\ \textbf{Right Column:} Final relative state errors for testing data across all reduced dimensions \( r \).}\label{fig:burgers_final} 
\end{figure}

\Cref{fig:burgers_rls} illustrates the evolution of relative state errors (top) and streaming operator errors (bottom) throughout the streaming process. State errors decrease with fluctuations from incremental updates, ultimately plateauing for all \( r \), confirming that \gls*{rls} effectively refines the operator solution with each stream. Operator errors decrease sharply initially, plateau during the middle phase, then drop dramatically to near machine precision at the final stream when the update incorporates the full data. Higher \( r \) yields lower final state errors due to greater energy retention, while operator errors remain within the same order of magnitude across dimensions. Notably, state errors converge much earlier than operator errors reach machine precision, suggesting that approximate operator estimates can still yield accurate state reconstructions.

\Cref{fig:burgers_final} presents the final relative state errors for both training (left) and testing data (right) across all reduced dimensions \( r \) and Streaming \gls*{opinf} paradigms. We observe that the final relative state errors for both training and testing data decrease monotonically with increasing reduced dimension \( r \) up to \( r \approx 10 \), beyond which closure errors---arising from unmodeled dynamics in the reduced space---cause the error to level out, whereas the intrusive POD method continues to decrease beyond this point~\cite{peherstorferW2016}. Moreover, the testing results indicate successful generalization of the Streaming \gls*{opinf} models within the trained parameter range. The slight error increase observed at \( r = 13, 14 \) for less important \gls*{pod} modes reflects numerical ill-conditioning in the \gls*{ls} problem caused by the abundance of steady-state data at the end of trajectories for higher viscosity parameters, consistent with observations in~\cite{peherstorferW2016}. Nevertheless, as shown in~\Cref{tab:computational-cost}, which shows memory costs and savings for all experiments, we see that Streaming \gls*{opinf} reduces total memory usage by over 99\% while obtaining comparable accuracy to the standard \gls*{opinf} for this example. 

These results validate that Streaming \gls*{opinf} on a benchmark problem, achieving accuracy comparable to batch OpInf while providing substantial memory efficiency and scalability for streaming applications.

\subsection{Kuramoto-Sivashinsky Equation}\label{sec:exp:kse}

Next, we examine the Kuramoto-Sivashinsky Equation (KSE), a canonical example of spatiotemporal chaos originally developed to model flame propagation~\cite{kuramoto1978diffusion,sivashinksy1977nonlinear}. Although this small-scale problem does not require data streaming, \gls*{kse}'s rich nonlinear dynamics make it a popular benchmark for model reduction~\cite{lu2017data,almeida2022Nonintrusive,koike2024Energy} and a useful test case for evaluating Streaming OpInf on chaotic systems before applying it to larger-scale turbulent problems. We present the problem setup, evaluation metrics, and results with discussions in~\Cref{sec:exp:kse:setup,sec:exp:kse:metrics,sec:exp:kse:results}, respectively.

\subsubsection{Problem Setup}\label{sec:exp:kse:setup}

The governing partial differential equation is
\begin{equation}\label{eqn:kse}
    \frac{\partial }{\partial t}x(\omega,t) + \frac{\partial^2 }{\partial \omega^2}x(\omega, t) + \mu \frac{\partial^4 }{\partial \omega^4}x(\omega, t) + x(\omega,t)\frac{\partial }{\partial \omega}x(\omega,t) = 0,
\end{equation}
where \( x(\omega,t) \) is the flame front position, \( \omega \) is the spatial coordinate, and \( \mu > 0 \) is the kinematic viscosity. \Gls{kse} incorporates temporal evolution, destabilizing second-order diffusion, stabilizing fourth-order hyperdiffusion, and nonlinear convection, with these competing effects producing complex spatiotemporal dynamics~\cite{sprott2010elegant,szezech2007onset}.

For this investigation, we focus on a single parameter value \( \mu = 1 \) with no external control input. After spatial discretization, the resulting model takes the quadratic form:
\begin{equation*}
    \xdot(t) = \Amat_1\xvec(t) + \Amat_2(\xvec(t) \otimes \xvec(t)),
\end{equation*}
where \( \xvec(t) \in \mathbb{R}^n \) is the state vector, \( \Amat_1 \in \mathbb{R}^{n \times n} \) and \( \Amat_2 \in \mathbb{R}^{n \times n^2} \) are as defined in~\eqref{eqn:poly-sys}. 

Following the computational setup from~\cite{koike2024Energy}, we define the spatial domain as \( \omega \in [0, L] \) with \( L = 22 \) and periodic boundary conditions. The domain length \( L = 22 \) is chosen to ensure the system exhibits chaotic dynamics while maintaining computational tractability~\cite{cvitanovic2010State}. The initial condition is specified as \( x(\omega,0) = a\cos(2\pi\omega/L) + b\cos(4\pi\omega/L) \), where \( a \in \{0.2, 0.7, 1.2\} \) and \( b \in \{0.1, 0.5, 0.9\} \) are amplitude parameters that control the initial perturbation strength.

The spatial discretization employs \( n = 512 \) uniformly distributed grid points, providing adequate resolution to capture the fine-scale structures characteristic of the KSE dynamics. Time integration is performed using a semi-implicit scheme that combines the second-order Crank-Nicolson method for the linear terms (second and fourth-order derivatives) with the second-order Adams-Bashforth method for the nonlinear convection term. This hybrid approach ensures numerical stability while maintaining temporal accuracy. The time step is set to \( \delta t = 1 \times 10^{-3} \), and simulations extend from \( t = 0 \) to \( t = 300 \), capturing both transient and fully developed chaotic regimes.

The training dataset comprises 9 trajectories generated from all combinations of the amplitude parameters \( a \) and \( b \). Time derivative data is approximated using first-order finite differences applied to the solution snapshots. To reduce computational burden while preserving essential dynamics, both snapshot and time derivative data are downsampled by retaining every 100th time step, resulting in \( K = 3,000 \) snapshots per trajectory and a total of \( 9K = 27,000 \) snapshots across all trajectories. For the \gls*{isvd} phase, we compute the \gls*{svd} for the entire dataset of \( 9K \) snapshots. We investigate reduced dimensions \( r \in \{9, 12, 15, 18, 21, 24\} \). 

For generalization assessment, testing data is generated using 50 additional combinations of amplitude parameters \( (a,b) \) drawn uniformly from the continuous ranges \( a \in [0.2, 1.2] \) and \( b \in [0.1, 0.9] \). These testing models are evaluated on the same temporal grid as the training data. We compute the \gls*{svd} using both Baker's \gls*{isvd} and SketchySVD methods for comparison, with the \gls*{svd} components from Baker's \gls*{isvd} deliberately chosen for the subsequent \gls*{ls} phase to explore the impact of different iSVD methods on final model performance.

In this example, we employ both the standard \gls*{rls} and \gls*{iqrrls} algorithms to solve the \gls*{ls} problem at each data stream. The Tikhonov regularization parameter is set to \( \Gammamat = \gamma\mathbf{I}_d \) with \( \gamma = 1 \times 10^{-9} \) for both streaming and batch Tikhonov-regularized OpInf methods, providing minimal regularization while ensuring numerical stability.

Since the provided data includes both downsampled snapshot data and time-derivative data, using finite difference to approximate the time derivative will introduce significant errors in the operator inference process. Hence, this experiment specifically applies the iSVD-Projection-RLS paradigm in~\Cref{sec:stream:selection}, where both downsampled snapshot data and time derivative data are directly projected onto the POD basis to form the reduced data.

\subsubsection{Evaluation Metrics}\label{sec:exp:kse:metrics}
For \gls*{isvd} assessment, we measure subspace alignment using subspace angles and approximation quality using relative projection errors, as in~\Cref{sec:exp:burgers:metrics}. For streaming operator learning, we track only the MR-SOE at each stream with \( M = 1 \). Unlike the Burgers' equation case, we omit state error metrics because KSE's chaotic dynamics cause exponentially diverging trajectories from small perturbations, making point-wise error metrics less meaningful. Hence, it is crucial that we capture the statistical properties and attractor geometry of the chaotic system using \glspl*{qoi} that show dynamical invariants, i.e., invariants to initial conditions and individual trajectories. This is also true for the large-scale turbulent flow example considered in~\Cref{sec:exp:channel}.

The first \gls*{qoi} we compute is the leading 10 Lyapunov exponents (LEs) \( (\lambda_1, \lambda_2, \ldots, \lambda_{10}) \), which quantify the average exponential rate of trajectory divergence or convergence in phase space~\cite{wolf1985determining}, with positive LEs indicating chaotic directions and negative LEs representing stable directions. LEs are computed using the algorithm from~\cite{kuptsov2012theory,sandri1996numerical}, averaging over all trajectories for training and testing datasets and comparing against the full model and batch methods.

The second \gls*{qoi} is the Kaplan-Yorke dimension (KY-dimension) computed from the Lyapunov spectrum, which estimates the fractal dimension of the chaotic attractor~\cite{kaplan2006chaotic,chlouverakis2005Comparison}: \( D_{KY} = j + \frac{\sum_{i=1}^{j} \lambda_i}{|\lambda_{j+1}|}\), where \( j \) is the largest integer such that \( \sum_{i=1}^{j} \lambda_i \geq 0 \). We validate our computed values against reference values in~\cite{edson2019Lyapunov}. Moreover, for qualitative assessment, we present spatiotemporal flow field predictions for representative trajectories from both training and testing datasets.

\subsubsection{Results \& Discussion}\label{sec:exp:kse:results}

\begin{figure}[b!]
    \centering
    \includegraphics[width=0.49\textwidth]{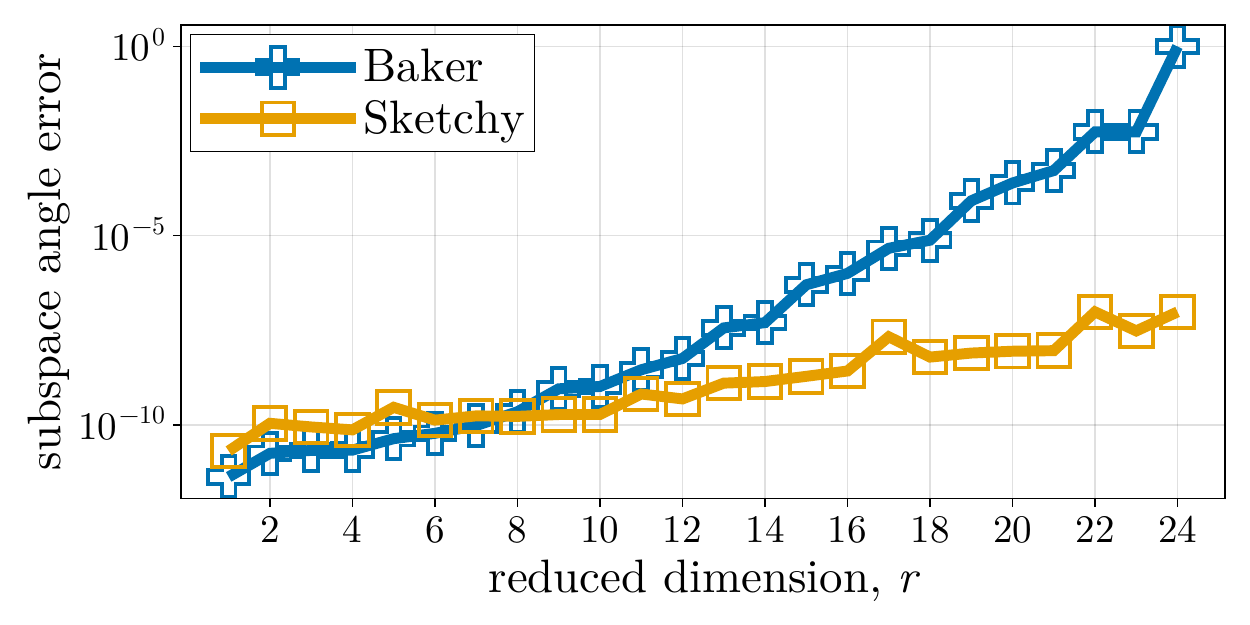} 
    \includegraphics[width=0.49\textwidth]{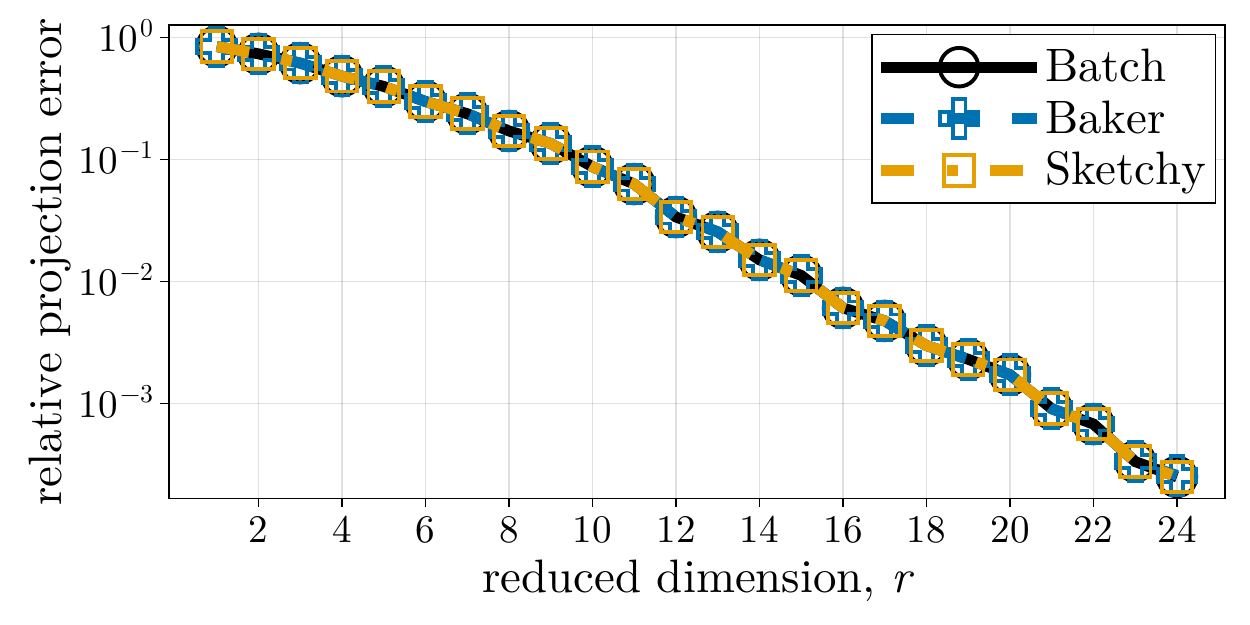}
    \caption{KSE Streaming SVD Assessment.\ \textbf{Left:} Subspace angles between POD bases computed from full SVD and iSVD methods.\ \textbf{Right:} Relative projection errors of the snapshot matrix onto POD bases computed from iSVD methods.}\label{fig:kse_isvd} 
\end{figure}

\begin{figure}[b!]
    \centering
    \includegraphics[width=\textwidth]{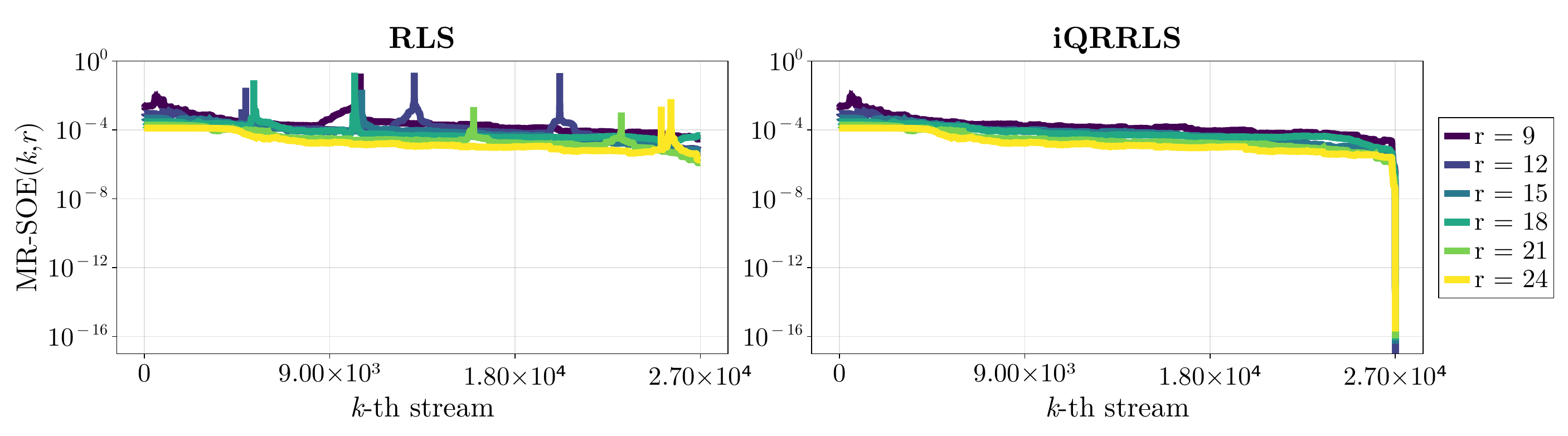}
    \vspace{-1.5em}
    \caption{KSE RLS Assessment. Relative streaming operator error normalized by the total number of operator elements, averaged over all trajectories at each stream. Gradient lines represent different reduced dimensions \( r \).}\label{fig:kse_rls}
\end{figure}

\Cref{fig:kse_isvd} presents the performance comparison between Baker's \gls*{isvd} and SketchySVD methods. The left panel shows subspace angles relative to the full SVD basis, while the right panel displays relative projection errors. Both methods exhibit precise subspace alignment for reduced dimensions up to \( r = 12 \), with Baker's \gls*{isvd} showing substantially larger alignment errors for higher dimensions. This behavior agrees with the results in the Burgers' equation example. The projection error comparison demonstrates that both \gls*{isvd} methods maintain approximation quality comparable to full SVD across all investigated dimensions. As observed in the Burgers' equation experiment, SketchySVD demonstrates superior accuracy in subspace alignment. However, we deliberately employ Baker's iSVD components in the subsequent \gls*{ls} phase to investigate the robustness of the streaming approach to suboptimal basis selection.


\Cref{fig:kse_rls} illustrates the evolution of relative streaming operator errors, revealing distinct behaviors between RLS and iQRRLS algorithms. The iQRRLS algorithm exhibits the characteristic sharp error decrease near final streams, similar to the Burgers' experiment, while standard RLS displays acute error spikes throughout streaming without dramatic final improvement. This difference stems from numerical conditioning, in which iQRRLS maintains superior conditioning through QR factorization-based updates, providing robustness even as \( \Pmat_k \) becomes ill-conditioned, while standard RLS becomes susceptible to numerical instabilities as more data streams are processed. Overall, iQRRLS demonstrates more near machine-precision convergence to the batch \gls*{opinf} solution across all reduced dimensions.

\begin{figure}[t!]
    \centering
    \includegraphics[width=0.9\textwidth]{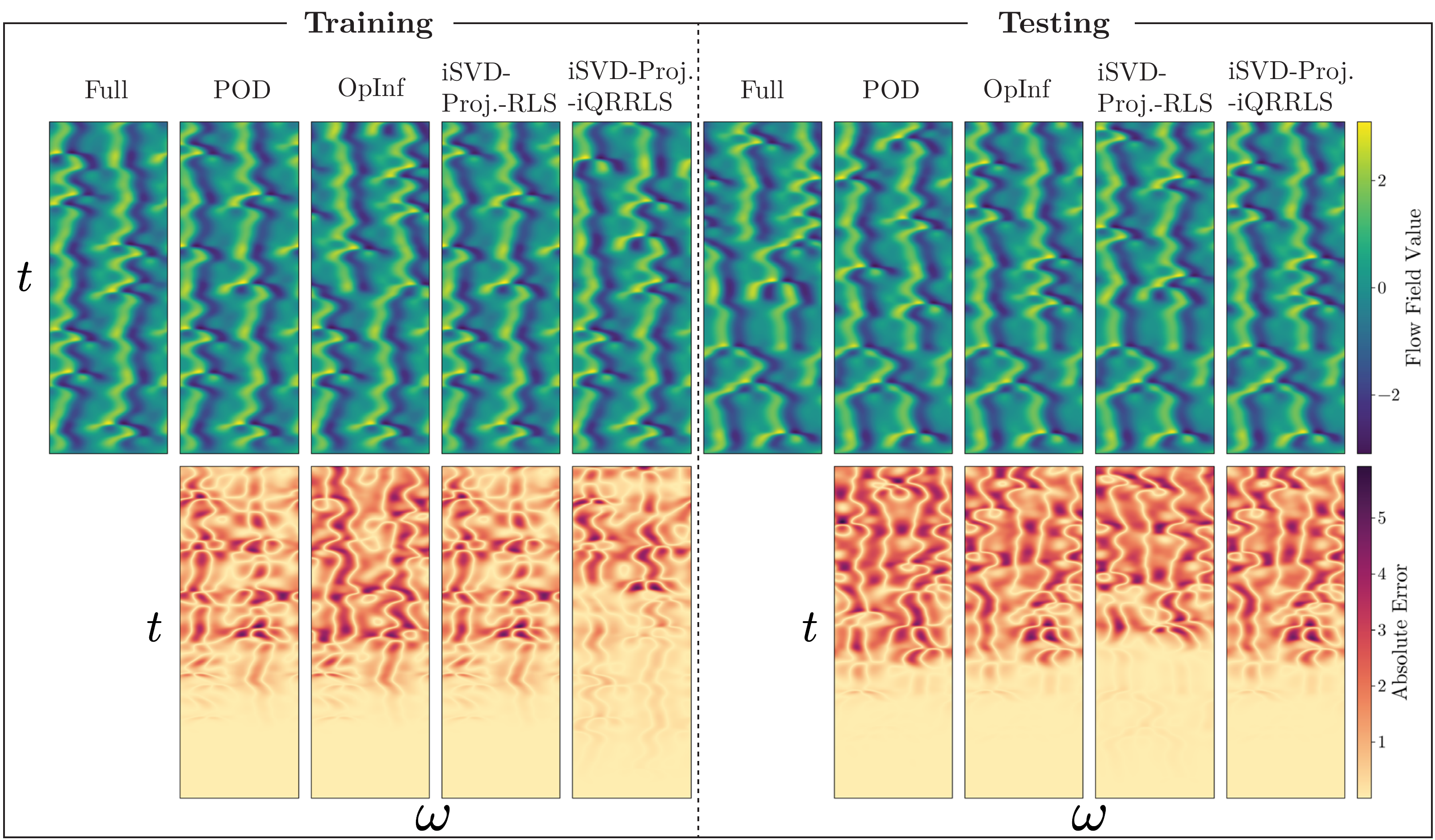}
    \vspace{-0.5em}
    \caption{Comparison of KSE flow field predictions for training (left box) and testing (right box) data using full model and reduced models with \( r = 24 \).\ \textbf{Row 1:} Predictions for training and testing trajectories with initial condition parameters \( (a, b) = (0.2, 0.5) \) and \( (0.6869, 0.4186) \), respectively.\ \textbf{Row 2:} Absolute errors between full model and reduced models. Time evolves from bottom to top.\ ``iSVD-Proj.-RLS'' and ``iSVD-Proj.-iQRRLS'' indicate iSVD-Projection paradigms.}\label{fig:kse_flow}
\end{figure}


\Cref{fig:kse_flow} presents a qualitative comparison of spatiotemporal flow field predictions for representative trajectories from training and testing datasets using \( r = 24 \). The top row displays the characteristic spatiotemporal evolution patterns, while the bottom row shows absolute errors relative to the full model. All learning approaches successfully capture the primary flow features and spatiotemporal patterns of the full model, with error magnitudes comparable to intrusive POD, indicating similar approximation quality. The absolute errors can exceed flow field magnitudes in certain regions, which is characteristic of chaotic systems where trajectories exponentially diverge and lead to phase-shifts. Moreover, error patterns show larger magnitudes in the latter portion of the domain (top half of error plots) compared to the initial transient region, reflecting the challenge of representing fully developed chaotic dynamics where small discrepancies amplify over time. Nevertheless, the bounded errors without unphysical growth confirm that the reduced models remain dynamically consistent.

\Cref{fig:kse_final} assesses the learned models' ability to predict dynamical invariants. The first two columns show the Lyapunov spectrum (first 10 largest exponents) for \( r = 24 \) compared to reference values~\cite{edson2019Lyapunov}. All reduced models successfully capture both positive Lyapunov exponents and negative exponents, which is crucial for maintaining the correct balance between expansion and contraction in phase space. The Kaplan-Yorke dimension analysis (third column) reveals close agreement with reference values~\cite{edson2019Lyapunov} for specific dimensions (\( r = 9, 21, 24 \)), with larger deviations for intermediate values (\( r = 12, 15, 18 \)). This non-monotonic behavior reflects the metric's sensitive dependence on precise Lyapunov exponent values, where small errors in individual exponents propagate when \( \sum_{i=1}^{j} \lambda_i \) approaches zero. Both RLS and iQRRLS variants achieve dynamical invariants comparable to batch methods.

In summary, our results demonstrate that Streaming \gls*{opinf} preserves essential dynamical characteristics of chaotic systems while providing over 99\% memory savings compared to batch methods as shown in~\Cref{tab:computational-cost}. The preservation of Lyapunov exponents and accurate approximation of fractal dimensions confirm that the streaming approach captures the fundamental geometric and dynamical properties of the underlying attractor, making it suitable for applications requiring long-term statistical accuracy.

\begin{figure}[t!]
    \centering
    \includegraphics[width=0.9\textwidth]{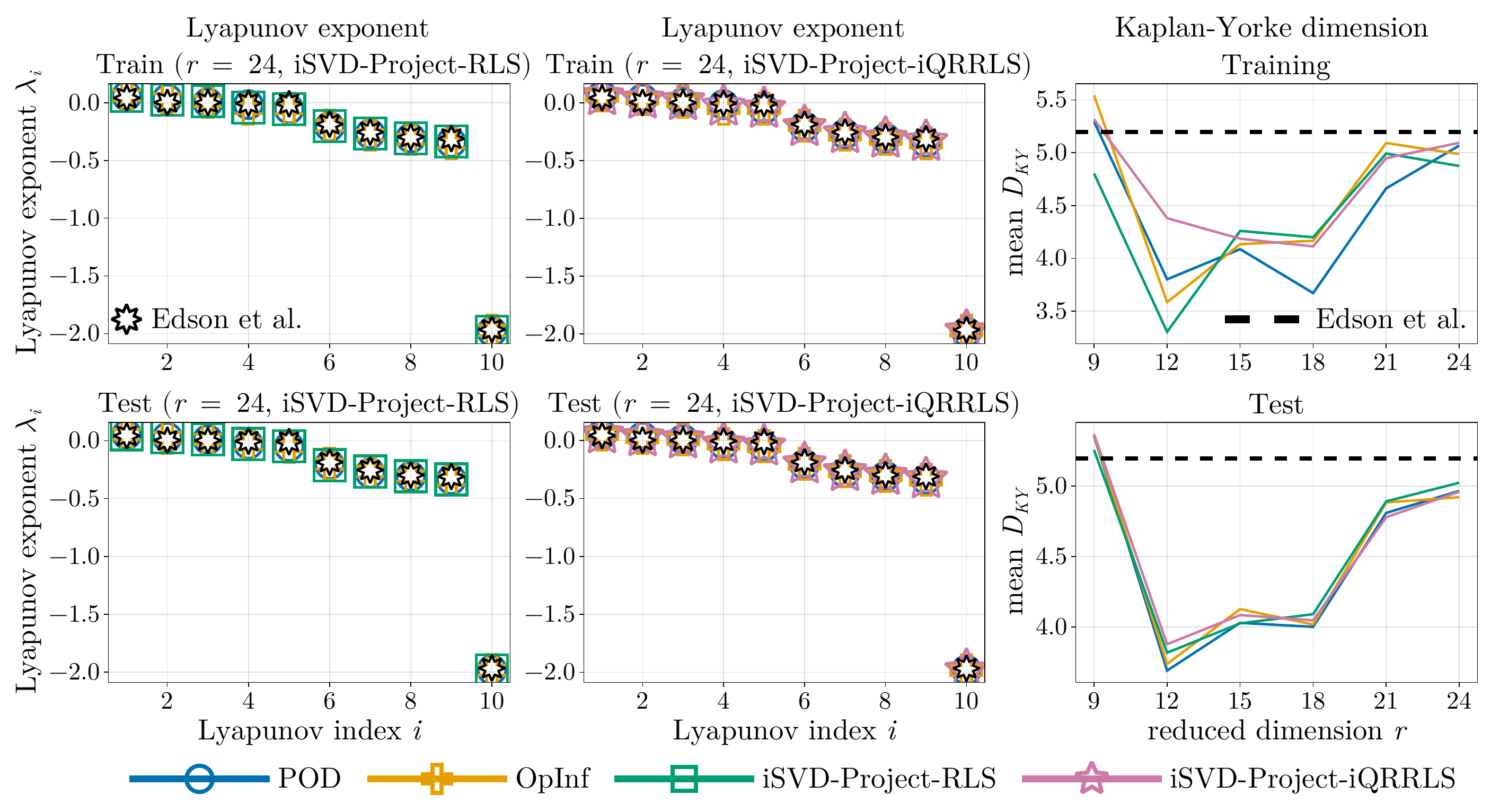}
    \vspace{-0.7em}
    \caption{KSE QoI Assessment.\ \textbf{Columns 1 \& 2:} First 10 largest Lyapunov exponents for training and testing data with \( r = 24 \).\ \textbf{Column 3:} Kaplan-Yorke dimensions for training and testing data across all reduced dimensions. Streaming-OpInf models are ``iSVD-Project-RLS'' and ``iSVD-Project-iQRRLS''. Black markers and lines represent reference LE values and KY-dimensions from~\cite{edson2019Lyapunov}.}\label{fig:kse_final} 
\end{figure}

\subsection{Turbulent Channel Flow}\label{sec:exp:channel}
The final numerical experiment examines three-dimensional turbulent channel flow at high Reynolds number, a problem whose large-scale data requirements exceed available memory and necessitate streaming methods. This canonical wall-bounded turbulent flow is pressure-driven between parallel walls and exhibits complex three-dimensional velocity structures that have been extensively studied to understand turbulence physics near boundaries. We perform simulations of a symmetric half-channel flow using a wall-modeled \gls*{les} formulation described in~\Cref{sec:exp:channel:setup}, with evaluation metrics and results presented in~\Cref{sec:exp:channel:metrics,sec:exp:channel:results}, respectively.

\subsubsection{Problem Setup}\label{sec:exp:channel:setup}
We use AMR-Wind, a parallel, block-structured adaptive-mesh incompressible flow solver, to perform wall-modeled \gls*{les} of the channel flow~\cite{kuhn2025amr,exawind2024}. AMR-Wind solves the \gls*{les} formulation of the incompressible \gls*{ns} equations. With \( (\,\widetilde{\cdot}\,) \) denoting the spatial filtering operator, the governing equations in Cartesian coordinates using Einstein notation are
\begin{align}
    \frac{\partial \ot{u_j}}{\partial \omega_j} & = 0,\label{eqn:ns-les-cont}\\
    \frac{\partial \ot{u_i}}{\partial t} +
    \frac{\partial \ot{u_i} \ot{u_j}}{\partial \omega_j} &=
    - \frac{1}{\rho}{\frac{\partial \ot{p}}{\partial \omega_i}}
    - {\frac{\partial \tau_{ij}}{\partial \omega_j} }
    + \nu \frac{\partial^2 \ot{u_i}}{\partial \omega_j \partial \omega_j}
    + {F_{i}},\label{eqn:ns-les-mom}
\end{align}
where \( \omega_i \) denotes the coordinate in direction \( i = \{1,2,3\} \), \( \ot{u_i} \) is the filtered velocity, \( \ot{p} \) is the pressure, \( \nu \) is the molecular viscosity, \( \rho \) is the density, and \( \tau_{ij} \) are the subgrid stress terms representing interactions with unresolved scales, defined as \( \tau_{ij} = \ot{u_i u_j} - \ot{u_i}\ot{u_j} \).
The forcing term \( F_i \) drives the flow, and each direction \( i \) corresponds to the streamwise, spanwise, and wall-normal directions, respectively.

The Smagorinsky eddy viscosity model is used to model \( \tau_{ij} \), where the dissipation is calculated as:
\begin{equation*}
    \tau_{ij} = -2 \nu_t \widetilde{S}_{ij}, \qquad
    \nu_t = C_s^2 {(2 \widetilde{S}_{ij} \widetilde{S}_{ij})}^{\frac{1}{2}} \Omega^2, \qquad
    \widetilde{S}_{ij} = \frac{1}{2}\left(\frac{\partial \ot{u_i}}{\partial \omega_j} + \frac{\partial \ot{u_j}}{\partial \omega_i}\right),
\end{equation*}
where \( \ot{S}_{ij} \) is the resolved strain rate tensor, \( C_s \) is the Smagorinsky constant set to 0.1, and \( \Omega \) is the filter width defined by \( \Omega = \sqrt[3]{\delta\omega_1 \times \delta\omega_2 \times \delta\omega_3} \) with \( \delta\omega_i \) denoting the grid spacing in direction \(i\). 

AMR-Wind employs discretization schemes that are second-order accurate in both space and time. The spatial discretization is a second-order finite-volume method. Velocity, scalar quantities, and pressure gradients are located at cell centers, whereas pressure is located at nodes. In addition to spatial staggering, the temporal discretization employs staggering similar to a Crank-Nicolson formulation.

The boundary conditions for the half-channel simulation are periodic in the streamwise and spanwise directions. A log-law-based wall model is applied at the bottom boundary, and centerline symmetry conditions are imposed at the top boundary. The flow is driven by a constant streamwise pressure gradient.

Denoting the full state as \( \xvec = [\ot{u_1}, \ot{u_2}, \ot{u_3}, p]\tran \), and recognizing that the \gls*{ns} equations are governed by quadratic nonlinearity, we assume the spatially discretized form of the governing equations is quadratic with a constant term arising from the forcing and from the mean-centering and normalization described in~\Cref{sec:stream:practical}. The \gls*{ode} model thus takes the form:
\begin{equation*}
    \dot{{\xvec}}(t) = {\Amat}_1\xvec(t) + \Amat_2(\xvec(t) \otimes \xvec(t)) + \cvec,
\end{equation*}
where the full state vector is \( \xvec \in \R^n \).

The data are generated at a high friction Reynolds number \( Re_\tau = 5,200 \) and discretized on a finite-volume grid of \( 384 \times 192 \times 32 \) cells in the streamwise, spanwise, and wall-normal directions, respectively, yielding a full state dimension of \( n = 384 \times 192 \times 32 \times 4 = 9,437,184 \). We collect 10,000 consecutive snapshots after the simulation reaches a statistically stationary state. 

Normalized by \( L \), the half-channel width, the domain size is \( [12L, 6L, 1L] \), the molecular viscosity is \( \nu = 8\times 10^{-6} \), and the forcing term is a constant streamwise pressure gradient given by \( F_{\omega_1} = \frac{\mathrm{d}\overline{p}}{\mathrm{d}\omega_1} = \overline{u_\tau}^2 \), where \( \overline{u_\tau} = 0.0415 \) is the mean friction velocity. The time stepping is adjusted to maintain a constant CFL number of 0.95 throughout the simulation, corresponding to a time step of \( \delta t \approx 0.023 \). The total simulation time of 10,000 time steps corresponds to approximately 10 units of \( L/u_\tau \), capturing sufficient dynamics of this highly chaotic system to demonstrate our reduced operator learning approach.

We partition the 10,000 snapshots by taking the first 8,000 as training snapshots and the remaining 2,000 as testing snapshots. The data are preprocessed via mean-centering followed by normalization to the range \([-1, 1]\), where each grid point is normalized individually using its temporal mean, minimum, and maximum values.

\Cref{fig:channel:spectral_decay} shows the spectral decay of the training data up to 500 modes, revealing the significant challenges posed by this turbulent flow. The slow spectral decay, where singular values remain on the order of \( 10^3 \) for many modes (presumably even beyond 500 modes), indicates that a large number of modes are required to capture the flow dynamics accurately. To address this challenge, we compute the \gls*{svd} components using the SketchySVD method, which, despite slightly higher memory requirements, provides superior accuracy in low-dimensional approximation compared to Baker's method as demonstrated in the previous experiments. We construct a \gls*{pod} basis of dimension 500 to improve accuracy by reducing errors from larger truncated singular values, as discussed in~\Cref{sec:bg:isvd}, and investigate reduced operators of dimension \( r = 300 \). Time derivatives are approximated using fourth-order central differences in the interior with fourth-order forward and backward differences at the boundaries.

\begin{figure}[t!]
    \centering
    \includegraphics[width=0.5\textwidth]{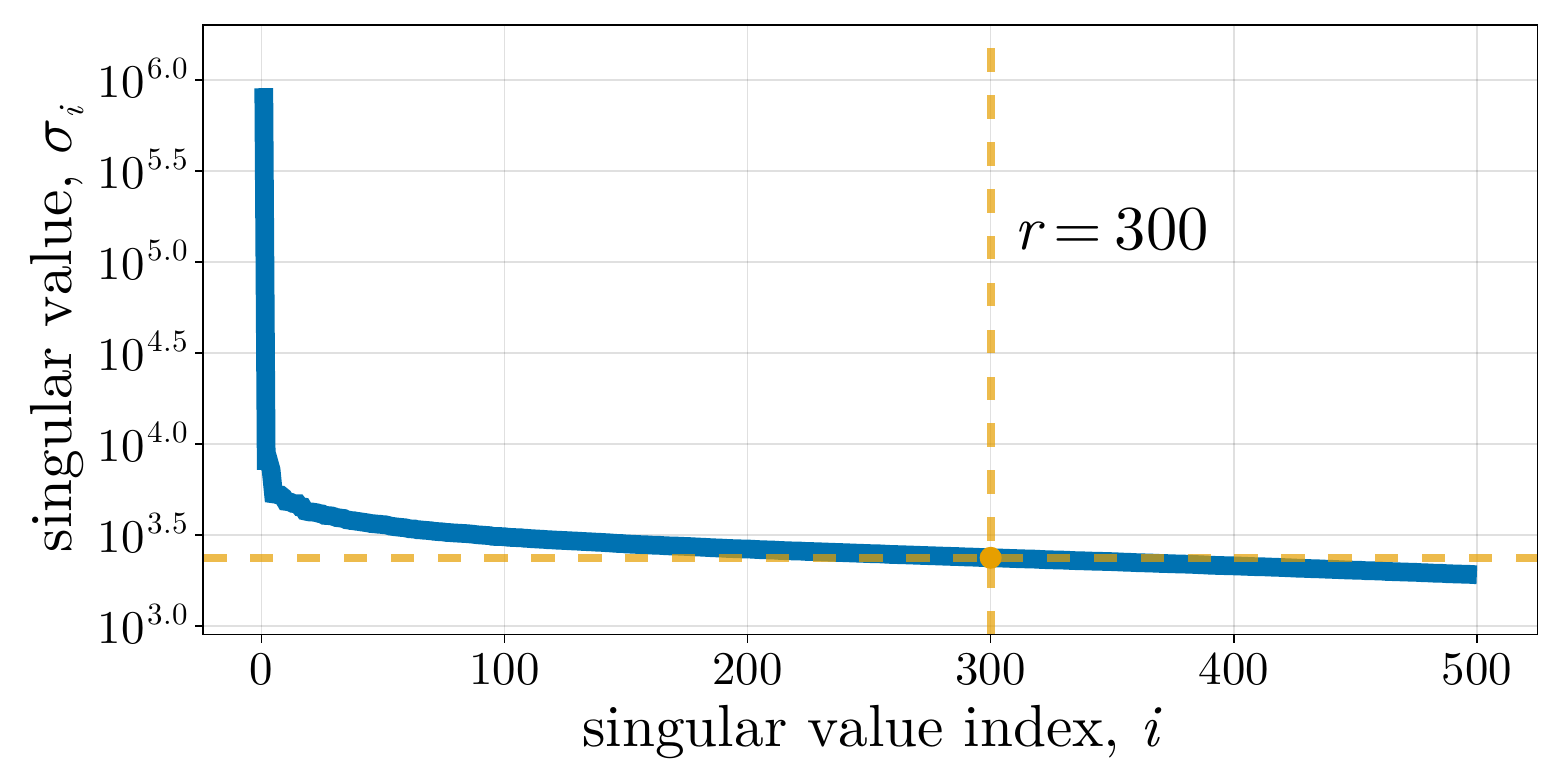}
    \vspace{-1.em}
    \caption{Spectral decay of the turbulent channel flow example for 500 singular values.}\label{fig:channel:spectral_decay}
\end{figure}

We will learn the reduced operators using iSVD-LS\@. Finally, 4\textsuperscript{th}-order Runge-Kutta scheme is used to integrate the learned reduced model in time. Given the problem's complexity, Tikhonov regularization parameters for the \gls*{ls} are carefully selected via grid search over 25 log-equidistant points in \( \gamma_1 \in [10^6, 10^{12}] \) and 20 log-equidistant points in \( \gamma_2 \in [10^{12}, 10^{16}] \). The optimal parameters are found to be \( (\gamma_1 , \gamma_2) = (1.7783\times 10^7,\, 4.2813\times 10^{12}) \). These high regularization values align with our discussion of mitigating operator errors in~\Cref{sec:stream:error:svd}.

\subsubsection{Evaluation Metrics}\label{sec:exp:channel:metrics}
The channel flow is a highly turbulent dynamical system with a large number of positive Lyapunov exponents. Hence, similar to the \gls{kse} example, point-wise error metrics are not very meaningful for this problem. Thus, we focus on an assessment of the flow field predictions, where we compare two-dimensional slices of the flow field at specific time instances, and two important \glspl*{qoi} that are commonly used in the fluid mechanics community to assess the accuracy of turbulent channel flow simulations. 

The first \gls*{qoi} is the friction velocity \( u_\tau \), which is an important scaling parameter for wall-bounded turbulent flows~\cite{bredberg2000wall,kadivar2021review}. In our problem, the friction velocity oscillates around the target mean \( \overline{u_\tau} \) used in the forcing term, and thus, we report the target mean values and visualize the friction velocity normalized by the target mean, i.e., \( u_\tau/\overline{u_\tau} \). 

The second \gls*{qoi} is the wall-normal profile of the streamwise velocity \( u_1 \), which represents the spatial characteristics of the flow, exhibiting a low-law profile for the mean velocity given by~\cite{moser1999direct}:
\begin{equation}\label{eqn:log-law}
    \frac{u_{1}}{u_\tau} = \frac{1}{\kappa} \log{\left(\frac{u_\tau \omega_3}{\nu}\right)} + B,
\end{equation}
where \( \kappa \) is the von Karman constant, and \( B \) is an intercept constant. The two \glspl*{qoi} are computed for both the training and testing data and compared with the results from the original data. 

\subsubsection{Results \& Discussion}\label{sec:exp:channel:results}

\begin{table}[htbp!]
    \centering
    \parbox{0.85\textwidth}{\caption{Target mean of friction velocity \( \overline{u_\tau} \) compared with the mean values from training and testing data calculated by computing the friction velocities then averaging them over time. Errors are relative to the reference value of \(\overline{u_\tau} = 0.0415\).}\label{tab:target-mean-utau}}\\
    \vspace{2mm}
    \begin{tabular}{c c c}
        \toprule
         & Training \( \overline{u_\tau} \) (Rel. Error \%) & Testing \( \overline{u_\tau} \) (Rel. Error \%) \\
        \midrule
        Original Data & 0.04137 (0.3049\%) & 0.04119 (0.7364\%) \\
        Streaming OpInf & 0.04126 (0.5611\%) & 0.04130 (0.4857\%) \\
        \bottomrule
    \end{tabular}
\end{table}

\paragraph{Assessment of the Flow Field Predictions}
\Cref{fig:channel:flow} present two-dimensional slices of the three-dimensional turbulent channel flow field at 3 or 2 temporal instances during training and testing, respectively. The visualizations focus on the streamwise velocity component \( \ot{u_1} \), which exhibits the most energetic turbulent structures. The first row displays the original high-fidelity data, the second row shows predictions from Streaming \gls*{opinf} with reduced dimension \( r = 300 \), and the third row presents the \gls*{pod} projection as a benchmark for the best achievable accuracy given the same reduced dimensionality. The fourth and fifth rows visualize the absolute errors between the original data and the Streaming \gls*{opinf} predictions and \gls*{pod} projections, respectively.

During the training phase in~\Cref{fig:channel:flow}, Streaming \gls*{opinf} successfully captures the complex turbulent structures characteristic of wall-bounded flows, including the elongated streaks and smaller-scale vortical structures similar to the original data across all temporal instances, demonstrating the method's ability to learn the underlying dynamics from streaming data. The \gls*{pod} projection results in the third and fifth row provide a visual reference for the inherent approximation error introduced by dimensionality reduction alone, independent of the operator learning process. Compared to the projection errors, the error distribution of Streaming \gls*{opinf} (fourth row) remains at a comparable magnitude and is spatially distributed throughout the domain. This indicates that the learned reduced operators accurately capture the system dynamics without introducing significant additional modeling error beyond the truncation error from dimensionality reduction. 

The testing phase results in~\Cref{fig:channel:flow} demonstrate the model's predictive capability on unseen data. The Streaming \gls*{opinf} predictions maintain their fidelity to the original flow structures throughout the testing period, successfully reproducing the characteristic turbulent patterns. Further, the Streaming \gls*{opinf} predictions present the spatial patterns more clearly than the \gls*{pod} projections. The error magnitudes and distributions remain consistent with those observed during training, indicating that the learned reduced model generalizes well without overfitting to the training data. This is particularly significant given the chaotic nature of turbulent flows, where small perturbations can lead to trajectory divergence. The sustained accuracy across 2,000 testing snapshots (approximately 2.5 time units of \( L/u_\tau \)) demonstrates the robustness of the learned reduced operators. 

\begin{figure}[t!]
    \centering
    \includegraphics[width=0.95\textwidth]{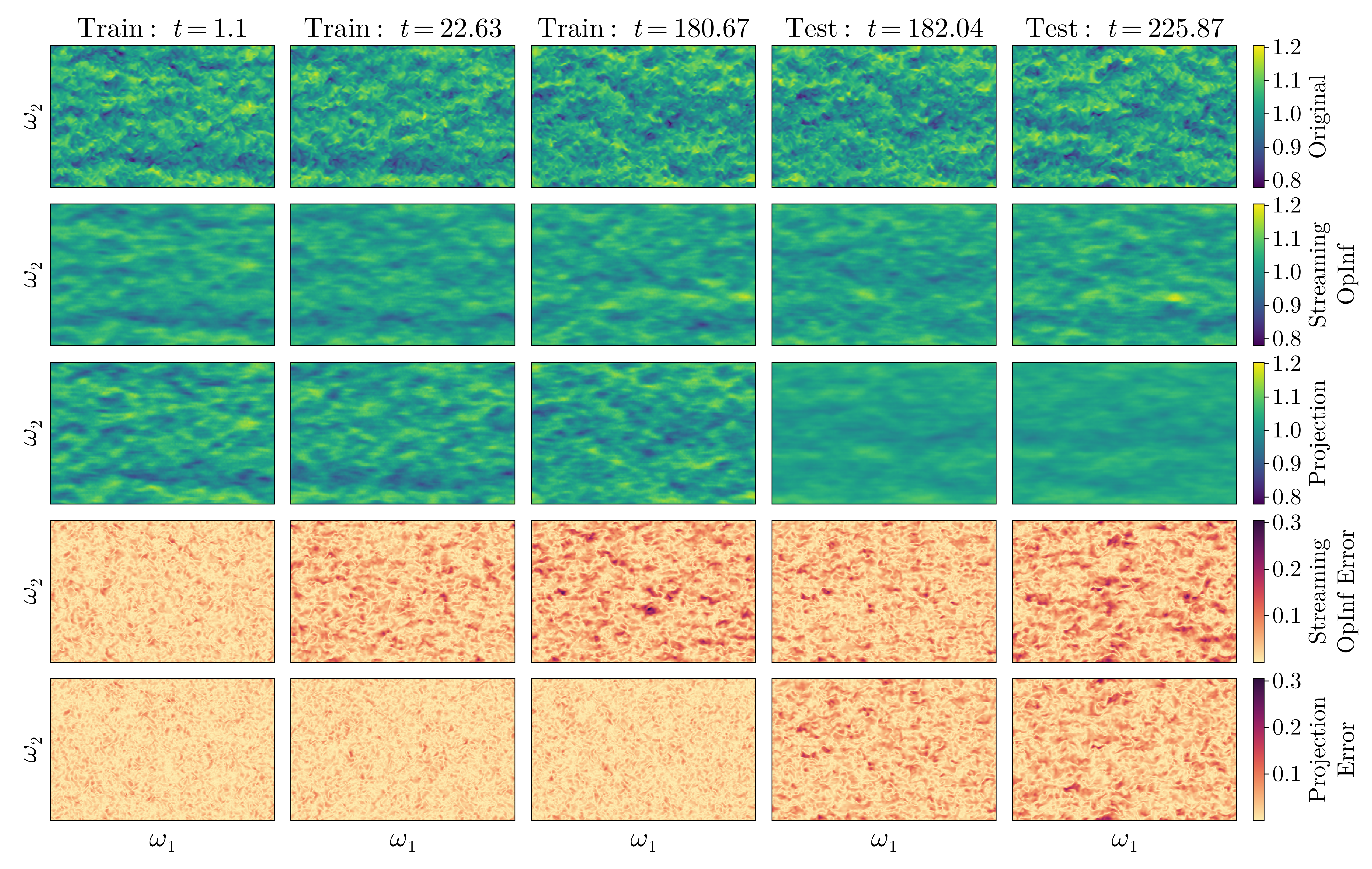}
    \vspace{-1.0em}
    \caption{Comparison of 2D slices of the turbulent channel flow field predictions for training and testing data of the \textbf{streamwise velocity} component. First row is the original data, the second row is the prediction from Streaming-OpInf with \( r = 300 \), the third row is the projection result from POD with \( r = 300 \), and the last two rows is the absolute error between the original data and the prediction from Streaming-OpInf and POD, respectively. Time evolves from left to right.}\label{fig:channel:flow}
\end{figure}


\paragraph{Quantitative Assessment via Quantities of Interest}
\Cref{fig:channel:friction-velocity} displays the friction velocity \( u_\tau \) evolution over time. The top panel shows both the original and predicted friction velocities normalized by the target mean friction velocity \( \overline{u_\tau} \) computed by averaging the calculated friction velocities over time. Streaming \gls*{opinf} show similar temporal fluctuations compared to the original data for both training and testing phases, capturing the characteristic variations arising from the passage of turbulent structures near the wall. The bottom panel presents the relative errors on a logarithmic scale, revealing errors consistently below \( 10^{-2} \) for the majority of snapshots, with occasional downward spikes decreasing to \( 10^{-5} \) that correspond to instantaneous overlaps between the predicted and true friction velocities.

\Cref{fig:channel:wall-normal-profile} presents the wall-normal profile of the time-averaged streamwise velocity component. The logarithmic scale on the wall-normal coordinate axis emphasizes the near-wall region where velocity gradients are steepest~\cite{bredberg2000wall}. For both training and testing data, Streaming \gls*{opinf} predictions closely match the original data across the entire wall-normal extent, accurately capturing both the viscous sublayer near the wall and the turbulent/logarithmic layer in the outer region while adhering to the log-law~\eqref{eqn:log-law}~\cite{pinier2019stochastic,marusic2013logarithmic}. The relative errors, shown in the bottom panels, remain below \( 10^{-3} \) for training and \( 10^{-2} \) for testing throughout most of the domain, with slightly elevated errors in the near-wall region where the flow physics are most complex. The consistency between training and testing errors confirms the model's ability to preserve statistical properties of the turbulent flow.

\begin{figure}[t!]
    \centering
    \includegraphics[width=0.75\textwidth]{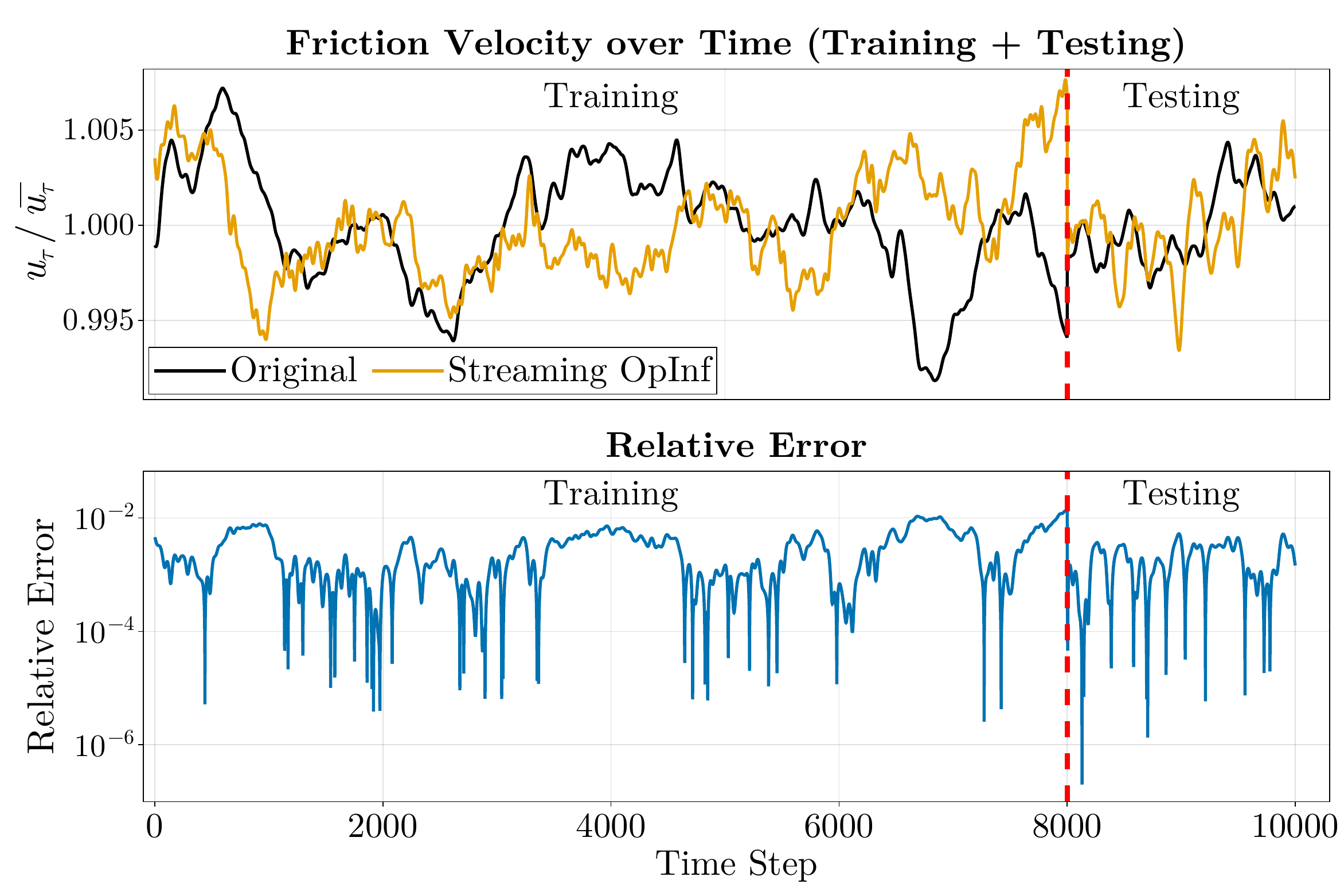}
    \vspace{-1.0em}
    \caption{(\textbf{QoI 1}) Friction velocity \( u_\tau \) for training and testing data.}\label{fig:channel:friction-velocity} 
\end{figure}

\begin{figure}[t!]
    \centering
    \includegraphics[width=0.75\textwidth]{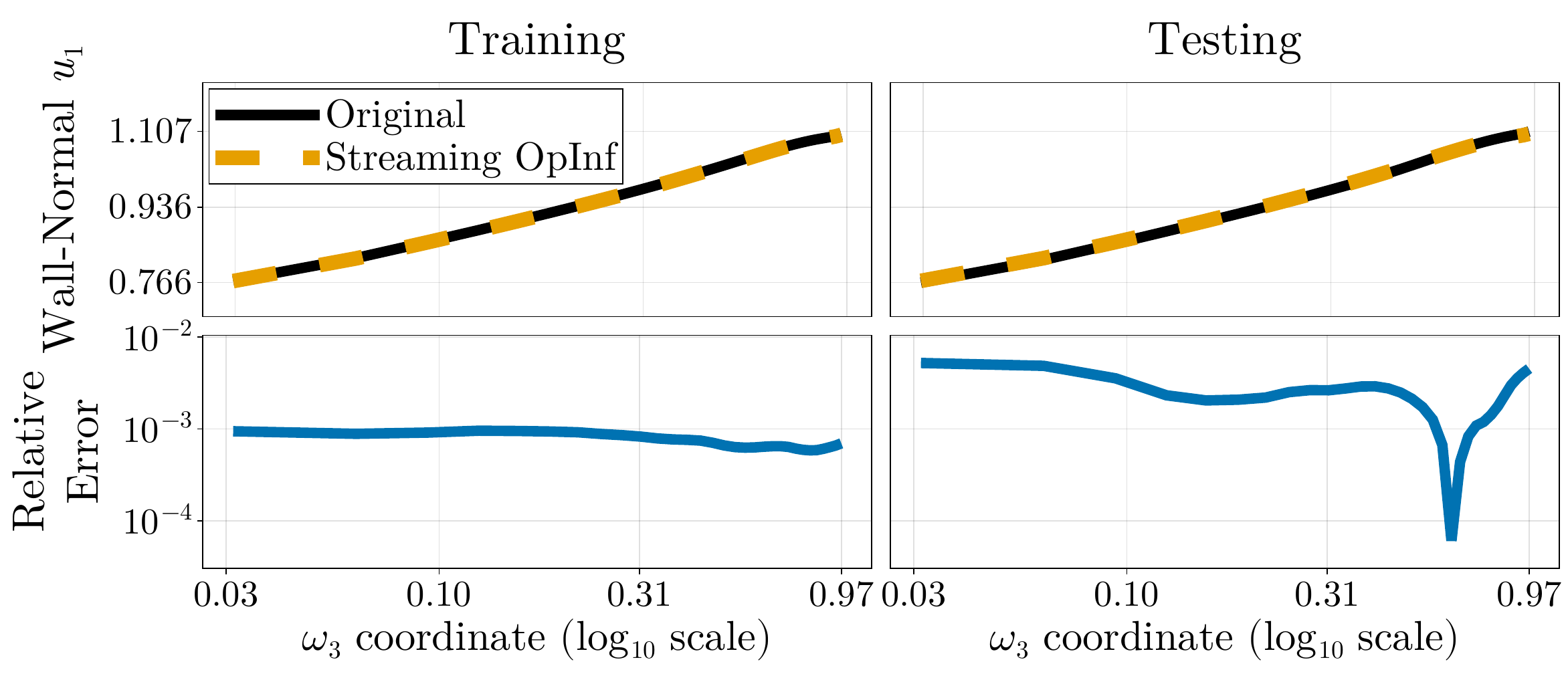}
    \vspace{-1.0em}
    \caption{(\textbf{QoI 2}) Wall normal profile of the streamwise velocity component for training and testing data.}\label{fig:channel:wall-normal-profile} 
\end{figure}

\Cref{tab:target-mean-utau} quantifies the time-averaged friction velocities for both training and testing data, comparing them against the target value \( \overline{u_\tau} = 0.0415 \) specified in the simulation setup. The Streaming \gls*{opinf} predictions yield mean friction velocities of 0.04126 and 0.04130 for training and testing data, respectively, closely matching both the original data statistics (0.04137 and 0.04119) and the target value. The relative errors of approximately 0.5\% demonstrate that Streaming \gls*{opinf} preserves the global flow statistics with high fidelity, an essential requirement for reduced models intended to replace high-fidelity simulations in many-query scenarios.

\paragraph{Discussion}
The results demonstrate that Streaming \gls*{opinf} enables accurate reduced model learning for a large-scale turbulent flow problem where batch \gls*{opinf} is infeasible. With over 9 million degrees of freedom and 8,000 snapshots, the dataset exceeds practical memory limits for simultaneous storage and batch processing. By processing snapshots sequentially, Streaming \gls*{opinf} makes this type of model reduction tractable, achieving a state dimension reduction exceeding 31,000x while maintaining high predictive accuracy and enabling significantly faster state predictions. This fundamentally expands the scope of \gls*{opinf} to large-scale problems previously inaccessible to batch methods.

For turbulent channel flow, Streaming \gls*{opinf} realizes 76\% memory reduction for the SVD process and 68\% total reduction compared to batch \gls*{opinf}~(\Cref{tab:computational-cost}). These substantial savings, combined with sequential data processing, establish the framework's suitability for real-time applications with continuous data arrival and limited storage.

Several algorithmic choices contribute to success in this challenging application: (i) selecting \( r = 300 \) captures dominant turbulent structures with sufficient energy retention, (ii) careful calibration of Tikhonov regularization \( (\gamma_1, \gamma_2) \) through systematic grid search prevents overfitting while promoting stable operator learning despite chaotic dynamics, and (iii) employing SketchySVD achieves favorable trade-offs between computational efficiency and approximation accuracy for massive data matrices.

Finally, the comparable performance between training and testing phases demonstrates that learned operators capture essential physics governing turbulent channel flow for long-term accuracy rather than memorizing training patterns. The sustained fidelity in both instantaneous flow fields and time-averaged statistics demonstrates that Streaming \gls*{opinf} preserves dynamical and statistical properties essential for reliable turbulence modeling.

\begin{table}[htpb!]
    \centering
    \caption{Memory cost comparison between batch and streaming approaches across all three experiments.}\label{tab:computational-cost}
    \vspace{2mm}
    
    \textbf{(a) Memory costs based on~\Cref{tab:isvd-costs,tab:rls-costs,tab:streaming-opinf-costs}}
    \vspace{1mm}
    
    \begin{tabular}{lcccc}
        \toprule
        Experiments & Batch SVD & Streaming SVD & Batch LS & Streaming OpInf LS \\
        \midrule
        Burgers' & 1.28e+07 & 8.32e+03 & 1.21e+06 & 1.44e+04 \\
        KSE & 1.38e+07 & 1.23e+04 & 1.08e+06 & 1.05e+05 \\
        Channel flow & 1.96e+10 & 3.80e+09 & 2.43e+09 & 2.43e+09 \\
        \bottomrule
    \end{tabular}
    
    \vspace{4mm}
    
    \textbf{(b) Memory cost reduction percentages}
    \vspace{1mm}
    
    \begin{tabular}{lccc}
        \toprule
        Experiments & SVD Reduction & LS Reduction & Total Reduction \\
        \midrule
        Burgers' & 99.94\% & 98.81\% & 99.84\% \\
        KSE & 99.91\% & 90.25\% & 99.21\% \\
        Channel flow & 76.36\% & 0.00\% & 68.11\% \\
        \bottomrule
    \end{tabular}
\end{table}

\section{Conclusion}\label{sec:conclusion}

In this work, we have developed a novel reduced operator learning approach that reformulates the traditional \gls*{opinf} framework to a data-streaming context. 
By integrating \gls*{isvd} with \gls*{rls} methods, our method efficiently constructs the reduced operators by only relying on the \gls*{svd} components and the current data sample, thus significantly reducing memory requirements and enabling the handling of large-scale datasets as well as online operator learning scenarios, which lays the groundwork for real-time applications.

We have explored various algorithmic choices within the Streaming \gls*{opinf} framework, including different \gls*{isvd} techniques (Baker's method and SketchySVD) and operator learning strategies (a new reformulation of the batch \gls*{ls} and \gls*{rls} methods). We have also presented practical strategies for choosing different algorithmic variations of the framework based on data availability and computational resources. 

Through comprehensive numerical experiments on benchmark problems such as the one-dimensional viscous Burgers' equation, the Kuramoto-Sivashinsky equation, and a large-scale turbulent channel flow simulation, we have achieved over 99\% memory savings compared to batch \gls*{opinf}---expanding the scope of \gls{opinf} to previously infeasible large-scale applications. 
The experiments also highlighted the importance of practical considerations such as data standardization and regularization in enhancing model performance and stability.

While the results are highly encouraging, some limitations warrant mention. Given the advection-dominant nature of the turbulent channel flow, \( r = 300 \) may still be insufficient to fully resolve all relevant scales, particularly in the near-wall region where velocity gradients are steepest. This is reflected in the friction velocity at later time instances where the trend deviates larger from the original data and the opaqueness compared to the original data in the flow field predictions of~\Cref{fig:channel:flow}. Future work could explore streaming formulations of more advanced reduced modeling methods, such as those utilizing nonlinear manifold approaches as demonstrated in~\cite{geelen2023Operator,schwerdtner2025operator}, which could potentially capture the complex flow physics with even lower dimensions. Additionally, extending the framework to accommodate more complex dynamical systems, exploring the combination of parallelization and streaming, and integrating uncertainty quantification techniques~\cite{guo2022Bayesiana,mcquarrie2025Bayesian} to further enhance the robustness of the learned models. 

Overall, Streaming \gls*{opinf} represents a significant advancement in the field of operator learning, providing a scalable, efficient solution for large-scale data scenarios and establishing foundations for rapid real-time predictions.

\section*{CRediT authorship contribution statement}
Tomoki Koike: Writing - review \& editing, Writing - original draft, Visualization, Validation, Software, Methodology, Investigation, Formal analysis, Data curation, Conceptualization; Prakash Mohan: Writing – original draft, review \& editing, Supervision, Resources, Data curation; Marc T.\ Henry de Frahan: Writing - review \& editing, Supervision, Resources, Data curation; Julie Bessac: Supervision, Funding Acquisition, Project administration; Elizabeth Qian: Conceptualization, Writing - review \& editing, Supervision, Funding Acquisition, Project administration.
\section*{Data Availability}

Data and experiment code are available at \href{https://github.com/smallpondtom/StreamingOpInf}{https://github.com/smallpondtom/StreamingOpInf} for reproducibility.

\section*{Declaration of competing interest}
The authors declare that they have no known competing financial interests or personal relationships that could have appeared to influence the work reported in this paper.
\section*{Acknowledgement}
The authors were supported by the Department of Energy Office of Science Advanced Scientific Computing Research, DOE Award DE-SC0024721. 
This work was authored in part by NLR for the U.S. Department of Energy (DOE), operated under Contract No. DE-AC36-08GO28308. The views expressed in the article do not necessarily represent the views of the DOE or the U.S. Government. The U.S. Government retains and the publisher, by accepting the article for publication, acknowledges that the U.S. Government retains a nonexclusive, paid-up, irrevocable, worldwide license to publish or reproduce the published form of this work, or allow others to do so, for U.S. Government purposes. A portion of the research was performed using computational resources sponsored by the Department of Energy's Office of Critical Minerals and Energy Innovation and located at the National Laboratory of the Rockies.

\appendix
\crefalias{section}{appendix}

\makeatletter
\renewcommand{\thesection}{\Alph{section}}
\renewcommand{\thelemma}{\thesection.\arabic{lemma}}
\renewcommand{\thetheorem}{\thesection.\arabic{theorem}}
\setcounter{section}{0}
\setcounter{theorem}{0}
\makeatother

\section{Streaming SVD Algorithms}\label{app:svd-algos}

\begin{algorithm}[H]
    \caption{Baker's iSVD Algorithm}\label{alg:baker-isvd}
    \begin{algorithmic}[1]
        \Require{} truncation dimension \(r\), total number of data vectors \(K\)
        \State{}Receive initial data vector \(\xvec_1\) and initialize: \(\Vmat_{r_k} \,\boldsymbol{\Sigma}_{r_k} \gets \mathsf{qr}(\xvec_1), ~ \Wmat_{r_k} \gets 1, ~ r_k \gets 1\)
        \For{\(k = 2\) to \(K\)}
            \State{Receive new data vector:} \(\xvec_{k}\)

            \State{Compute projection and residual:} \(\qvec \gets \Vmat^\top_{r_k} \xvec_k,\quad \xvec_{\perp} \gets \xvec_k - \Vmat_{r_k} \,\qvec \)
            \State{Second orthogonalization:} \(\qvec_2 \gets \Vmat^\top_{r_k} \xvec_{\perp},\quad \xvec_{\perp} \gets \xvec_{\perp} - \Vmat_{r_k} \,\qvec_2,\quad \qvec \gets \qvec + \qvec_2\)
            \State{QR decomposition of residual:} \(\xvec_{\perp} p \gets \mathsf{qr}(\xvec_{\perp})\)
            \State{Form augmented matrices:}
            \[
                \wh{\Vmat} \gets \begin{bmatrix}
                    \Vmat_{r_k} & \xvec_{\perp}
                \end{bmatrix},\quad 
                \Jmat \gets \begin{bmatrix}
                    \boldsymbol{\Sigma}_{r_k} & \qvec \\ 
                    \mathbf{0}_{1\times r_k} & p
                \end{bmatrix},\quad 
                \wh\Wmat \gets \begin{bmatrix}
                    \Wmat_{r_k} & \mathbf{0}_{k \times 1} \\ 
                    \mathbf{0}_{1\times r_k} & 1
                \end{bmatrix}, \quad
                r_k \gets r_k + 1
            \]
            \State{Compute SVD\@:} \(\Vmat_J \,\Sigmamat_J \,\Wmat^\top_J \gets \mathsf{svd}(\Jmat)\)
            \State{Update matrices:} \(\Vmat_{r_k} \gets \wh{\Vmat}\,\Vmat_J,\quad \Sigmamat_{r_k} \gets \Sigmamat_J,\quad \Wmat_{r_k} \gets \wh\Wmat \,\Wmat_J \)

            \If{\( r_k > r \)}
                \State{Truncate: }\(\Vmat_{r_k} \gets \Vmat_{r_{k}}(:,1:r),\quad \boldsymbol{\Sigma}_{r_k} \gets \boldsymbol{\Sigma}_{r_{k}}(1:r,1:r),\quad \Wmat_{r_k} \gets \Wmat_{r_{k}}(:,1:r), \quad r_k \gets r\)
            \EndIf{}

        \EndFor{}
        \State{\bfseries Return: }\(\Vmat_r,\ \Sigmamat_r,\ \Wmat_r \)
    \end{algorithmic}
\end{algorithm}

\begin{algorithm}[H]
    \caption{SketchySVD Algorithm}\label{alg:sketchy}
    \begin{algorithmic}[1]
        \Require{}truncation dimension \(r\), total number of data vectors \(K\), sketch sizes \(q\) and \(s\)
        \State{}Initialize reduction maps: \(\boldsymbol{\Upsilon}\in\R^{q\times n}, ~ \boldsymbol{\Omega}\in\R^{q\times K}, ~ \boldsymbol{\Xi}\in\R^{s\times n}, ~ \boldsymbol{\Psi}\in\R^{s\times K}\)
        \State{}Initialize sketches: \(\Xcal_{\mathsf{range}} \gets \mathbf{0}_{n\times q}, ~ \Xcal_{\mathsf{corange}} \gets \mathbf{0}_{q\times K}, ~ \Xcal_{\mathsf{core}} \gets \mathbf{0}_{s\times s}\)
        \For{\(k = 1\) to \(K\)}
            \State{Receive new data vector:} \(\xvec_{k}\)
            \State{Update sketches:} \Comment{\((i,j)\) is entry at row \(i\) and column \(j\), \( (:) \) are all entries in that dimension}
                \[
                    \Xcal_{\mathsf{range}} \gets \Xcal_{\mathsf{range}} + \xvec_k \boldsymbol{\Omega}^\top(k,:), \quad 
                    \Xcal_{\mathsf{corange}}(:,k) \gets \Xcal_{\mathsf{corange}}(:,k) + \boldsymbol{\Upsilon} \xvec_k, \quad 
                    \Xcal_{\mathsf{core}} \gets \Xcal_{\mathsf{core}} + (\boldsymbol{\Xi} \xvec_k) \boldsymbol{\Psi}^\top(k,:)
                \]
        \EndFor{}

        \State{Compute the core matrix \( \Cmat \in \R^{s\times s}\):} \Comment{\(\dagger \): Moore-Penrose pseudo-inverse}
            \[
                \Qmat_{\mathsf{range}}\Rmat_{\mathsf{range}} \gets \mathsf{qr}(\Xcal_{\mathsf{range}}), \quad \Qmat_{\mathsf{corange}}\Rmat_{\mathsf{corange}} \gets \mathsf{qr}(\Xcal_{\mathsf{corange}}^\top), \quad 
                \Cmat \gets {(\boldsymbol{\Xi}\Qmat_{\mathsf{range}})}^\dagger \,\Xcal_{\mathsf{core}}\,({(\boldsymbol{\Psi}\Qmat_{\mathsf{corange}})}^\dagger)\tran
            \]
        \State{Compute SVD of the core matrix: } \(\Vmat_c \Sigmamat_c \Wmat_c^\top \gets \mathsf{svd}(\Cmat)\)
        \State{Form SVD of \( \Xmat \):} \(\Vmat_r \gets \Qmat_{\mathsf{range}}\Vmat_c(:,1:r), ~ \Sigmamat_r \gets \Sigmamat_c(1:r,1:r), ~ \Wmat_r \gets \Qmat_{\mathsf{corange}} \Wmat_c(:,1:r) \)
        \State{\bfseries Return: }\(\Vmat_r,\ \Sigmamat_r,\ \Wmat_r \)
    \end{algorithmic}
\end{algorithm}
\section{Recursive Least-Squares Algorithms}\label{app:rls-algos}

\begin{algorithm}[H]
    \caption{RLS Algorithm}\label{alg:rls}
    \begin{algorithmic}[1]
        \Require{Tikhonov matrix \( \Gammamat = \Gammamat^\top > 0 \)}
        \State{\( \Pmat_0 = \Gammamat^{-1} \)}  \Comment{Initialize inverse correlation matrix}
        \State{\( \Omat_0 = \0_{d\times r} \)}  \Comment{Initialize operator matrix}
        \For{\( k = 1,2,3,\ldots,K \)}
            \State{Receive new data: \( \dvec_k, \rvec_k \)}
            \State{\( c_k = {\left(1 + \dvec_k \Pmat_{k-1} \dvec_k^\top \right)}^{-1} \)}  \Comment{Compute conversion factor}
            \State{\( \gvec_k = \Pmat_{k-1}\dvec_k^\top c_k \)} \Comment{Update Kalman gain}
            \State{\( \Pmat_k = \Pmat_{k-1} - \gvec_k\gvec_k^\top c_k^{-1} \)}  \Comment{Update inverse correlation matrix}
            \State{\( \boldsymbol{\xi}_k^- = \rvec_k - \dvec_k\Omat_{k-1} \)}  \Comment{Compute \textit{a priori} error}
            \State{\( \Omat_k = \Omat_{k-1} + \gvec_k\boldsymbol{\xi}_k^- \)} \Comment{Update operator matrix}
        \EndFor{}
        \State{\bfseries Return: }\( \Omat_k \)
    \end{algorithmic}
\end{algorithm}

\begin{algorithm}[H]
    \caption{Inverse QR-Decomposition RLS Algorithm}\label{alg:iqrrls}
    \begin{algorithmic}[1]
        \Require{Tikhonov matrix \( \Gammamat = \Gammamat^\top > 0 \) }
        \State{Initialize \( \Pmat_0^{1/2} \gets \Gammamat^{-1/2} \)} \Comment{Cholesky decomposition of inverse correlation matrix}
        \State{Initialize \( \Omat_0 \gets \mathbf{0}_{d \times r} \)} \Comment{Operator estimate}
        \For{\( k = 1 \) \textbf{to} \( K \)}
            \State{Receive new data: \( \dvec_k \), \( \rvec_k \)}
            \State{Form tranposed pre-array:}
            \[
                \Acal_k^\top = \begin{bmatrix}
                    1 & \mathbf{0}_{1 \times d} \\[0.3em]
                    \Pmat_{k-1}^{\top/2} \dvec_k^\top & \Pmat_{k-1}^{\top/2}
                \end{bmatrix}
            \]
            \State{Perform QR decomposition (or Givens rotations): \( \Thetamat_k\Bcal_k^\top  \gets \mathsf{qr}(\Acal_k^\top)\)}
            \State{Extract updated quantities:}
            \[
                c_k^{-1/2} \gets \Bcal_k^\top(1,1), \quad
                \gvec_k c_k^{-1/2} \gets \Bcal_k^\top(1,2:d+1), \quad
                \Pmat_k^{1/2} \gets {(\Bcal_k^\top(2:d+1,2:d+1))}^\top
            \]
            \State{Compute the a priori error: \( \boldsymbol{\xi}_k^- \gets \rvec_k - \dvec_k \Omat_{k-1} \)}
            \State{Update the operator estimate:}
            \[
                \Omat_k = \Omat_{k-1} + \left( \gvec_k c_k^{-1/2} \right) {\left( c_k^{-1/2} \right)}^{-1} \boldsymbol{\xi}_k^- = \Omat_{k-1} + \gvec_k \boldsymbol{\xi}_k^-
            \]
        \EndFor{}
        \State{\bfseries Return: }\( \Omat_K \)
    \end{algorithmic}
\end{algorithm}

\section{Proof of Operator Error Bounds}\label[appendix]{app:proof}

This section establishes the error bounds in Theorems~\ref{thm:operator-error-projection} and~\ref{thm:operator-error-reformulation}, which quantify how errors from streaming \gls*{svd} propagate through the data matrices into the learned operators. We begin by proving the general operator perturbation bound in Lemma~\ref{lem:operator-perturbation}, then derive the specific bounds for each streaming paradigm.

\subsection{General Operator Perturbation Bound}

The proof of Lemma~\ref{lem:operator-perturbation} relies on a classical perturbation result for pseudoinverses.

\begin{lemma}[Theorem 4.1 of~\cite{wedin1973perturbation}]\label{lem:wedin-perturbation}
    Let \( \Amat, \Bmat \in \R^{m \times n} \) with \( \mathrm{rank}(\Amat) = \mathrm{rank}(\Bmat) = \min \{m, n\} \). Then
    \begin{equation*}
        \| \Amat^\dagger - \Bmat^\dagger \|_2 \le \alpha \| \Amat^\dagger \|_2 \| \Bmat^\dagger \|_2 \| \Amat - \Bmat \|_2,
    \end{equation*}    
    where \( \alpha = \sqrt{2} \) if \( m \ne n \) and \( \alpha = 1 \) if \( m = n \).
\end{lemma}

\begin{proof}[Proof of Lemma~\ref{lem:operator-perturbation}]
    Let \( \bbar\Dmat \) and \( \bbar\Rmat \) be the data matrices defined in~\eqref{eqn:opinf-pinv-regularized}, and let \( \ot\Dmat = \bbar\Dmat + [\Emat_D^\top,\0_{d\times d}]\tran \) and \( \ot\Rmat = \bbar\Rmat + [\Emat_R^\top,\0_{r\times d}]\tran \) denote the perturbed data matrices. The corresponding \gls{ls} solutions are \( \Omat = \bbar\Dmat^\dagger\bbar\Rmat \) and \( \ot\Omat = \ot\Dmat^\dagger \ot\Rmat \). Adding and subtracting \( \bbar\Dmat^\dagger\ot\Rmat \) and applying the triangle inequality yields
    \begin{align*}
        \| \Omat - \ot\Omat \|_F
        &= \| \bbar\Dmat^\dagger \bbar\Rmat - \ot\Dmat^\dagger \ot\Rmat \|_F \\
        &= \| \bbar\Dmat^\dagger\bbar\Rmat - \bbar\Dmat^\dagger\ot\Rmat + \bbar\Dmat^\dagger\ot\Rmat - \ot\Dmat^\dagger\ot\Rmat \|_F \\
        &\le \| \bbar\Dmat^\dagger(\bbar\Rmat - \ot\Rmat) \|_F + \| (\bbar\Dmat^\dagger - \ot\Dmat^\dagger)\ot\Rmat \|_F.
    \end{align*}
    Applying the submultiplicativity property \( \| \Amat\Bmat \|_F \leq \| \Amat \|_2 \|\Bmat \|_F \) to each term gives
    \begin{equation*}
        \| \Omat - \ot\Omat \|_F \le \| \bbar\Dmat^\dagger \|_2 \|\bbar\Rmat - \ot\Rmat \|_F + \| \bbar\Dmat^\dagger - \ot\Dmat^\dagger \|_2 \|\ot\Rmat \|_F.
    \end{equation*}
    Since \( \| [\Emat_R^\top,\0_{r\times d}]\tran \|_F = \| \Emat_R\|_F \) and \( \| [\Emat_D^\top,\0_{d\times d}]\tran \|_2 = \| \Emat_D\|_2 \), we have \( \|\bbar\Rmat - \ot\Rmat \|_F = \|\Emat_R\|_F \) and \( \|\bbar\Dmat - \ot\Dmat \|_2 = \|\Emat_D\|_2 \). Applying Lemma~\ref{lem:wedin-perturbation} to bound \( \| \bbar\Dmat^\dagger - \ot\Dmat^\dagger \|_2 \) completes the proof.
\end{proof}

\subsection{Operator Error Bounds for Each Paradigm}

We now prove Theorems~\ref{thm:operator-error-projection} and~\ref{thm:operator-error-reformulation} by deriving explicit bounds on the data matrix perturbations \( \Emat_D \) and \( \Emat_R \) for each streaming paradigm, then applying Lemma~\ref{lem:operator-perturbation}. We establish several auxiliary results needed for both proofs.

\begin{lemma}[Norm of the pseudoinverse]\label{lem:norm-pinv}
    For any \( \Amat \in \R^{m \times n} \), we have \( \| \Amat^\dagger \|_2 = \sigma_{\min}(\Amat)^{-1} \), where \( \sigma_{\min}(\Amat) \) denotes the smallest nonzero singular value of \( \Amat \).
\end{lemma}

\begin{proof}
    This follows directly from the \gls{svd} \( \Amat = \Umat \Sigmamat \Vmat^\top \) and the pseudoinverse definition \( \Amat^\dagger = \Vmat \Sigmamat^\dagger \Umat^\top \).
\end{proof}

The following result from Weyl's inequality~\cite[Cor.\,8.6.2]{golub2013Matrix} bounds the singular value perturbations.

\begin{lemma}[Singular value perturbation]\label{lem:weyl}
    If \( \| \Xmat - \ot{\Xmat} \|_2 \leq \varepsilon \), then \( |\sigma_j - \ot{\sigma}_j| \leq \varepsilon \) for all \( j = 1, \ldots, r \).
\end{lemma}

\subsubsection{Proof of Theorem~\ref{thm:operator-error-projection}}

\begin{proof}
We derive bounds on the data matrix perturbations for the projection paradigms, then combine them with Lemma~\ref{lem:operator-perturbation} to obtain the operator error bound.

\paragraph{Bounding the left-hand side perturbation \( \Emat_D \)}
To systematically bound \( \Emat_D \) for the projection paradigms, we introduce the full-state data matrix
\begin{equation*}
    \mathbb{D} = \begin{bmatrix}
        \Xmat^\top & (\Xmat \odot \Xmat)^\top & \Umat^\top & \mathbf{1}_K
    \end{bmatrix} \in \R^{K \times (n + n^2 + m + 1)},
\end{equation*}
and the projection matrices
\begin{equation*}
    \mathbb{V}_r = \begin{bmatrix}
        \Vmat_r & 0 & 0 & 0 \\
        0 & \Vmat_r \otimes \Vmat_r & 0 & 0 \\
        0 & 0 & \Id_m & 0 \\
        0 & 0 & 0 & 1
    \end{bmatrix}, \quad
    \ot{\mathbb{V}}_r = \begin{bmatrix}
        \ot\Vmat_r & 0 & 0 & 0 \\
        0 & \ot\Vmat_r \otimes \ot\Vmat_r & 0 & 0 \\
        0 & 0 & \Id_m & 0 \\
        0 & 0 & 0 & 1
    \end{bmatrix},
\end{equation*}
where \( \mathbb{V}_r, \ot{\mathbb{V}}_r \in \R^{(n + n^2 + m + 1) \times (r + r^2 + m + 1)} \) have orthonormal columns. This notation allows us to express the data matrices compactly as
\begin{equation*}
    \bbar\Dmat = [(\mathbb{D} \mathbb{V}_r)\tran,\, \Gammamat^{1/2}]\tran \quad \text{and} \quad 
    \ot\Dmat = [(\mathbb{D} \ot{\mathbb{V}}_r)\tran,\, \Gammamat^{1/2}]\tran.
\end{equation*}
By submultiplicativity,
\begin{equation}\label{eqn:proof-ED-proj-start}
    \| \Emat_D \|_2 = \| \mathbb{D} ( \mathbb{V}_r - \ot{\mathbb{V}}_r ) \|_2 \le \| \mathbb{D} \|_2 \| \mathbb{V}_r - \ot{\mathbb{V}}_r \|_2.
\end{equation}

We first bound \( \| \mathbb{V}_r - \ot{\mathbb{V}}_r \|_2 \). The block-diagonal structure implies~\cite[p.\,17]{tropp2016expected}
\begin{equation*}
    \| \mathbb{V}_r - \ot{\mathbb{V}}_r \|_2 = \max \bigl\{ \| \Vmat_r - \ot\Vmat_r \|_2, \| \Vmat_r \otimes \Vmat_r - \ot\Vmat_r \otimes \ot\Vmat_r \|_2 \bigr\}.
\end{equation*}
Since the Frobenius norm dominates the spectral norm and \( \| \Vmat_r - \ot\Vmat_r \|_F = \tau_v \), we have \( \| \Vmat_r - \ot\Vmat_r \|_2 \le \tau_v \). For the Kronecker product term, we add and subtract \( \ot\Vmat_r \otimes \Vmat_r \) to obtain
\begin{equation*}
    \Vmat_r \otimes \Vmat_r - \ot\Vmat_r \otimes \ot\Vmat_r = (\Vmat_r - \ot\Vmat_r) \otimes \Vmat_r + \ot\Vmat_r \otimes (\Vmat_r - \ot\Vmat_r).
\end{equation*}
The triangle inequality and the spectral norm property \( \| \Amat \otimes \Bmat \|_2 = \| \Amat \|_2 \| \Bmat \|_2 \)~\cite[Thm.\,8]{lancaster1972norms} yield
\begin{align*}
    \| \Vmat_r \otimes \Vmat_r - \ot\Vmat_r \otimes \ot\Vmat_r \|_2 
    &\le \| \Vmat_r - \ot\Vmat_r \|_2 \| \Vmat_r \|_2 + \| \ot\Vmat_r \|_2 \| \Vmat_r - \ot\Vmat_r \|_2 \\
    &= 2\| \Vmat_r - \ot\Vmat_r \|_2 \le 2\tau_v,
\end{align*}
where we used \( \| \Vmat_r \|_2 = \| \ot\Vmat_r \|_2 = 1 \) since both matrices have orthonormal columns. Thus, \( \| \mathbb{V}_r - \ot{\mathbb{V}}_r \|_2 \le 2\tau_v \).

Next, we bound \( \| \mathbb{D} \|_2 \). With bounded inputs \( \| \Umat \|_2 = \eta \) for some \( \eta > 0 \), we have
\begin{equation*}
    \| \mathbb{D} \|_2^2 \le \| \Xmat \|_2^2 + \| \Xmat \odot \Xmat \|_2^2 + \eta^2 + K.
\end{equation*}
For the Khatri-Rao product term, we use the identity \( (\Xmat \odot \Xmat)^\top (\Xmat \odot \Xmat) = (\Xmat^\top \Xmat) \circ (\Xmat^\top \Xmat) \), where \( \circ \) denotes the Hadamard product. By~\cite[Thm.\,5.5.1]{horn1994topics}, \( \sigma_1((\Xmat^\top \Xmat) \circ (\Xmat^\top \Xmat)) \le \sigma_1(\Xmat^\top \Xmat)^2 = \| \Xmat \|_2^4 \). Therefore,
\begin{equation*}
    \| \mathbb{D} \|_2 \le \bigl( \sigma_1^2 + \sigma_1^4 + \eta^2 + K \bigr)^{1/2} =: \beta_1',
\end{equation*}
where \( \sigma_1 = \| \Xmat \|_2 \). Substituting into~\eqref{eqn:proof-ED-proj-start} gives
\begin{equation}\label{eqn:proof-ED-proj-final}
    \| \Emat_D \|_2 \le 2\beta_1' \tau_v = \frac{\beta_1}{\sqrt{\min\{n,K\}}} \tau_v,
\end{equation}
where \( \beta_1 = 2\beta_1' \sqrt{\min\{n,K\}} \).

\paragraph{Bounding the right-hand side perturbation \( \Emat_R \)}
For the projection paradigms, \( \Emat_R = (\Vmat_r - \ot\Vmat_r)^\top \Xdot \). Using the finite-difference assumption \( \Xdot = \Xmat\Deltamat_K \) and submultiplicativity,
\begin{equation}\label{eqn:proof-ER-proj}
    \| \Emat_R \|_F = \| (\Vmat_r - \ot\Vmat_r)^\top \Xmat \Deltamat_K \|_F \le \| \Vmat_r - \ot\Vmat_r \|_F \| \Xmat \|_2 \| \Deltamat_K \|_2 = \sigma_1 \|\Deltamat_K \|_2 \tau_v.
\end{equation}

\paragraph{Bounding the pseudoinverse norms}
In \gls*{opinf}, the data matrix \( \Dmat \) is commonly ill-conditioned due to linear dependence among columns arising from small time steps and the quadratic nonlinearity~\cite{peherstorferW2016,geelen2022Localized}. When the state data exhibits rapid spectral decay, quadratic terms amplify this effect by squaring the singular values, further reducing \( \sigma_{\min}(\Dmat) \). The regularization parameter is typically chosen to satisfy \( \sqrt{\gamma_{\min}} = \min_j \sqrt{\gamma_j} \gg \sigma_{\min}(\Dmat) \). Similar reasoning applies to \( \ot\Dmat \), giving \( \sqrt{\gamma_{\min}} \gg \sigma_{\min}(\Dmat + \Emat_D) \).

Hence, Lemma~\ref{lem:norm-pinv} yields
\begin{equation}\label{eqn:proof-Dmat-pinv}
    \| \bbar\Dmat^\dagger \|_2 = \sigma_{\min}(\bbar\Dmat)^{-1} = \bigl[\sigma_{\min}(\Dmat)^2 + \gamma_{\min}\bigr]^{-1/2} \approx \gamma_{\min}^{-1/2},
\end{equation}
and similarly,
\begin{equation}\label{eqn:proof-otDmat-pinv}
    \| \ot\Dmat^\dagger \|_2 = \bigl[\sigma_{\min}(\Dmat + \Emat_D)^2 + \gamma_{\min}\bigr]^{-1/2} \approx \gamma_{\min}^{-1/2}.
\end{equation}

For \( \| \ot\Rmat \|_F \), we use \( \ot\Rmat = \ot\Vmat_r^\top \Xdot = \ot\Vmat_r^\top \Xmat\Deltamat_K \) to obtain
\begin{equation}\label{eqn:proof-Rmat-proj}
    \| \ot\Rmat \|_F \le \| \ot\Vmat_r \|_2 \| \Xmat \|_F \| \Deltamat_K \|_2 \le \sigma_1 \sqrt{\min\{n, K\}} \| \Deltamat_K \|_2,
\end{equation}
where we used \( \| \ot\Vmat_r \|_2 = 1 \) and \( \| \Xmat \|_F \le \sigma_1 \sqrt{\min\{n, K\}} \).

\paragraph{Completing the proof}
Substituting the bounds from~\eqref{eqn:proof-ED-proj-final}, \eqref{eqn:proof-ER-proj}, \eqref{eqn:proof-Dmat-pinv}, \eqref{eqn:proof-otDmat-pinv}, and~\eqref{eqn:proof-Rmat-proj} into Lemma~\ref{lem:operator-perturbation} yields
\begin{align*}
    \| \Omat - \ot\Omat \|_F 
    &\le \frac{\alpha}{\gamma_{\min}} \cdot \frac{\beta_1}{\sqrt{\min\{n,K\}}} \tau_v \cdot \sigma_1 \sqrt{\min\{n, K\}} \| \Deltamat_K \|_2 + \frac{1}{\sqrt{\gamma_{\min}}} \cdot \sigma_1 \| \Deltamat_K \|_2 \tau_v = \frac{\sigma_1 \| \Deltamat_K \|_2}{\sqrt{\gamma_{\min}}} \tau_v \Bigl( 1 + \frac{\alpha \beta_1}{\sqrt{\gamma_{\min}}} \Bigr),
\end{align*}
which completes the proof.
\end{proof}

\subsubsection{Proof of Theorem~\ref{thm:operator-error-reformulation}}

\begin{proof}
We follow a similar strategy as for Theorem~\ref{thm:operator-error-projection}, deriving bounds on \( \Emat_D \) and \( \Emat_R \) specific to the reformulation paradigms.

\paragraph{Bounding the left-hand side perturbation \( \Emat_D \)}
For the reformulation paradigms, the data matrix perturbation has the block structure
\begin{equation*}
    \Emat_D = \begin{bmatrix}
        \Wmat_r \Sigmamat_r - \ot\Wmat_r \ot\Sigmamat_r &
        (\Wmat_r \Sigmamat_r) \otimes (\Wmat_r\Sigmamat_r) - (\ot\Wmat_r \ot\Sigmamat_r) \otimes (\ot\Wmat_r\ot\Sigmamat_r) &
        \mathbf{0} &
        \mathbf{0}
    \end{bmatrix}.
\end{equation*}
This implies
\begin{equation}\label{eqn:proof-ED-reform}
    \| \Emat_D \|_2^2 \le \| \Wmat_r \Sigmamat_r - \ot\Wmat_r \ot\Sigmamat_r \|_2^2 + \| (\Wmat_r \Sigmamat_r) \otimes (\Wmat_r\Sigmamat_r) - (\ot\Wmat_r \ot\Sigmamat_r) \otimes (\ot\Wmat_r\ot\Sigmamat_r) \|_2^2.
\end{equation}

For the first term, we add and subtract \( \ot\Wmat_r \Sigmamat_r \):
\begin{equation*}
    \Wmat_r \Sigmamat_r - \ot\Wmat_r \ot\Sigmamat_r = (\Wmat_r - \ot\Wmat_r)\Sigmamat_r + \ot\Wmat_r(\Sigmamat_r - \ot\Sigmamat_r).
\end{equation*}
Applying the triangle inequality, submultiplicativity, the relation \( \| \Wmat_r - \ot{\Wmat}_r \|_F = \tau_w \), and Lemma~\ref{lem:weyl} yields
\begin{equation}\label{eqn:proof-first-term-reform}
    \| \Wmat_r \Sigmamat_r - \ot\Wmat_r \ot\Sigmamat_r \|_F
    \le \sigma_1 \| \Wmat_r - \ot\Wmat_r \|_F + \| \ot\Wmat_r \|_2 \| \Sigmamat_r - \ot\Sigmamat_r \|_F 
    \le \sigma_1 \tau_w + \sqrt{r}\, \varepsilon =: \phi_w,
\end{equation}
where we used \( \| \ot\Wmat_r \|_2 = 1 \) and \( \| \Sigmamat_r - \ot\Sigmamat_r \|_F \le \sqrt{r}\varepsilon \). Since the Frobenius norm dominates the spectral norm, \( \| \Wmat_r \Sigmamat_r - \ot\Wmat_r \ot\Sigmamat_r \|_2 \le \phi_w \).

For the Kronecker product term in~\eqref{eqn:proof-ED-reform}, we add and subtract \( (\ot\Wmat_r \ot\Sigmamat_r) \otimes (\Wmat_r \Sigmamat_r) \) and apply the triangle inequality and~\cite[Thm.\,8]{lancaster1972norms}:
\begin{equation*}
    \| (\Wmat_r \Sigmamat_r) \otimes (\Wmat_r\Sigmamat_r) - (\ot\Wmat_r \ot\Sigmamat_r) \otimes (\ot\Wmat_r\ot\Sigmamat_r) \|_2
    \le \bigl(\| \Wmat_r \Sigmamat_r \|_2 + \| \ot\Wmat_r \ot\Sigmamat_r \|_2\bigr) \| \Wmat_r \Sigmamat_r - \ot\Wmat_r \ot\Sigmamat_r \|_2.
\end{equation*}
Since \( \| \Wmat_r \Sigmamat_r \|_2 = \sigma_1 \) and \( \| \ot\Wmat_r \ot\Sigmamat_r \|_2 = \max_j \ot\sigma_j \le \sigma_1 + \varepsilon \) by Lemma~\ref{lem:weyl}, we obtain
\begin{equation}\label{eqn:proof-second-term-reform}
    \| (\Wmat_r \Sigmamat_r) \otimes (\Wmat_r\Sigmamat_r) - (\ot\Wmat_r \ot\Sigmamat_r) \otimes (\ot\Wmat_r\ot\Sigmamat_r) \|_2 \le (2\sigma_1 + \varepsilon)\phi_w.
\end{equation}

Substituting~\eqref{eqn:proof-first-term-reform} and~\eqref{eqn:proof-second-term-reform} into~\eqref{eqn:proof-ED-reform} and taking square roots gives
\begin{equation*}
    \| \Emat_D \|_2 \le \phi_w \sqrt{1 + (2\sigma_1 + \varepsilon)^2}.
\end{equation*}
Expanding \( \phi_w = \sigma_1 \tau_w + \sqrt{r}\varepsilon \) and defining \( \beta_2 = \sqrt{r}(\sigma_1 + \varepsilon)\sqrt{1 + (2\sigma_1 + \varepsilon)^2} \) yields
\begin{equation}\label{eqn:proof-ED-reform-final}
    \| \Emat_D \|_2 = \frac{\sigma_1\beta_2}{\sqrt{r}(\sigma_1 + \varepsilon)}\tau_w + \frac{\beta_2}{\sigma_1 + \varepsilon}\varepsilon.
\end{equation}

\paragraph{Bounding the right-hand side perturbation \( \Emat_R \)}
For reformulation paradigms, \( \Rmat = \Deltamat_K^\top \Wmat_r \Sigmamat_r \) and \( \ot\Rmat = \Deltamat_K^\top \ot\Wmat_r \ot\Sigmamat_r \), so
\begin{equation}\label{eqn:proof-ER-reform}
    \| \Emat_R \|_F = \| \Deltamat_K^\top ( \Wmat_r \Sigmamat_r - \ot\Wmat_r \ot\Sigmamat_r ) \|_F \le \| \Deltamat_K \|_2 \phi_w 
    = \sigma_1 \| \Deltamat_K \|_2 \tau_w + \sqrt{r} \| \Deltamat_K \|_2 \varepsilon.
\end{equation}

\paragraph{Bounding \( \| \ot\Rmat \|_F \)}
Since \( \ot\Rmat = \Deltamat_K^\top \ot\Wmat_r \ot\Sigmamat_r \) and \( \ot\Wmat_r \) has orthonormal columns,
\begin{equation}\label{eqn:proof-Rmat-reform-start}
    \| \ot\Rmat \|_F \le \| \Deltamat_K \|_2 \| \ot\Sigmamat_r \|_F.
\end{equation}
To bound \( \| \ot\Sigmamat_r \|_F \), let \( \sigmavec = [\sigma_1, \ldots, \sigma_r]^\top \) and \( \deltavec = [\delta_1, \ldots, \delta_r]^\top \) where \( \delta_j = \ot\sigma_j - \sigma_j \) satisfies \( |\delta_j| \le \varepsilon \) by Lemma~\ref{lem:weyl}. Then
\begin{equation*}
    \| \ot\Sigmamat_r \|_F^2 = \sum_{j=1}^r (\sigma_j + \delta_j)^2 = \| \Sigmamat_r \|_F^2 + 2 \sigmavec^\top \deltavec + \| \deltavec \|_2^2.
\end{equation*}
The Cauchy-Schwarz inequality gives \( | \sigmavec^\top \deltavec | \le \| \sigmavec \|_2 \| \deltavec \|_2 \le \| \Sigmamat_r \|_F \cdot \sqrt{r}\varepsilon \), and \( \| \deltavec \|_2 \le \sqrt{r}\varepsilon \). Therefore,
\begin{equation*}
    \| \ot\Sigmamat_r \|_F \le \| \Sigmamat_r \|_F + \sqrt{r}\varepsilon \le \sigma_1\sqrt{r} + \sqrt{r}\varepsilon = \sqrt{r}(\sigma_1 + \varepsilon).
\end{equation*}
Substituting into~\eqref{eqn:proof-Rmat-reform-start} yields
\begin{equation}\label{eqn:proof-Rmat-reform}
    \| \ot\Rmat \|_F \le \| \Deltamat_K \|_2 \sqrt{r}(\sigma_1 + \varepsilon).
\end{equation}

\paragraph{Completing the proof}
The bounds for \( \| \bbar\Dmat^\dagger \|_2 \) and \( \| \ot\Dmat^\dagger \|_2 \) follow from~\eqref{eqn:proof-Dmat-pinv} and~\eqref{eqn:proof-otDmat-pinv} as in the projection case. Substituting the expressions from~\eqref{eqn:proof-ED-reform-final}, \eqref{eqn:proof-ER-reform}, \eqref{eqn:proof-Dmat-pinv}, \eqref{eqn:proof-otDmat-pinv}, and~\eqref{eqn:proof-Rmat-reform} into Lemma~\ref{lem:operator-perturbation} gives
\begin{align*}
    \| \Omat - \ot\Omat \|_F 
    &\le \frac{\alpha}{\gamma_{\min}} \cdot \Biggl( \frac{\sigma_1\beta_2}{\sqrt{r}(\sigma_1 + \varepsilon)}\tau_w + \frac{\beta_2}{\sigma_1 + \varepsilon}\varepsilon \Biggr) \cdot \| \Deltamat_K \|_2 \sqrt{r}(\sigma_1 + \varepsilon) + \frac{1}{\sqrt{\gamma_{\min}}} \cdot \| \Deltamat_K \|_2 (\sigma_1 \tau_w + \sqrt{r}\varepsilon).
\end{align*}
Collecting terms and simplifying yields
\begin{equation*}
    \| \Omat - \ot\Omat \|_F \le \frac{\sigma_1 \| \Deltamat_K \|_2}{\sqrt{\gamma_{\min}}} \tau_w \Bigl( 1 + \frac{\alpha \beta_2}{\sqrt{\gamma_{\min}}} \Bigr) + \frac{(\alpha \beta_2 + \sqrt{\gamma_{\min}}) \sqrt{r} \| \Deltamat_K \|_2}{\gamma_{\min}} \varepsilon,
\end{equation*}
which completes the proof.
\end{proof}

\section{Additional Turbulent Channel Flow Field Predictions}\label{app:more-channel-plots}

\begin{figure}[H]
    \centering
    \includegraphics[width=0.92\textwidth]{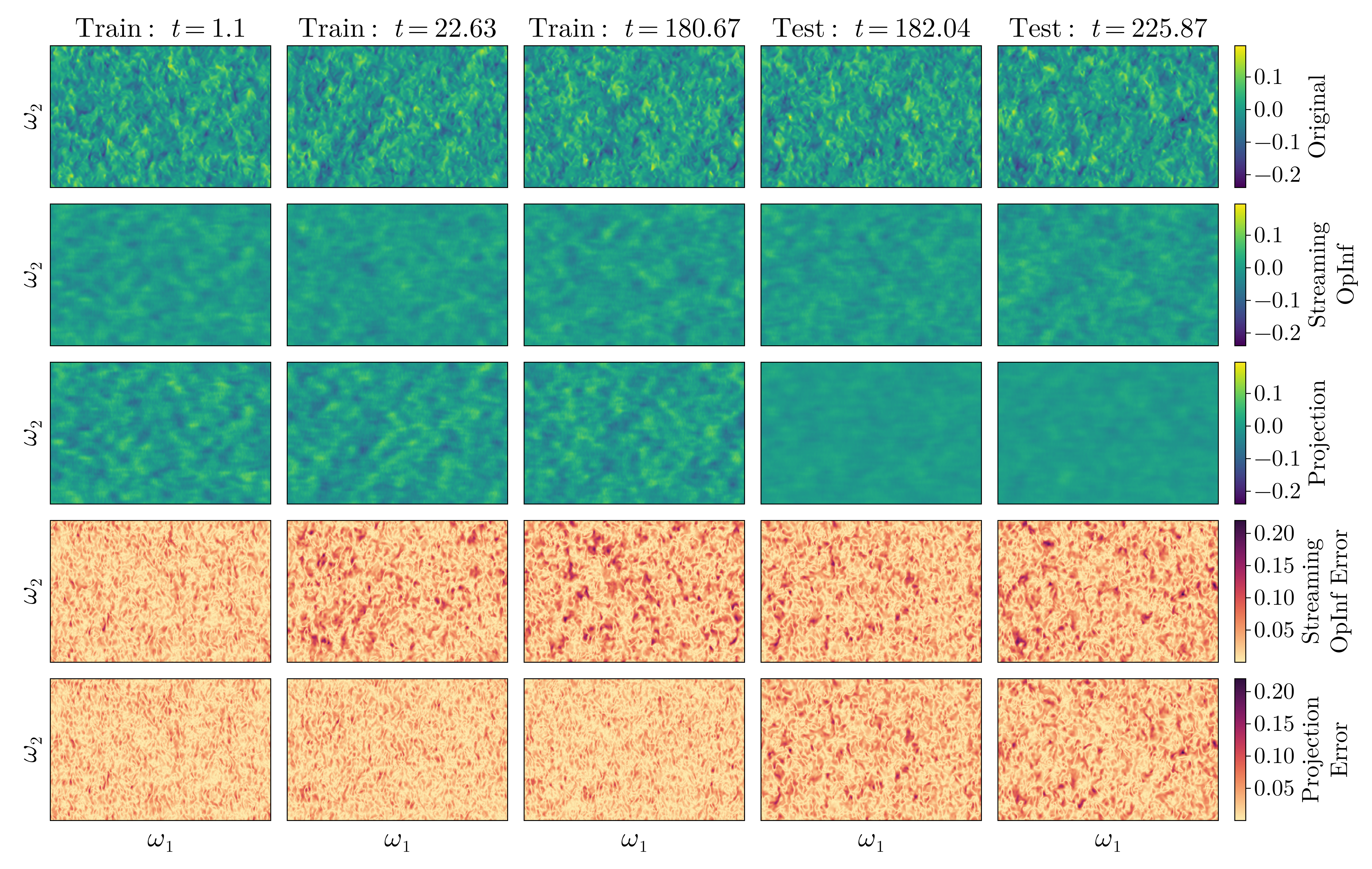}
    \includegraphics[width=0.92\textwidth]{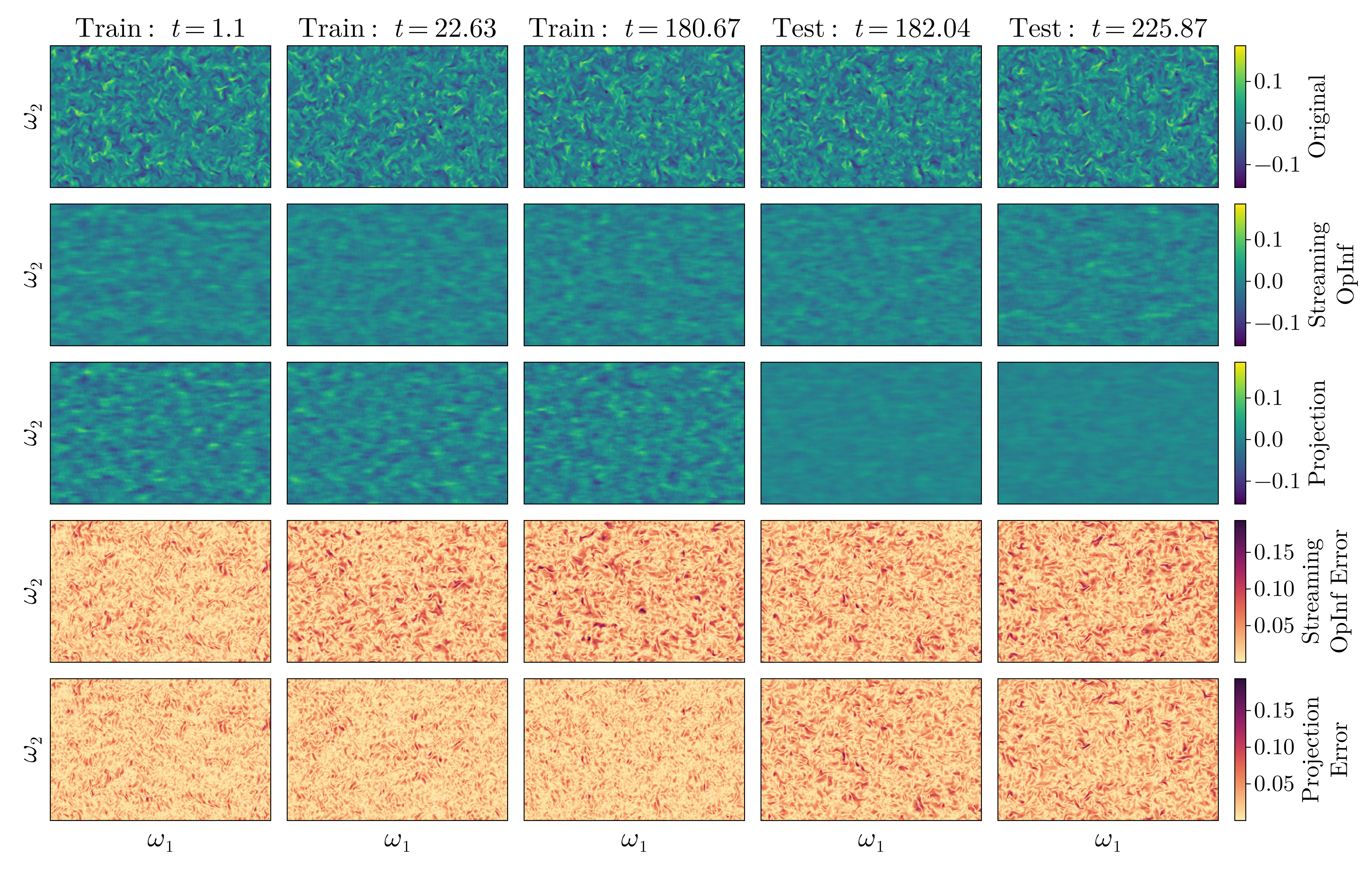}
    \vspace{-0.5em}
    \caption{Comparison of 2D slices of the 3D turbulent channel flow field predictions for the (\textbf{top}) \textbf{spanwise} and (\textbf{bottom}) \textbf{wall-normal} velocity components. First row is the original data, the second row is the prediction from Streaming-OpInf with \( r = 300 \), the third row is the projection result from POD with \( r = 300 \), and the last two rows is the absolute error between the original data and the predictions.}\label{fig:channel:v-train}
\end{figure}


\begin{figure}[H]
    \centering
    \includegraphics[width=0.92\textwidth]{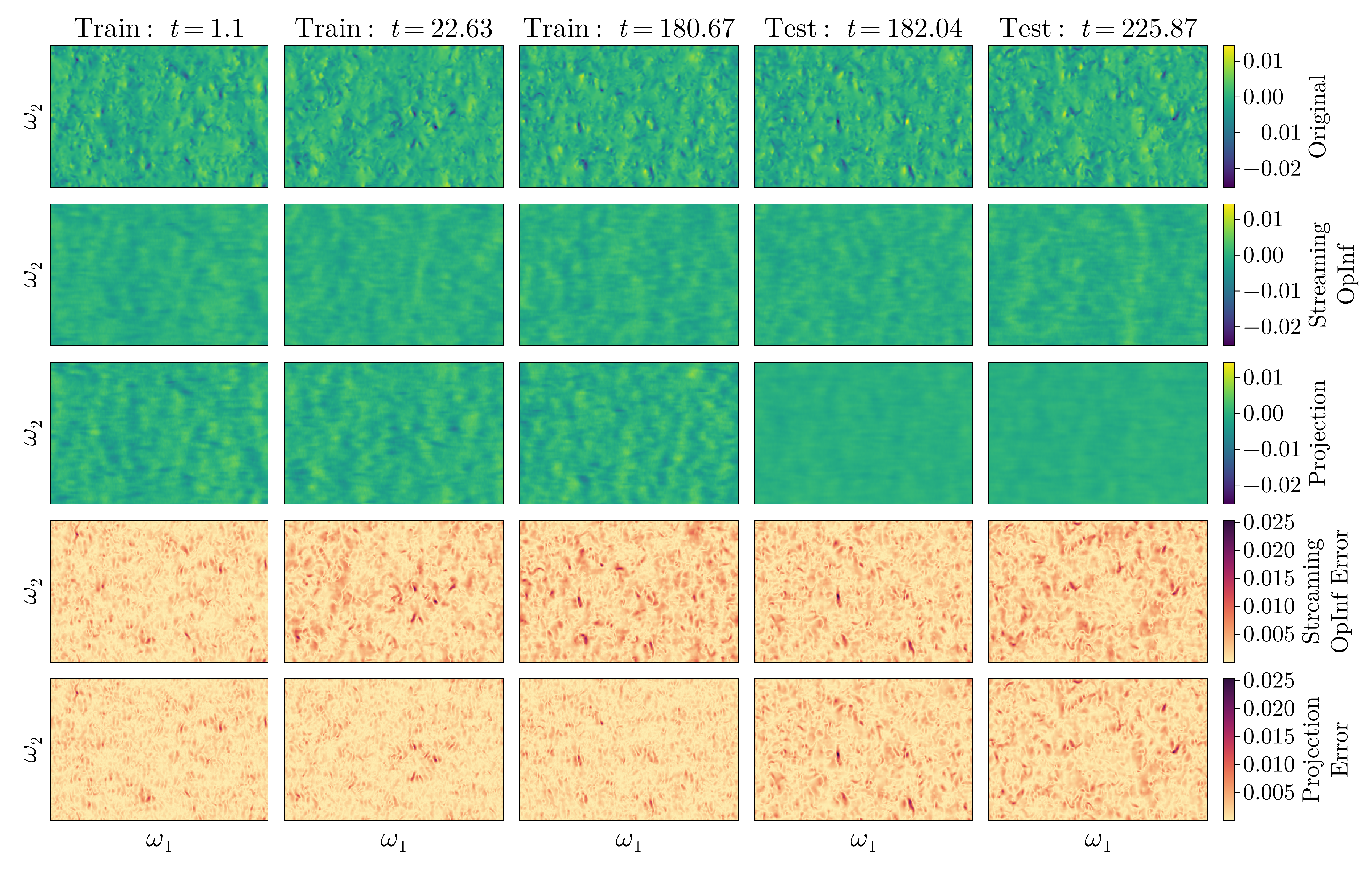}
    \vspace{-0.5em}
    \caption{Comparison of 2D slices of the 3D turbulent channel flow field predictions the \textbf{pressure field}. First row is the original data, the second row is the prediction from Streaming-OpInf with \( r = 300 \), the third row is the projection result from POD with \( r = 300 \), and the last two rows are the absolute error between the original data and the prediction from Streaming-OpInf and POD, respectively.}\label{fig:channel:p-train}
\end{figure}

\bibliographystyle{elsarticle-num} 
{\small\bibliography{refs}}

\end{document}